\documentclass[11pt, reqno]{amsart}
\usepackage{amsmath, amssymb, amsthm}
\usepackage{accents}

\usepackage{enumerate, enumitem}
\usepackage[margin=1in]{geometry}
\usepackage{esint}
\usepackage{xcolor}
\usepackage{hyperref}
\hypersetup{colorlinks = true, linkbordercolor = red}

\pdfoutput=1

\newtheorem{theorem}{Theorem}[section]
\newtheorem{prop}[theorem]{Proposition}
\newtheorem{lemma}[theorem]{Lemma}
\newtheorem{coro}[theorem]{Corollary}
\newtheorem{defn}[theorem]{Definition}
\newtheorem{rem}[theorem]{Remark}

\numberwithin{equation}{section}

\def\R{\mathbb{R}}
\def\C{\mathbb{C}}

\def\e{\varepsilon}

\def\G{\Gamma}
\def\D{\Delta}
\def\O{\Omega}

\def\a{\alpha}
\def\b{\beta}
\def\g{\gamma}
\def\d{\delta}
\def\s{\sigma}
\def\th{\theta}
\def\f{\varphi}
\def\k{\kappa}
\def\l{\lambda}
\def\o{\omega}
\def\r{\rho}
\def\t{\tau}
\def\z{\zeta}
\def\H{\mathcal{H}}

\def\X{\mathcal{X}}
\def\Y{\mathcal{Y}}
\def\Z{\mathcal{Z}}

\def\I{\mathcal{I}}
\def\N{\mathcal{N}}

\def\p{\partial}
\def\n{\nabla}
\def\la{\langle}
\def\ra{\rangle}

\def\id{\mathrm{id} \,}
\def\im{\mathrm{im} \,}
\def\ind{\mathrm{ind}}

\newcommand{\ol}[1]{\overline{#1}}

\newcommand{\emb}{\hookrightarrow}

\newcommand{\chrc}[1]{\mathrm{1}_{#1}}

\def\de{\partial}

\def\Xint#1{\mathchoice
{\XXint\displaystyle\textstyle{#1}}%
{\XXint\textstyle\scriptstyle{#1}}%
{\XXint\scriptstyle\scriptscriptstyle{#1}}%
{\XXint\scriptscriptstyle\scriptscriptstyle{#1}}%
\!\int}
\def\XXint#1#2#3{{\setbox0=\hbox{$#1{#2#3}{\int}$ }
\vcenter{\hbox{$#2#3$ }}\kern-.6\wd0}}

\def\dashint{\Xint-}

\title[Nontrivial periodic solutions to the one-phase FBP]{Nontrivial periodic solutions to the one-phase free boundary problem}

\author{Jos\'e Basulto}
\address{Jos\'e Basulto, Facultad de Matem\'aticas, Pontificia Universidad Cat\'olica de Chile, Avenida Vicu\~na Mackenna 4860, Santiago 7820436, Chile}
\email{jos.basulto@uc.cl}

\author{Nikola Kamburov}
\address{Nikola Kamburov, Facultad de Matem\'aticas, Pontificia Universidad Cat\'olica de Chile, Avenida Vicu\~na Mackenna 4860, Santiago 7820436, Chile}
\email{nikamburov@uc.cl}

\begin{document}

\begin{abstract}
We construct new classical global solutions of the one-phase Bernoulli free boundary
problem in $\R^3$ whose positive phases are periodic perturbations of the product of
the Hauswirth--H\'elein--Pacard double hairpin domain with a line. In contrast with
previous constructions of undulating cylinders, the free boundary of the base solution
is not compact modulo translations. More generally, we show that the same scheme yields
periodic solutions in $\R^{n+1}$, $n \geq 3$, bifurcating from product solutions
$U(x,t)=u(x)$, when the base $u:\R^n\to [0,\infty)$ has nontrivial finite Morse index, is
nondegenerate within a symmetry class, and satisfies certain natural geometric
conditions at infinity. In dimensions $n\in\{3,4\}$, where the stable Bernstein theorem
is available, the recent classification of finite Morse index solutions by Fern\'andez-Real, Florit-Simon and Serra shows that every such solution satisfies these geometric conditions. Thus, the existence of undulating solutions in $\R^{4}$ and $\R^{5}$ bifurcating from their respective products reduces to a nondegeneracy property of the base solution. In order to confirm that the perturbed solutions are positive, we prove that finite Morse index solutions with noncompact free boundary satisfy a strong distance nondegeneracy estimate.
\end{abstract}

\subjclass[2020]{35R35, 35B32, 35N25, 35B10, 35J25, 47J15}
\keywords{one-phase free boundary problem, Bernoulli problem, global solutions, double hairpin, finite Morse index, overdetermined problems, Crandall--Rabinowitz bifurcation}


\bibliographystyle{alpha}
\maketitle

\section{Introduction}
The classical one-phase (also known as ``Bernoulli'') free boundary problem (FBP) seeks continuous functions $u:D\to \R$ on a domain $D\subseteq \R^n$ that solve 
\begin{equation}\label{FBP}
\left\{
\begin{aligned}
 u\geq 0 & \quad \mbox{in } D, \\
   \Delta  u=0 & \quad \mbox{in } \omega:=\{x \in \R^n: u(x)>0\}, \\
   \de_\nu u = -1 & \quad \mbox{on }  \text{FB}(u):=\de \o\cap D,
  \end{aligned}\right.
  \end{equation}
where the free boundary $\text{FB}(u)$ separates the positive phase $\o$ of a solution $u$ from its zero phase, all of which depend on $u$. Here $\nu$ denotes the \emph{outer} unit normal and we will be interested in the \emph{classical} solutions of \eqref{FBP}, that is, those for which $\o$ is a smooth domain. 

The problem \eqref{FBP} is one of the most widely studied FBPs, appearing in physical models ranging from fluid dynamics to electrostatics, and has deep ties to optimization, semilinear PDEs and minimal surfaces.  Formally, \eqref{FBP} is the Euler--Lagrange equation of the Alt--Caffarelli functional
\begin{equation}\label{eq:AC}
J(v,D):=\int_D |\n v|^2 + \chrc{\{v>0\}}\, dx, \quad v: D\to [0,\infty).
\end{equation}
and, since the foundational work of Alt and Caffarelli \cite{AC}, the existence and regularity theory for \emph{minimizers} of \eqref{eq:AC} has been developed extensively (see, e.g., \cite{CafSalsa, VelichkovOP}). 

Recent years, however, have seen increased interest in the FBP beyond the energy-minimizing setting (\cite{JK16, JK19, engelstein2023graphical, KW24, BK24, Kriventsov-Weiss25, Figalli-et-al26-Bern3D, ERZ26, FR-FS-Serra26, FR-Serra26-Bern4D}). Taking limits of rescaled solutions leads to \emph{global solutions} $u:\R^n \to \R_{\geq 0}$ of \eqref{FBP}, so that understanding the latter is a central step of the analysis. 

In $n=2$ dimensions, the only known global classical solutions to \eqref{FBP} of connected positive phase are, up to a similarity transformation:
\begin{itemize}
\item the one-plane solution $P(x)=x_1^+$;
\item the disk-complement solution $L_2(x)=(\log |x|)^+$;
\item the \emph{double hairpin} solution $HP(x)$ of Hauswirth--H\'elein--Pacard \cite{HHP}, whose positive phase is trapped between the two catenary curves
\[
\{x=(x_1,x_2)\in \R^2: x_2 = \pm(\pi/2 + \cosh x_1)\}
\]
and which satisfies $HP(x_1,x_2)\approx |x_1|$ as $|x|\to \infty$;
\item the Baker--Saffman--Sheffield \cite{BSS76} one-parameter family of solutions $\{W_s(x_1, x_2)\}_{s\in (0,1)}$, which are $1$-periodic in the $x_2$-variable, $1$-connected in each period cell, and satisfy asymptotically $W_s(x_1,x_2)\approx s |x_1|$ as $x_1\to \infty$; see also \cite{Traizet,JK19} for the description of this family.
\end{itemize}
In fact, the first three solutions were shown by Traizet \cite{Traizet} to be the only ones, under a connected positive phase and a \emph{finite connectivity} hypothesis on the free boundary. Whether this list of solutions is exhaustive of all possible classical solutions in $\R^2$ remains an open problem.

In \cite{BK24} we showed that the solutions $P$, $L_2$ and $HP$ are also the only ones of connected positive phase in the plane that have \emph{finite Morse index}. The Morse index $\ind(u)$ of a classical solution $u:D \to \R_{\geq 0}$ to \eqref{FBP} is defined to be the index of the second variation form associated to the Alt-Caffarelli functional, (\cite{CJK}):
\[
Q(\f, \f; D) := \int_{\o} |\n \f|^2 \, dx - \int_{\text{FB}(u)} H \f^2 \, d\s, \quad \text{for } \f \in C^1_c(\ol{\o}\cap D),
\]
where $H = \text{div}(\n u/|\n u|)=-u_{\nu\nu}$ denotes the mean curvature of $\text{FB}(u)$ with respect to the outer unit normal $\nu$. That is, $\ind(u)$ is the largest dimension of a vector subspace $V\subset C^1_c(\ol{\o}\cap D)$, on which $Q$ is negative definite, and if $Q$ is non-negative definite on all of $C^1_c(\ol{\o}\cap D)$, we say that $u$ is a \emph{stable} solution  in $D$. 

In higher dimensions $n \geq 3$, one expects a wider variety of solutions, including the ball complement solution $L_n(x):=\frac{1}{n-2}(1-|x|^{2-n})^+$, the axially symmetric solution constructed by Liu, Wang and Wei \cite{LWW21}, as well as products $U(x,t):=u(x)$, $(x,t)\in \R^{n-1}\times \R$, generated by lower-dimensional solutions $u:\R^{n-1}\to \R_{\geq 0}$. In dimension $n\geq 7$ there are also non-planar global energy-minimizing solutions (\cite{DSDJShah, EdeSpoVel24}) that desingularize energy-minimizing singular homogeneous solutions.

The list is further enriched by solutions on $\mathbb{R}^{n+1}$, whose free boundaries are \emph{undulating cylinders}, i.e.\ nontrivial periodic perturbations, along the axial $t$-direction, of the product solution $U(x,t)=u(x)$ with disk/ball complement base $u(x)=L_n(x)$, $n\ge 2$. The constructions were carried out by Fall--Minlend--Weth \cite{FMW23} for $n=2$  and by Minlend--Weth--Wu \cite{MWethWu26} for $n\geq 3$, and utilize the bifurcation method: the nontrivial branches of periodic solutions bifurcate away from the trivial branch of the product solution, at certain values of the period, treated as a bifurcation parameter. 

Over the last 15 years, the bifurcation method has emerged as a powerful tool for constructing symmetry breaking solutions to semilinear overdetermined  problems:
\[
-\D u = f(u), \quad u>0 ~ \text{ in } \O\subset \R^d, \qquad u=\text{const}, ~~ \de_\nu u = \text{const} ~\text{ on } \de\O.
\]
The approach has been successfully wielded to produce counterexamples to the Berestycki--Caffarelli--Nirenberg conjecture (\cite{Sic10, SS12, RRS20, RSW22}), has yielded nontrivial solutions to Serrin's problem (\cite{FMW17, KS21}), as well as symmetry breaking solutions to Schiffer-type overdetermined eigenvalue problems on $\mathbb{S}^2$, $\R^2$, and cylinders (\cite{FMW24, EFRS25}). 

The goal of our paper is to investigate if periodic solutions like the ones constructed by \cite{FMW23, MWethWu26} in the context of the one-phase FBP can more generally arise from other base configurations, in particular ones that have \emph{noncompact} free boundary. Our first result answers this in the affirmative for the base given by the double hairpin solution $HP$ in $n=2$. 

\begin{theorem}\label{thm:mainHP} Let  $U(x,t):=HP(x)$ for $(x,t)\in \R^2\times \R$. There exists a family of classical solutions of \eqref{FBP} $$\{U^s:\R^3\to \R_{\geq 0}: U^0=U\}_{s\in (-\e,\e)}$$ 
that are  even in $x_1, x_2$ and $t$, nonconstant periodic in the axial variable $t$ for $s\neq 0$, and such that the free boundary $\text{FB}(U^s)$ is a normal graph over $\text{FB}(U)$ of a function $\check{\f}_s$, converging in $C^{2,\a}$ to $0$ as $s\to 0$. Moreover, $\check{\f}_s= c(s)+\phi_s$, where $c(s)\in \R$ and $\phi_s(x,t)\to 0$ uniformly in $t$, at the rate $|x|^{-\b}$ for a fixed $\b\in (0,1)$.
\end{theorem} 
These solutions are constructed by the means of the Crandall--Rabinowitz Theorem (\cite{crandall1971bifurcation}) and bifurcate from the double hairpin product at a certain value of the period. This is made precise in the body of Theorem \ref{thm:bif}, of which Theorem \ref{thm:mainHP} is an immediate corollary. 

Clearly, such a construction cannot start from a one-plane solution base $x_1^+$, since the corresponding product solution is a one-plane solution itself; nearby solutions would thus have to possess asymptotically flat free boundary, which upgrades to fully flat by the regularity theory of the one-phase FBP (\cite{AC, DS2011}). It is generally recognized that the existence of undulating solutions is related to an underlying instability (\cite{Dancer01}) and our second main result provides broad sufficient conditions for the bifurcation scheme to produce new global periodic solutions in $\R^{n+1}$, starting from certain unstable global solutions in $\R^n$, for $n \geq 3$.

\begin{theorem}\label{thm:maingen}
Let $n\geq 3$ and let $u:\R^n \to \R_{\geq 0}$ be a global classical solution to \eqref{FBP} that satisfies the following hypotheses:
\begin{enumerate}[label=(H\arabic*), ref=H\arabic*]
\item \label{H1} $u$ has the Hessian decay
\[
|D^2 u(x)| = o(|x|^{-1}), \quad \text{as } x\to \infty~ \text{ in } ~\o=\{u>0\}.
\]

\item\label{H2} There exists $\d_0>0$, such that the rays $\{\{p+z\nu(p): 0< z < \d_0\}\}_{p\in \de\o}$ into the zero phase of $u$ are disjoint.

\item\label{H3} The free boundary of $u$ satisfies
\[
\lim_{R\to \infty} \sup_{x\in \text{FB}(u), |x|\geq R} (x/|x|)\cdot \nu \leq 0,
\]
where $\nu$ is the outer unit normal to $\text{FB}(u)=\de\o$.

\item\label{H4} The Morse index $\ind(u)$ relative to the Alt-Caffarelli functional is nontrivial and finite: $$1\leq \ind(u)<\infty.$$

\item\label{H5} There exists a group $G_0\subset O(n)$ of isometries of $\R^n$, such that $u$ is $G_0$-invariant and there are no nontrivial $G_0$-invariant Jacobi fields that decay to $0$ at infinity, i.e.\ if $w\in C^2(\o)\cap C^1(\ol{\o})$ is an element in the kernel of the associated linear operator to the second variation form $Q$:
\begin{equation*}
\left\{
\begin{aligned}
- \Delta w &= 0 && \text{ in } \o, \\
\de_\nu w - Hw &= 0 && \text{ on } \de\o,
\end{aligned}
\right.
\end{equation*}
such that $\lim_{x\to\infty} w(x) = 0$ and $w(g\cdot x)=w(x)$ for all $g\in G_0$, then $w\equiv 0$.
\end{enumerate}
Denote its corresponding product solution $$U(x,t):=u(x)  \quad \text{for  }(x,t)\in \R^n\times \R.$$ 
Then there exists a family of classical solutions $\{U^s:\R^{n+1}\to \R_{\geq 0}: U^0=U\}_{s\in (-\e,\e)}$  to \eqref{FBP}, which are $G_0$-invariant in $x$, even and nonconstant periodic in $t$ for $s\neq 0$, and such that the free boundary $\text{FB}(U^s)$ is normal graph over $\text{FB}(U)$ of a function $\check{\f}_s$, which converges in $C^{2,\a}$ to $0$, as $s\to 0$. Moreover, if $\text{FB}(u)$ is unbounded, $\check{\f}_s(x,t) \to 0$ uniformly in $t$, at the rate of $|x|^{-\b}$ for a fixed $\b\in \left(\frac{n-3}{2}, \min(\frac{n}{2}, n-2)\right)$.
\end{theorem}

\begin{rem} Using separation of variables, one can show that  the ball complement solution $u=L_n(x)$, $n\geq 3$, has Morse index $1$ and all its bounded Jacobi fields are a linear combination of 
\[
w_0(x) = \frac{1-(n-1)|x|^{2-n}}{n-2} \quad \text{and}\quad  w_i(x)=\frac{e_i\cdot x}{|x|^n}, \quad i=1,\ldots,n, \quad \text{for } x\in\o=\R^n\setminus B_1.
\]
Note that $w_0(x) = u(x)-\n u \cdot x$ is precisely the field generated by the dilation invariance of \eqref{FBP} while $w_i=\de_{x_i} u$ is the field generated by the translation invariance of the problem along $e_i$. Thus, a Jacobi field that tends to $0$ at infinity must be a linear combination of the $\{w_i\}_{i=1}^n$ only. Taking $G_0=O(n)$, we confirm that the only $G_0$-invariant Jacobi field, tending to $0$ at infinity, is the trivial one. Therefore, for $G_0=O(n)$, $u=L_n$ meets all the hypotheses (\ref{H1}-\ref{H5}) (the first three trivially satisfied), and Theorem \ref{thm:maingen} gives another construction of undulating cylinder solutions, next to the one in \cite{MWethWu26}.
\end{rem}

The group of hypotheses (\ref{H1}--\ref{H3}) are statements on the asymptotic geometry of the base solution. The natural scaling-invariant bound on the Hessian is $|D^2 u(x)| = O(|x|^{-1})$, as $|x|\to\infty$, and yields uniform local estimates for $\text{FB}(u)$ at every scale; hypothesis (\ref{H1}) requires a slightly faster decay at infinity to guarantee the Fredholm property of the linearized operator. Note that the double hairpin solution also falls under (\ref{H1}) as we show that $|D^2 HP(x)| = O(|x|^{-2})$ in Lemma \ref{lem:Hessdecay}. The condition in (\ref{H2}) is technical and allows us to perturb $\text{FB}(u)$ up to distance $\d_0$ in the normal direction without getting self-intersections\footnote{This is not directly implied by (\ref{H1}) as the zero-phase could  be arbitrarily thin at infinity a priori.}. Hypothesis (\ref{H3}) is a one-sided angle condition which holds, in particular, when the free boundary is asymptotically conical; it is used to construct an appropriate barrier for the linearized operator. 

The remaining two conditions (\ref{H4}) and (\ref{H5}) are spectral in nature. Hypothesis (\ref{H5}) allows us to sidestep perturbations of the product solution $U$, generated by the continuous symmetries of the base $u$. Requiring that $G_0\subset O(n)$ is no loss of generality: any group of isometries leaving $u$ invariant fixes the barycenter $\int_\o x w_{\l_1}^2$ of the $L^2$-normalized ground state $w_{\l_1}$ (unique and rapidly decaying at infinity) of the Jacobi operator, which we may thus translate to the origin.

The finite nontrivial Morse index condition (\ref{H4}) is the key hypothesis that underlies the construction. We note that the finiteness of the index implies that $u$ is stable  outside a large ball (see \cite{BK24}), which in turn yields the natural Hessian decay for dimensions $n\geq 2$, where the \emph{stable Bernstein-type theorem} for the one-phase FBP holds (see \cite{KW24, Figalli-et-al26-Bern3D}): 
\begin{equation}\label{eq:bern}
\text{If } u:\R^n \to \R_{\geq 0} \text{ is a stable classical solution to } \eqref{FBP} \text{, then FB}(u) \text{ is flat}.   \tag{BT}
\end{equation}
The groundbreaking works \cite{Figalli-et-al26-Bern3D} and \cite{FR-Serra26-Bern4D} have extended the validity of \eqref{eq:bern} up to dimension $n=4$. Using the local Hessian bounds this entails, a novel improvement of flatness technique in annuli and the moving plane method, Fern\'andez-Real, Florit-Simon and Serra \cite{FR-FS-Serra26} have recently managed to completely characterize finite Morse index solutions for dimensions $n$ where \eqref{eq:bern} holds. Their result says that, unless $u$ is the one-plane or the ball complement solution, $u$ must be axially symmetric with noncompact zero phase given by
\begin{equation}\label{eq:hphigher}
\{u=0\} = \{(x', x_n)\in \R^n: |x'|\geq 1, |x_n|\leq g(|x'|)\},
\end{equation}
where $g=g(r):\R_{\geq 1} \to \R$ is smooth for $r>1$, continuous up to $r=1$ with $g(1)=0,$ and has the asymptotic expansion 
\begin{equation}\label{eq:exp}
g(r) =  b-cr^{2-n} + O(r^{1-n}), \quad \text{for some }b,c>0 \quad  \text{as } r\to\infty,
\end{equation}
with the corresponding expansions for $g'$ and $g''$. From it, one can check that (\ref{H1}--\ref{H3}) are automatically satisfied, as we do below.  As a result, Theorem \ref{thm:maingen} reduces the existence of undulating solutions in $\R^{n+1}$, $n\in\{3,4\}$, to a nondegeneracy property of the base solution:
 
\begin{coro}\label{coro:FRFSS}
Let $n\in\{3,4\}$ and let $u:\R^n\to\R_{\geq 0}$ be a classical solution of \eqref{FBP} of connected positive phase and finite Morse index, which is neither a one-plane nor a ball complement solution. After a rigid motion, $u(x',x_n)$ is $G_0$-invariant, where
\[
G_0:=O(n-1)\times\{\mathrm{id},\,x_n\mapsto -x_n\},
\]
acts on $x=(x',x_n)\in\R^{n-1}\times\R$. If $u$ has no nontrivial $G_0$-invariant Jacobi field that tends to $0$ at infinity, then the conclusion of Theorem \ref{thm:maingen} holds for $u$ with this group $G_0$. In particular, there exist undulating solutions of \eqref{FBP} in $\R^{n+1}$, which are $T_s$-periodic perturbations of the product $U(x,t)=u(x)$, $(x,t)\in\R^n\times\R$.
\end{coro}
\begin{proof}
By \cite{Figalli-et-al26-Bern3D, FR-Serra26-Bern4D}, $\ind(u)\geq 1$ and \cite{FR-FS-Serra26} implies that $\text{FB}(u)$ has the asymptotic profile \eqref{eq:exp}. We only need to check (\ref{H1}--\ref{H3}) and apply Theorem \ref{thm:maingen}. (\ref{H2}) is immediate from the zero phase description \eqref{eq:hphigher}-\eqref{eq:exp}. To get (\ref{H3}), one computes using the axial symmetry that at the upper leaf $\{x_n = g(|x'|)\}$ of the free boundary
\[
x\cdot \nu_{\text{FB}(u)} = (r, g)\cdot (g', -1)/\sqrt{1+g'^2} = - b+(n-1)c r^{2-n} + O(r^{1-n}) \to -b<0 \quad \text{as } r\to \infty,
\]
By symmetry, this also holds for the lower leaf. Finally, in order to show (\ref{H1}) one estimates the mean curvature at $(x', g(|x'|))$ with $r=|x'|$,
\[
|H| = |\text{div}(\nu)| = \left|r^{1-n}\left(r^{n-1}g'/\sqrt{1+g'^2}\right)'\right| = O(r^{-n}).
\]
Since the free boundary has $(n-1)$-dimensional area growth $\s(B_r\cap \de\o) = O(r^{n-1})$ (see Proposition \ref{prop:areagrowth}), we conclude that $u$ has \emph{finite total mean curvature} 
\begin{equation}\label{eq:ftmc}
\int_{\de\o} |H|^{n-1} \, d\s <\infty. 
\end{equation}
By \cite[Lemma 5.7]{FR-FS-Serra26}, \eqref{eq:ftmc} implies $|D^2 u(x)| = o(|x|^{-1})$ and we confirm (\ref{H1}), as well.
\end{proof}

\begin{rem}
The Liu--Wang--Wei axially symmetric solution (\cite{LWW21}), which exists in every dimension $n\geq 3$, satisfies the asymptotic description \eqref{eq:hphigher}, as pointed out by \cite[Remark 1.5]{FR-FS-Serra26}. It has finite total mean curvature, which implies finite Morse index (\cite{BK26_ftm}). If it is shown to possess no nontrivial $O(n-1)\times \mathbb{Z}_2$-invariant Jacobi fields that decay to $0$ at infinity, then Theorem \ref{thm:maingen} will construct a branch of undulating solutions, bifurcating from the corresponding product solution in $\R^{n+1}$.
\end{rem}

\subsection{Overview of the construction}
We follow the general bifurcation scheme applied to the product solution $\bar{U}(x,t)=u(x)$, but several new difficulties arise from the noncompactness of the free boundary cross-section and from the linear growth of $u$ at infinity. In particular, it is not a priori clear what the correct conditions at infinity are for the perturbed solutions, nor in which function spaces the linearized problem enjoys the Fredholm properties required by the Crandall--Rabinowitz theorem. 

We are looking for perturbations of $\O$ that are $T$-periodic in the axial variable $t$, for some parameter $T>0$. In order to work on fixed function spaces, we rescale the problem $t\to t/T$, so that we are looking for $1$-periodic normal perturbations of $\O$ in the rescaled metric $g_T := dx^2 + T^2 dt^2$. We then perturb $\O:=\o\times \R$ along a normal vector field $h_\f$ to $\G:=\de\O$, where $h_\f|_\G=\f \nu$ is parametrized via Fermi coordinates by a function $\f$ on $\G$. Let $\Phi_{h_\f}(x) := x + h_\f(x)$ denote the resulting diffeomorphism. Pulling \eqref{FBP} back to the fixed domain $\O$ and endowing it with the pull-back metric $g_{T,\f}:=\Phi^*_{h_\f} g_T$, we seek the solution in the form $$\tilde{U}=U\circ \Phi_{h_\f} + v,$$ 
where $U(x,t):=\bar{u}(x)$ with $\bar{u}$ being a harmonic extension of $u|_\o$ across a unit-width band around $\de\o$.  Since $\nabla U\neq 0$ on $\G$, the Dirichlet condition $\tilde{U}=0$ on $\G$ determines the normal perturbation function $\f$ implicitly in terms of $v|_\G$ (with $\f[v]\approx v|_\G$), so that for each fixed $T$, we are reduced to solving an equation
\[
M(T, v):= \left(-\D_{g_{T, \f[v]}} \tilde{U}, \left(1-|d\tilde{U}|^2_{g_{T,\f[v]}}\right)/2\right) = (0,0)
\]
for the single unknown $v$, whose linearization at $v=0$ is precisely the rescaled Jacobi operator $$\mathcal{L}_T: w\to \left(-\D_T w, w_\nu - Hw\right), \qquad \D_T := \D_x +T^{-2}\de_t^2.$$

The equation for $M(T,\cdot)$ is formulated in weighted H\"older spaces $C^{k,\a}_\b$ of functions on $\O$ and $\G$ that decay like $\la x\ra^{-\b}$ in the cross-sectional variables, with the natural scaling of derivatives, and in their subspaces, invariant under a group $G=G_0\times G_1$, where $G_1$ acts on the axial variable and contains the translation $t\to t+1$. Under hypotheses (\ref{H1}) and (\ref{H2}), $M(T,v)$ is well defined for $v$ in a neighbourhood of $0$ in $C^{2,\a}_{\b}(\ol{\O})$ into $C^{0,\a}_{\b+2}(\ol{\O})\times C^{1,\a}_{\b+1}(\G)$, and is jointly smooth in $T,v$ (in fact, affine in $T^{-2}$). 

We split its linearization as $L-K$, with $Lw:=(-\D_T w, \de_\nu w)$ and a compact part $Kw:=(0,Hw)$, and the aim is to find appropriate values of the decay rate $\b>0$ such that the Neumann problem operator $L$ is Fredholm of index zero, which it then passes on to the linearization $\mathcal{L}_T$. This issue occupies the bulk of the paper and just like for the construction of the undulating cylinders (\cite{FMW23, MWethWu26}), there is a difference between dimensions $n=2$ and $n\geq 3$.  We find that when $n\geq 3$, $L$ is an isomorphism for decay rates $\b \in \left(\frac{n-3}{2}, \min(\frac{n}{2}, n-2)\right)$, under (\ref{H1}) and (\ref{H3}), (see Theorem \ref{thm:iso-higher}).

For the case $n = 2$, relevant to Theorem \ref{thm:mainHP}, the operator $L$ is injective  but has a nontrivial cokernel for every $\beta > 0$. Indeed, it is not hard to see by testing the weak formulation of $L w = (f,g)$ against logarithmic cut-offs approximating the constant function $1$ on a period cell $\Omega_{\text{per}} = \omega \times (-1/2, 1/2)$ that every $(f,g)$ in the range of $L$ must satisfy the flux condition
\[
l(f,g) := \int_{\Omega_{\mathrm{per}}} f \, dX + \int_{\Gamma_{\mathrm{per}}} g \, dS = 0.
\]
To restore Fredholm index zero, we enlarge the domain to $C^{2,\alpha}_\beta(\ol{\Omega}) \oplus \R$, allowing the corrector function $v$ (and thus the normal perturbation $\f$) to tend to a constant at infinity. This is possible thanks to the \emph{improved} decay $$|D^2 HP(x)| \le C \langle x \rangle^{-2}, \quad \text{for all } x\in \o\subset \R^2,$$ enjoyed by the double hairpin, which ensures that the nonlinear operator $M$ remains well defined and smooth on the enlarged space. It is what motivates the careful and systematic asymptotic analysis of the nonlinear operators done in Section \ref{sec:setup}. We then show that $L : C^{2,\alpha}_{\beta,G}(\ol{\Omega}) \oplus \R \to C^{0,\alpha}_{\beta+2,G}(\ol{\Omega}) \times C^{1,\alpha}_{\beta+1,G}(\Gamma)$, $\beta \in (0,1)$, is Fredholm of index zero, with kernel the constants and range $\ker l$ (see Theorem \ref{thm:Fred}).
 
Separating variables, the kernel of $\mathcal{L}_T$ is nontrivial when $\l + (2\pi k / T)^2 = 0$ for some eigenvalue $\l$ of the cross-sectional Robin problem
\begin{equation}\label{eq:Robin0}
-\Delta_x w = \l w \ \text{ in } \o, \qquad w_\nu - H w = 0 \ \text{ on } \de\o.
\end{equation}
By \cite{BK24}, a finite Morse index $\ind(u)=m \geq 1$ produces exactly $m$ negative eigenvalues $\l_1 < \l_2 \le \dots \le \l_m$ of \eqref{eq:Robin0}, with $\l_1$ simple, and $L^2$-normalized eigenfunctions decaying at infinity. Combining this with the natural curvature bound, local boundary estimates and rescaling, we show that these eigenfunctions in fact decay faster than any power of $\la x\ra$, so that they belong to the weighted spaces above. In particular, a bifurcation always occurs at the lowest eigenvalue $\l_1$, for $T_0 = 2\pi / \sqrt{-\l_1}$. Contributions of the zero eigenspace of \eqref{eq:Robin0} to the kernel are excluded by the nondegeneracy hypothesis (\ref{H5}). For the double hairpin, we verify directly that there are no nontrivial Jacobi fields with the symmetries $x_i \mapsto -x_i$, $i = 1,2$. Imposing in addition evenness in $t$, we guarantee that the kernel at $T_0$ is $1$-dimensional, spanned by
\[
W(x,t) := w_{\l_1}(x) \cos(2\pi t), \quad (x,t)\in \o\times \R,
\]
where $w_{\l_1}$ is the ground state eigenfunction of \eqref{eq:Robin0}. At this stage, we invoke the Crandall--Rabinowitz bifurcation theorem and obtain the branch of nontrivial periodic solutions that we have sought.

One final caveat is that the resulting perturbed solutions $\tilde{U} = U\circ \Phi_{h_\f[v]} + v$ solve the boundary value problem, yet it is not immediately clear if they are positive in $\O$. For small enough perturbations, $\tilde{U}>0$ up to a unit distance from $\G$ in $\O$, whereas the region beyond is taken care of by the fact that $U$ is uniformly positive a fixed distance away from $\G$. When the base $u$ has compact free boundary, this is a straightforward argument. In the case of noncompact \text{FB}(u), it is the result of the following \emph{distance nondegeneracy estimate} for such finite Morse index solutions, which is of independent interest.

\begin{theorem}\label{thm:nondeg} Let $D\subset \R^n$ be a domain and let $u$ be a classical solution of \eqref{FBP} in $D$. Denote by $d(x):=d(x, \text{FB}(u))$.
\begin{enumerate}[label=(\alph*)]
\item If $u$ is stable in $D$, then for every $x\in D\cap \{u>0\}$, $y\in \text{FB}(u)$ such that $|x-y|=d(x)$ and $B_{4d(x)}(y)\subset D$, we have $u(x)\geq c d(x)$, for some constant $c=c(n)>0$. In particular, if $u$ is a global stable solution of \eqref{FBP} with $\text{FB}(u)\neq \emptyset$, then 
\[
u(x)\geq c d(x)  \quad \text{for all } x\in \{u>0\}.
\]

\item If $D=\R^n$, $u$ has finite Morse index and $\text{FB}(u)\neq \emptyset$ is noncompact, then 
\begin{equation*}
u(x)\geq c_u d(x) \quad \text{for all } x\in \{u>0\}, 
\end{equation*}
for some positive constant $c_u$, depending on $u$.
\end{enumerate}
\end{theorem}

The statement of Theorem \ref{thm:nondeg}(a) strengthens the nondegeneracy bound for stable solutions, proved in \cite{KW24} and is obtained from it by employing the Kriventsov--Weiss compactness theorem (\cite{Kriventsov-Weiss25}) for solutions of \eqref{FBP}. Part (b) is established combining (a) with a different compactness argument, in which the unboundedness of $\text{FB}(u)$ rules out trivial blowdown limits for $u$. The noncompactness is a necessary condition, as manifested by the disk complement solution, which has $\ind(L_2)=1$, but grows only logarithmically in the distance to the free boundary. 


The paper is organized as follows. In Section \ref{sec:prelim} we collect preliminaries on classical solutions of \eqref{FBP} and their regularity theory in order to obtain asymptotic free boundary bounds under (\ref{H1}); we also provide specific details on the double hairpin solution and show its improved Hessian decay. Section \ref{sec:morse} focuses on finite Morse index solutions of \eqref{FBP}: it discusses the spectral properties of \eqref{eq:Robin0} and the decay of its eigenfunctions associated to negative eigenvalues. There we also establish the nondegeneracy growth estimates of Theorem \ref{thm:nondeg}. In Section \ref{sec:setup} we introduce the weighted H\"older spaces and formulate the free boundary problem perturbed around the given product solution as an operator equation, reducing to $M(T, v) = 0$. In Sections \ref{sec:linHP} and \ref{sec:lingen} we study the linearized operator in $v$ and establish the Fredholm properties, respectively, for the case $n=2$ of the product double hairpin and for the general case $n\geq 3$ of Theorem \ref{thm:maingen}. For the purpose, we develop the solvability theory of the inhomogeneous Neumann problem $Lw=(f,g)$ in the corresponding function spaces.  Section \ref{sec:bif} contains the bifurcation argument and the proofs of our main results, Theorems \ref{thm:mainHP} and \ref{thm:maingen}. In Appendix \ref{app:apriori} we provide a priori estimates for the Dirichlet and Neumann problems in our weighted H\"older spaces. In Appendix \ref{app:Sobolev} we give some $L^2$-Sobolev trace inequalities, specific for the one-phase FBP context; both sets of results are necessary to develop the solvability theory for the Neumann problem in Sections \ref{sec:linHP} and \ref{sec:lingen}. In Appendix \ref{app:technical}, we collect a few technical lemmas, concerning deformations with horizontal vector fields, necessary for Section \ref{sec:setup}. 

\section{Preliminaries on the one-phase free boundary problem}\label{sec:prelim}
In this section we collect several results about solutions to the one-phase FBP \eqref{FBP}, which will be used heavily later on in the construction of the undulating solutions. We first discuss the asymptotic free boundary bounds entailed by a Hessian bound, and then provide background on the double hairpin solution.

\subsection{Asymptotic free boundary bounds}\label{subsec:asympfb}
As the base of our construction, we will be using global solutions $u:\R^n\to [0,\infty)$ to \eqref{FBP} whose Hessian obeys the bound
\begin{equation}\label{eq:Hess}
|D^2 u(x)| \leq K(x) \la x \ra^{-1}, \qquad \la x\ra:=\sqrt{1+|x|^2}, \qquad \text{for all } x\in \o:=\{x\in \R^n: u(x)>0\},
\end{equation}
where $K:\o\to[0,\infty)$ is a bounded function, and we set $\bar{K}:=\|K\|_{L^\infty(\o)}$. When ${\bar K}<\infty$, $u$ has the natural, scaling-invariant Hessian decay. The free boundary bounds below are governed by the \emph{local} size of $K$, which we record through
\begin{equation}\label{eq:kx}
\k(x):=\sup\{K(y):\ y\in\o,\ |y-x|\leq \la x\ra/2\}\leq \bar K,\qquad x\in\ol{\o}.
\end{equation}
Two instances of \eqref{eq:Hess} will be relevant:
\begin{itemize}
\item Hypothesis (\ref{H1}) means precisely that \eqref{eq:Hess} holds with $K(x)\to 0$ as $|x|\to\infty$; then also $\k(x)\to 0$ as $|x|\to\infty$.
\item For the double hairpin solution $u=HP$, \eqref{eq:Hess} holds with $K(x)=C_0\la x\ra^{-1}$ (Lemma \ref{lem:Hessdecay} below); then $\k(x)\leq 2C_0\la x\ra^{-1}$, since $\la y\ra\ge\la x\ra/2$ whenever $|y-x|\le\la x\ra/2$.
\end{itemize}
In this subsection, we show how \eqref{eq:Hess} entails higher-order bounds for the free boundary $\text{FB}(u)$, which improve at the rate $\k(x)$.

For a set $S\subset \R^n$, $x\in S$ and $r>0$, we denote by $S_r(x)$ the connected component of $S\cap B_r(x)$ containing $x$ (in the relative topology of $S$). Let $u_R(x):=R^{-1}u(Rx)$ denote the Lipschitz rescaling of $u$ by $R>0$, and denote $\o^R:= \o/R = \{u_R>0\}$ and $x^R:=x/R$.

The bounds are a consequence of the classical $\e$-regularity theory for \eqref{FBP}.

\begin{prop}[$\e$-regularity]\label{prop:epsreg} Let $u$ be a classical solution to \eqref{FBP} in $B_1\subset \R^n$, such that $0\in\text{FB}(u)$ and $e_n = \n u(0)$. Denote by $V$ the connected component of $\{u>0\}\cap B_1$ bordering $0$. There exists $\e_0 = \e_0(n)>0$ such that, if
\begin{equation}\label{prop:epsreg:hessbd}
\e:=\|D^2 u\|_{L^\infty(V)}\leq \e_0,
\end{equation}
then $V\cap B_{1/2}$ is the supergraph of a function $g:\R^{n-1}\to \R$,
\[
V\cap B_{1/2} = \{x=(x', x_n) \in B_{1/2}: x_n> g(x')\},
\]
with $g(0)=|\n g(0)|=0$ and $\|g\|_{C^{k,\a}}\leq \bar{c}_k \e$, for some $\bar c_k=\bar c_k(n)>0$. Moreover, $u|_{V\cap B_{1/2}}$ has an extension as a harmonic function $\bar{u}$ to all of $B_{1/2}$ with
\begin{equation*}
\|\bar{u}\|_{C^{k,\a}(B_{1/2})} \leq \bar c_k.
\end{equation*}
\end{prop}
\begin{proof}
We claim that, for small enough $\e_0>0$, $V$ is the supergraph within $B_1$ of a function $g:\R^{n-1}\to \R$ with
\begin{equation}\label{prop:epsreg:g}
\|g\|_{C^2(\R^{n-1})} \leq C\e,
\end{equation}
for some $C=C(n)>0$. Indeed, denote by $\G$ the connected component of the free boundary containing $0$, and let $A$ denote its second fundamental form. Since $|\n u|=1$ on $\G$, the Hessian bound \eqref{prop:epsreg:hessbd} gives
\[
|A| = |\n u| |A| \leq |D^2 u| \leq \e \quad \text{on } \G,
\]
so that $\G$ is the graph of a function $g$ satisfying \eqref{prop:epsreg:g}. We show that, when $\e_0$ is small enough, $\de V\cap B_1$ contains no free boundary points other than those of $\G$. Note that \eqref{prop:epsreg:g} implies that the gradient $\n u = (-\n_{x'} g, 1)/\sqrt{1+|\n_{x'}g|^2}$ deviates little from $e_n$ along $\G$:
\begin{equation}\label{prop:epsreg:grad}
|\n u - e_n| \leq C\e \quad \text{on } \G.
\end{equation}
Assume, to the contrary, that there is a free boundary point $p=(y', y_n) \in \de V\cap B_1$, not on $\G$, lying vertically above the point $p_0=(y', g(y'))\in \G$, and choose it with the smallest vertical coordinate $y_n> g(y')$. Then the vertical segment from $p_0$ to $p$ lies in $V$, and $\de_{x_n} u(p) \leq 0$, since $u$ vanishes at $p$ and is positive just below it. The fundamental theorem of calculus gives
\[
\de_{x_n}u(p_0) - \de_{x_n}u(p)\leq \|D^2 u\|_{L^\infty(V)} |y_n-g(y')| \leq 2\e.
\]
However, $\de_{x_n}u(p_0) - \de_{x_n}u(p) \geq \de_{x_n}u(p_0)\geq 1/2$ by \eqref{prop:epsreg:grad}, when $\e$ is small enough, a contradiction.

Because $V$ is the supergraph of a function of small $C^1$-norm, every point $p\in V$ is connected to $0$ by a polygonal path in $V$ of length $l\leq c(n)$. Hence
\[
|\n u (p) - e_n| = |\n u (p) - \n u(0)| \leq l \|D^2 u\|_{L^\infty(V)} \leq c \e \quad \text{for all } p\in V,
\]
and integrating $\n u$ along the same path, we get $|u(p) - p_n| \leq c^2 \e$ for all $p\in V$. Thus, $u|_V$ is $c^2\e$-close to the one-dimensional solution $x_n^+$, and the conclusion follows from the well known $\e$-regularity result, as stated in \cite[Lemma 3.7]{Figalli-et-al26-Bern3D}, building upon \cite{KindNiren, AC, DS2011}.
\end{proof}

\begin{prop}\label{prop:asymp}
Let $u:\R^n\to [0,\infty)$ be a global solution of \eqref{FBP}, satisfying \eqref{eq:Hess}, and let $k$ be as in \eqref{eq:kx}. Then there exist constants $\d=\d(\bar{K}, n) \in (0,1/8)$ and $C_k=C_k(\bar{K},n)>0$, $k\in \mathbb{N}$, with the following property. For every $x\in \text{FB}(u)$, with $R:=\la x\ra$, the connected component $[\ol{\o^R}]_{2\d}(x^R)$ is, in a suitable Euclidean coordinate system centred at $x^R$, the supergraph within $B_{2\d}(x^R)$ of a function $g_x:\R^{n-1}\to \R$ with $g_x(0)=|\n g_x(0)|=0$ and
\begin{equation}\label{prop:asymp:eq:g}
\|g_x\|_{C^{k,\a}(\R^{n-1})}\leq C_k\, \k(x).
\end{equation}
Moreover, $u_R|_{[\ol{\o^R}]_{2\d}(x^R)}$ has an extension to all of $B_{2\d}(x^R)$ as a harmonic function $\bar{u}$ with
\begin{equation}\label{prop:asymp:eq:ubds}
\| \bar{u} \|_{C^{k,\a}(B_{2\d}(x^R))} \leq C_k.
\end{equation}
\end{prop}
\begin{proof}
Let $\e_0=\e_0(n)$ be the constant from Proposition \ref{prop:epsreg} and let $\d\in(0,1/8)$, to be fixed. Take $x\in \text{FB}(u)$, set $R:=\la x \ra$, and consider the recentered and rescaled solution
\[
\tilde{u}(y) := \frac{u(x + 4\d R y)}{4\d R}, \qquad y\in \R^n.
\]
For $y\in B_1$, the point $x':=x+4\d Ry$ satisfies $|x'-x|\leq 4\d R\leq R/2$, so that $\la x'\ra\geq R/2$ and $K(x')\leq \k(x)$ whenever $x'\in\o$. Hence \eqref{eq:Hess} gives
\[
|D^2 \tilde{u}(y)| = 4\d R\, |D^2 u(x')| \leq 4\d R\, \k(x)\,(R/2)^{-1} = 8 \d\, \k(x) \leq 8\d\bar{K} \quad \text{for all } y\in B_1\cap \{\tilde{u}>0\}.
\]
We fix $\d=\d(\bar{K},n)$ so small that $8\d\bar{K}\leq \e_0$. Proposition \ref{prop:epsreg}, applied to $\tilde u$ (after a rotation taking $\n\tilde u(0)$ to $e_n$), shows that $(\ol{\{\tilde{u}>0\}})_{1/2}(0)$ is the supergraph of a function $\tilde g$ with $\|\tilde g\|_{C^{k,\a}}\leq \bar c_k\,8\d\,\k(x)$, and that $\tilde{u}$ has a harmonic extension $\tilde{u}_e$ to $B_{1/2}$ with $\|\tilde{u}_e\|_{C^{k,\a}(B_{1/2})}\leq \bar c_k$. The map $y\mapsto x^R + 4\d y$ takes $B_{1/2}$ onto $B_{2\d}(x^R)$, and $u_R(x^R+4\d y)=4\d\,\tilde u(y)$. Therefore, $g_x(z'):=4\d\,\tilde g(z'/(4\d))$ and $\bar u(z):=4\d\,\tilde u_e((z-x^R)/(4\d))$ have the required properties, with $C_k$ depending only on $\bar c_k$ and $\d$.
\end{proof}

Since the free boundary is locally the graph of $g_x$, its curvatures are controlled by \eqref{prop:asymp:eq:g}. We write this down for the mean curvature $H$ of $\text{FB}(u)$ with respect to the outer unit normal $\nu$, which is the coefficient appearing in the second variation (Section \ref{sec:morse}) and in the linearized operator.

\begin{coro}\label{coro:curv}
Under the hypotheses of Proposition \ref{prop:asymp}, for every $l\in\mathbb{N}_0$ there is $C_l=C_l(\bar{K},n)$ such that the intrinsic derivatives of $H$ satisfy
\begin{equation}\label{coro:curv:eq}
|\n^l H(p)|\leq C_l\, \k(p)\,\la p\ra^{-1-l}\qquad\text{for all } p\in\text{FB}(u).
\end{equation}
\end{coro}
\begin{proof}
Let $p\in\text{FB}(u)$ and $R:=\la p\ra$. The mean curvature $H^R$ of $\text{FB}(u_R)$ at the point $p^R$ is a smooth expression in $D g_p$ and $D^2g_p$ that is linear in $D^2 g_p$, and the same holds for its derivatives. Hence \eqref{prop:asymp:eq:g} gives $|\n^l H^R|\leq C_l \k(p)$ on $\text{FB}(u_R)\cap B_{\d}(p^R)$. The scaling $H(q)=R^{-1}H^R(q/R)$ yields $|\n^lH(q)|=R^{-1-l}|\n^lH^R(q/R)|$, and \eqref{coro:curv:eq} follows by taking $q=p$. 
\end{proof}

\begin{rem}\label{rem:norm_coords} Assume that the solution $u$ satisfies (\ref{H1}) and (\ref{H2}). The local structure result in Proposition \ref{prop:asymp} implies that there exists a small $\theta = \theta(n,\bar{K})\in (0, \d/2)$ such that the exponential map from the normal bundle $\de\o\times \R$ of the free boundary $\text{FB}(u)=\de\o$ to $\R^{n}$,
\[
\de\o\times \R \ni (p,z)\to x=p-z\nu(p),
\]
when restricted to the set
$$
\mathcal{O}:=\{(p,z) \in \de\o \times \R: - \d_0 < z < 4 \theta \la p \ra \},
$$
is a diffeomorphism onto its image, which we denote by $\N_0$. In other words, every point $x$ in the tubular neighbourhood $\N_0$ is uniquely represented in the Fermi coordinates $(p,z)$. Furthermore, possibly reducing $\d_0$, we may assume that $u|_\o$ has a harmonic extension $\bar{u}$ to $\o\cup \N_0$ satisfying
\begin{equation}\label{eq:ext-decay}
|D^k \bar{u}(x)| \leq C_k \la x \ra^{1-k} \quad \text{for all } x\in \N_0, \quad \text{for each } k\in \mathbb{N}_0.
\end{equation}
\end{rem}

We also recall that classical solutions of \eqref{FBP} satisfy an interior Lipschitz estimate with a constant $L=L(n)>0$ and when they are global, they achieve the sharp constant $L=1$.

\begin{prop}\label{prop:grad} Let $D\subset\R^n$ be a domain and let $u:D\to [0,\infty)$ be a classical solution of \eqref{FBP} in $D$, with $0\in \text{FB}(u)$.
\begin{enumerate}[label=(\alph*)]
\item If $D=B_2$, then $\|\n u\|_{L^\infty(B_1)} \leq L$ for some $L=L(n)>0$.
\item If $D=\R^n$, then  $\|\n u\|_{L^\infty(\R^n)} = 1$. 
\end{enumerate} 
\end{prop}
\begin{proof}
The statement in (a) is covered in \cite[Lemma 1.5]{CafSalsa}. The sharp estimate in (b) was proved in \cite[Proposition A.5]{KW24} in the more general framework of viscosity solutions. 
\end{proof}

As a consequence of this sharp Lipschitz estimate, the area of the free boundary in $B_R$ is bounded by the area of $\de B_R$.
\begin{prop}\label{prop:areagrowth}
Let $u:\R^n\to [0,\infty)$ be a classical solution of \eqref{FBP} with $\text{FB}(u)\neq \emptyset$. Then for every $R>0$,
\[
\s(B_R\cap \text{FB}(u)) \leq  \s(\de B_R) = c(n) R^{n-1}.
\]
\end{prop}
\begin{proof}
Apply the Divergence theorem to $\n u$ in the piecewise smooth domain $\o\cap B_R$, using $u_\nu = -1$ on $\de\o\cap B_R$ and $|u_\nu|\leq 1$ on $\de B_R\cap\o.$
\end{proof}

\subsection{The double hairpin solution}\label{subsec:hairpin}

Using the standard identification $\R^2\simeq \C$, the double hairpin solution $HP:\mathbb{R}^2\to [0,\infty)$ of Hauswirth, H\'elein and Pacard \cite{HHP} is a global classical solution of \eqref{FBP} in $\R^2$, defined via the holomorphic map
\begin{equation}\label{HHPconfdifeo}
z = F(y) := y + \sinh y, \quad y=y_1+i y_2 \in \C,
\end{equation}
which is a conformal diffeomorphism from the strip $\I:=\{y=y_1+i y_2\in \C: |y_2|<\pi/2\}$ onto the planar domain
$$
\o := \{z=x_1+i x_2\in \C: |x_2|< \pi/2 + \cosh(x_1)\}.
$$
Denote $u_0(y) := \cosh y$. Then $HP$ is defined as
$$
HP(z) := \begin{cases} \text{Re}(u_0(y)) = \cosh y_1 \cos y_2, & \text{when } z= F(y)\in \o,\\ 0, & \text{otherwise.}\end{cases}
$$
Its free boundary $\text{FB}(HP)=\de\o$ consists of the two catenary curves
\[
x_2 = \pm (\pi /2 + \cosh x_1).
\]
In particular, it is not hard to verify that
\begin{equation*}
\exists\,\d_0>0 \text{ such that the rays } \{p+t\nu: 0\leq t < \d_0\}_{p\in \de\o}  \text{ into } \{HP=0\} \text{ are pairwise disjoint},
\end{equation*}
so that the double hairpin solution satisfies hypothesis (\ref{H2}). We also note that $HP(x_1, x_2)$ is invariant under the group of isometries of the plane generated by the reflections $x_i\mapsto -x_i$, $i=1,2$.

\begin{rem}\label{rem:disk-confcoords}
The solution $u=HP$ itself provides natural conformal coordinates for $\o$ on the unit disk. Indeed, since $u$ is harmonic and smooth up to the free boundary $\de\o$, the map $$\Phi(z):=2\de_z u = \de_{x_1}u - i \de_{x_2} u, \quad z:=x_1+ix_2 \in \o,$$ is holomorphic in $\omega$ up to the boundary $\de\o$, and satisfies
\[
|\Phi| = |\nabla u| = 1 \quad \text{on } \partial\o,
\]
on account of the free boundary condition. Moreover, $\Phi$ maps $\o$ biholomorphically onto the open unit disk $\mathbb{D}=\{\zeta\in \C: |\zeta|<1\}$ and extends to a smooth bijection from $\ol{\o}$ onto $\ol{\mathbb{D}}\setminus \{\pm 1\}$ (see  \cite[Section 4]{HHP}). Changing variables $\de \o \ni z \mapsto \zeta = \Phi(z)\in \de\mathbb{D}$, one gets the following transformation rule for the arclength measure $|dz|$ of the free boundary (see \cite[Lemma 5.1]{BK24}):
\begin{equation}\label{eq:transform}
|d\zeta| = |\Phi'(z)| |dz| = H(z) |d z|,
\end{equation}
where $H$ is the curvature of the free boundary $\de \o$ with respect to the outer unit normal $\nu$.
\end{rem}

The double hairpin enjoys an improved Hessian decay: it satisfies \eqref{eq:Hess} with $\k(x)=C_0\la x\ra^{-1}$.

\begin{lemma}\label{lem:Hessdecay}
Let $u=HP$, $\o=\{u>0\}$. Then
\begin{equation*}
|D^2 u(x)| \leq C_0\la x\ra^{-2} \quad \text{for all } x\in \o,
\end{equation*}
for some absolute constant $C_0>0$. In other words, $u$ satisfies \eqref{eq:Hess} with $\k(x)=C_0\la x\ra^{-1}$, and in particular (\ref{H1}).
\end{lemma}
\begin{proof}
Write $z= x_1 + i x_2=F(y)=y+\sinh y$ and let $U(z)=u_0(F^{-1}(z)) = \cosh y$ denote the holomorphic extension of $u$, so that $u=\text{Re}\, U$. The Cauchy--Riemann equations and the chain rule give
\begin{align*}
|D^2 u|^2 = 2 |\de_z^2 U |^2 = 2 \left| \de_z \left(\frac{u_0'(y)}{F'(y)}\right)\right|^2 = 2 \left|\de_z \frac{\sinh y}{1+\cosh y}\right|^2 = 2 \left|\frac{1}{1+\cosh y}\frac{1}{F'(y)}\right|^2 = \frac{2}{|1+\cosh y|^4}.
\end{align*}
For $y\in\ol{\I}$, we have $\text{Re}\cosh y=\cosh y_1\cos y_2\geq 0$, so $|1+\cosh y|\geq 1$. Moreover, the ratio $|y+\sinh y|/|1+\cosh y|$ is continuous on $\ol{\I}$ and tends to $1$ as $|y_1|\to\infty$, so it is bounded by an absolute constant. Hence $\la x\ra\leq 1+|y+\sinh y|\leq C|1+\cosh y|$, and
\[
|D^2 u(x)| = \frac{\sqrt{2}}{|1+\cosh y|^2} \leq \frac{C_0}{\la x\ra^2} \quad \text{for all } x\in \o.
\]
\end{proof}

Combining Lemma \ref{lem:Hessdecay} with Proposition \ref{prop:asymp} and Corollary \ref{coro:curv}, for which $\k(x)\leq 2C_0\la x\ra^{-1}$, we obtain improved higher-order free boundary bounds for the double hairpin.

\begin{coro}\label{coro:doublehpbounds} Let $u=HP$, $\o=\{HP>0\}$. There exist a small constant $\d>0$ and constants $C_k>0$, $k\in \mathbb{N}$, such that for any $x\in \text{FB}(u)$, the rescaled domain $(\ol{\o^R})_{2\d}(x^R)$, with $R:=\la x\ra$, is the supergraph within $B_{2\d}(x^R)$ of a function $g_x:\R \to \R$ with
\begin{equation*}
\|g_x\|_{C^{k,\a}(\R)}\leq C_k/ R,
\end{equation*}
and $u_R|_{(\ol{\o^R})_{2\d}(x^R)}$ has an extension to all of $B_{2\d}(x^R)$ as a harmonic function $\bar{u}_R$ satisfying
\begin{equation*}
\| \bar{u}_R \|_{C^{k,\a}(B_{2\d}(x^R))} \leq C_k.
\end{equation*}
Moreover, the curvature of $\text{FB}(u)$ satisfies $|\n^lH(p)|\leq C_l\la p\ra^{-2-l}$ for $p\in\text{FB}(u)$, $l\in\mathbb{N}_0$.
\end{coro}
\begin{proof}
This is Proposition \ref{prop:asymp} and Corollary \ref{coro:curv} with $\k(x)\leq 2C_0\la x\ra^{-1}$.
\end{proof}

\section{Second variation form and finite index solutions}\label{sec:morse}

In this section we provide some background results on finite Morse index solutions to the one-phase FBP. First, we prove that if a solution meets the natural Hessian bound \eqref{eq:Hess}, then the eigenfunctions associated to negative eigenvalues of Jacobi operator, i.e.\ of the Robin problem \eqref{eq:Robin0}, decay rapidly at infinity. Second, we focus on the double hairpin solution and establish the kernel nondegeneracy of its associated Jacobi operator.

Recall that the second variation form relative to the Alt--Caffarelli functional, evaluated at a classical solution $u:D \to [0,\infty)$ of \eqref{FBP}, is given by
\[
Q(\f, \f) = \int_{\o} |\n \f|^2 \, dx - \int_{\text{FB}(u)} H \f^2 \, d\s, \quad \text{for } \f \in C^1_c(\ol{\o}\cap D),
\]
where $\o:=\{u>0\}$ and $H$ denotes the mean curvature of $\text{FB}(u)$ with respect to the outer unit normal $\nu$. The \emph{Morse index} of $u$ in $D$ is defined to be the index of the quadratic form $Q$, that is
\[
\ind(u) = \sup\{\text{dim }V : V \text{ a subspace of } C^1_c(\ol{\o}\cap D), \text{ with } Q(\f, \f)< 0 \text{ for all } \f\in V\setminus \{0\}\}. 
\]

In \cite{BK24}, we showed that if the Morse index of a global classical solution $u: \R^n \to [0,\infty)$ is finite, then $\ind(u)$ is precisely the number of negative $L^2(\o)$ eigenvalues (counted with multiplicity) of the Jacobi operator associated to $Q$, that is, the number of eigenvalues $\l<0$, counted with multiplicity, of the Robin eigenvalue problem
\begin{equation}\label{eq:Robineig}
\left\{
\begin{aligned}
- \Delta w &= \l w && \text{ in } \o, \\
\de_\nu w - Hw &= 0 && \text{ on } \text{FB}(u)=\de\o,
\end{aligned}
\right.
\end{equation}
having associated eigenfunctions $w\in L^2(\o)$. 

In the next theorem we will strengthen this result by showing that the $L^2(\o)$-eigenfunctions associated to negative eigenvalues of \eqref{eq:Robineig} actually decay faster than any negative power at infinity, provided that the solution $u$ satisfies the natural Hessian decay \eqref{eq:Hess}. 

\begin{theorem}\label{thm:eigdecay}
Let $n\geq 2$ and let $u:\R^n \to [0,\infty)$ be a classical solution to \eqref{FBP} that satisfies the Hessian bound \eqref{eq:Hess}. Assume that the positive phase $\o\neq \R^n$. Let $w\in L^2(\o)\cap C^\infty(\ol{\o})$ be an eigenfunction of \eqref{eq:Robineig} with $\|w\|_{L^2(\o)}=1$, associated to a negative eigenvalue $\l<0$. Denote the rescales $u_R(x):=R^{-1}u(Rx)$ and $w_R(x):= w(Rx)$, $R>0$. Then for any $\a\in (0,1)$ and $\mu\geq 0$, there exists a constant $C_\mu >0$, independent of $R$, such that 
\begin{equation}\label{thm:eigdecay:eq:main}
R^{\mu}\|w_R\|_{C^{2,\a}(\{u_R>0\}\cap (B_2\setminus \ol{B_1}))} \leq C_\mu \quad \text{for all } R\geq 1. 
\end{equation}
In particular, $|w(x)|\leq C_\mu |x|^{-\mu}$ for all $x\in \o$. 
\end{theorem}
\begin{proof}
The constants below may depend on $u$, $n$, $k$, $\mu$, $\l$, $\a$ or $w$, but never on $R$ or on the points considered. By Corollary \ref{coro:curv}, the mean curvature and its tangential derivatives satisfy:
\begin{equation*}
|D^l H(x)| \leq C_l \la x \ra^{-l-1} \quad \text{for } l\in \mathbb{N}_0,
\end{equation*}
for some constants $C_l>0$. We will show that these imply the rapid decay of $w$ in the steps below. 
 
\begin{proof}[Step 1] First, we claim that $w\in H^1(\o)$. Indeed, plugging in the weak formulation of \eqref{eq:Robineig}
    \begin{equation}\label{eq:Robin:weak}
        \begin{aligned}
            \int_{\p \o} Hw \phi\,d\s = \int _{\o} \nabla w \cdot \n \phi + |\l| w \phi \,dx \qquad \forall \phi \in C^\infty_c(\overline{\o}),
        \end{aligned}
    \end{equation}
the test function $\phi=w\eta_R^2$, where $\eta_R\in C^\infty_c(\R^n)$ is a standard cutoff function, such that $0\leq \eta_R\leq 1$, $\eta_R \equiv 1$ in $B_R$, $\eta_R \equiv 0$ in $B_{2R}^c$ and $|\n \eta_R|\leq C/R$, yields
\begin{align}
\int _{\o} |\nabla w|^2 \eta_R^2 + |\l| w^2 \eta_R^2 \,dx & \leq \int _{\o} \nabla w \cdot \n  (w \eta_R^2) + |\l| w(w \eta_R^2) \, dx + 2 \int_\o |\n w| \eta_R |\n \eta_R||w| \, dx \notag \\
& =  \int_{\p \o} Hw^2\eta_R^2 \,d\s + 2 \int_\o |\n w| \eta_R |\n \eta_R||w|\, dx  \notag \\
& \leq \|H\|_{L^\infty(\de\o)} \int_{\de\o}w^2\eta_R^2 \,d\s + \frac{1}{4} \int_\o |\n w|^2 
\eta_R^2 \, dx + 4 \int_{\o} w^2 |\n \eta_R|^2\,dx, \label{thm:eigdecay:eq1}
\end{align}
where in the last line we used the boundedness of $H$ and Young's inequality. To estimate the boundary integral, we use the trace inequality \eqref{prop:trace:L1} from Proposition \ref{prop:trace} with $L=\|\n u\|_{L^\infty(\o)} = 1$ (the gradient bound from Proposition \ref{prop:grad}), in a combination with Young's inequality:
\begin{align}
 \int_{\de\o}w^2\eta_R^2 \,d\s &\leq \int_\o |\n (w\eta_R)^2|\, dx \leq  2 \int_\o |\n w| w \eta_R^2 + |\n \eta_R|\eta_R w^2 \, dx \notag \\
& \leq \frac{1}{4 \|H\|_{L^\infty(\de\o)}} \int_\o |\n w|^2 \eta_R^2 \, dx \notag  \\
& \quad + (4 \|H\|_{L^\infty(\de\o)}+1) \int_\o w^2 \eta_R^2 \, dx + \int_\o |\n \eta_R|^2 w^2 \, dx. \label{thm:eigdecay:eq2}
\end{align}
Combining \eqref{thm:eigdecay:eq1} and \eqref{thm:eigdecay:eq2}, we obtain
\begin{align*}
\int _{\o} (|\nabla w|^2 + |\l| w^2) \eta_R^2 \,dx & \leq \frac{1}{2} \int _{\o} |\nabla w|^2 \eta_R^2 + C\int_\o w^2 (\eta_R^2 + |\n \eta_R|^2)\, dx,
\end{align*}
so that absorbing the gradient squared term on the left-hand side gives
\begin{equation*}
\int _{\o} (|\nabla w|^2 + |\l| w^2) \eta_R^2 \,dx \leq 2 C \|w\|_{L^2(\o)}^2 = 2C.
\end{equation*}
Taking $R\to \infty$, we get that $\|w\|_{H^1(\o)}<\infty$.
\end{proof}

\begin{proof}[Step 2] We will now show that for every $k\in \mathbb{N}$,
\begin{equation}\label{eq:weightedbd}
\int _{\o} (|\nabla w|^2 + w^2) \langle x\rangle^k \, dx <\infty.
\end{equation}
Argue by induction: from Step 1 we know that \eqref{eq:weightedbd} is true for $k=0$ and assume that it is true up to $k-1$. Taking as a test function $\phi= \eta_R^2 w \langle x\rangle^k$ in \eqref{eq:Robin:weak}, we obtain
\begin{align}\label{eq:indstep1}
\int _{\o} (|\nabla w|^2+ |\l| w^2) \eta_R^2   \la x \ra^k \,dx & \leq \|H \la x \ra \|_{L^\infty(\de\o)} \int_{\de\o} w^2 \eta_R^2 \la x\ra^{k-1}\, d\s \notag \\ & + \int_\o |\n w| w \left(2 \eta_R |\n \eta_R| \la x\ra^k + \eta_R^2 |\n \la x\ra^k|\right) \,dx, 
\end{align}
where we used the fact that $\|H \la x \ra \|_{L^\infty(\de\o)}<\infty$. Given that $|\n \eta_R| \la x \ra^{k}\leq c \la x \ra^{k-1}$ and that $|\n \la x \ra^k |\leq k \la x \ra^{k-1}$, we can bound using Young's inequality
\begin{equation}\label{eq:indstep2}
\int_\o |\n w| w \left(2 \eta_R |\n \eta_R| \la x\ra^k + \eta_R^2 |\n \la x\ra^k|\right) \,dx \leq C_k \int_{\o} (|\n w|^2 + w^2) \la x \ra^{k-1} \, dx, 
\end{equation}
for some $C_k>0$. On the other hand, we can estimate the boundary integral using the trace inequality \eqref{prop:trace:L1}: 
\begin{align}\label{eq:indstep3}
\int_{\de\o} w^2 \eta_R^2 \la x\ra^{k-1}\, d s &\leq \int_{\o} |\n w^2| \eta_R^2 \la x\ra^{k-1}\, dx + \int_{\o} w^2 \left( |\n\eta_R^2| \la x\ra^{k-1} + \eta_R^2 |\n \la x\ra^{k-1}| \right) \, dx \notag \\
& \leq \tilde{C}_k \int_\o (w^2 + |\n w|^2) \la x\ra^{k-1}\, dx.
\end{align}
Combining \eqref{eq:indstep1}--\eqref{eq:indstep3}, and letting $R\to\infty$, we conclude the induction step, proving \eqref{eq:weightedbd}. We remark that the validity of \eqref{eq:weightedbd} trivially extends to noninteger $k\geq 0$.
\end{proof}
 
\begin{proof}[Step 3: pointwise decay of all derivatives.]
We show that for every $m\in\mathbb{N}_0$ and $\mu\geq 0$,
\begin{equation}\label{thm:eigdecay:eq:ptwise}
\|w\|_{C^m(S)}\leq C_{m,\mu}\la x\ra^{-\mu},
\end{equation}
where $S=S(x)$, $x\in\o$, is one of the following sets: $S=B_{\d/2}(x)$ if $d(x,\de\o)\geq\d$; and $S=S_\d(p)$ if $d(x,\de\o)<\d$, where $p\in\de\o$ is a point closest to $x$. 
 
The function $w$ solves $\D w+\l w=0$ with the fixed coefficient $\l$. If $d(x,\de\o)\ge\d$, interior elliptic estimates give $\|w\|_{C^m(B_{\d/2}(x))}\leq C_m\|w\|_{L^2(B_\d(x))}$. If $d(x,\de\o)<\d$, $w$ solves the Robin problem $\de_\nu w-Hw=0$ on the boundary portion $\de S_{2\d}(p)\cap B_{2\d}(p)$, which is a graph with $C^k$-norms bounded uniformly in $p$, while $H$ has uniformly bounded $C^k$-norms. The local regularity theory for the oblique derivative problem (\cite{Liebermanbook}) and the Sobolev embedding on the uniformly Lipschitz domain $S_\d(p)$ give $\|w\|_{C^m(S_\d(p))}\le C_m\|w\|_{L^2(S_{2\d}(p))}$. In both cases, the $L^2$-norm is taken over $y$ in a subset of $\o\cap B_{3\d}(x)$, on which $\la y\ra\geq\la x\ra-3\d\ge\la x\ra/2$ (recall $\d<1/8$). Hence, by \eqref{eq:weightedbd} with $k=2\mu$,
\[
\|w\|_{L^2(\o\cap B_{3\d}(x))}^2\leq 2^{2\mu}\la x\ra^{-2\mu}\int_\o w^2\la y\ra^{2\mu}\,dy\leq C_\mu\la x\ra^{-2\mu},
\]
which proves \eqref{thm:eigdecay:eq:ptwise}.
 
Now fix $R\ge 1$ and set $A_R:=\{u_R>0\}\cap(B_{5/2}\setminus\ol{B_{1/2}})$. Since $D^jw_R(y)=R^j(D^jw)(Ry)$ and $\la Ry\ra\geq R/2$ for $y\in A_R$, the bound \eqref{thm:eigdecay:eq:ptwise}, applied with $\mu+3$ in place of $\mu$ at the point $x=Ry$, gives the \emph{pointwise} estimates
\begin{equation}\label{thm:eigdecay:eq:rescaledptwise}
|D^jw_R(y)|\leq C\,R^j\la Ry\ra^{-\mu-3}\leq C_\mu R^{-\mu},\quad y\in A_R,\quad j=0,1,2,3.
\end{equation}
These control the $C^2$-part of the norm in \eqref{thm:eigdecay:eq:main}.
 
It remains to estimate the H\"older seminorm of $D^2w_R$, for which we use pieces of size comparable to $\d$ in the rescaled picture, as follows. Since $\la q\ra$ is comparable to $R$ for $q\in \text{FB}(u)\cap(B_{3R}\setminus B_{R/4})$, Proposition~\ref{prop:asymp} shows that each point $y\in\{u_R>0\}\cap(B_2\setminus\ol{B_1})$ lies in a set $P\subset A_R$ of one of the following two types:
\begin{itemize}
\item $P=B_{\d/2}(y)$, if $d(y,\text{FB}(u_R))\geq\d/2$;
\item $P$ is the connected component of $B_{\d}(q)\cap\{u_R>0\}$ that borders $q$, where $q\in\text{FB}(u_R)$ is a point closest to $y$; $P$ is the part of $B_\d(q)$ above a graph with Lipschitz constant at most $c_1\d\leq 1/2$.
\end{itemize}
In either case, any two points $y,y'\in P$ are joined by a path in $P$ of length at most $C|y-y'|$, with $C$ independent of $R$. Thus, by \eqref{thm:eigdecay:eq:rescaledptwise},
\[
\frac{|D^2w_R(y)-D^2w_R(y')|}{|y-y'|^\a}\leq C\,\sup_{P}|D^3w_R|\,|y-y'|^{1-\a}\leq C_\mu R^{-\mu},\quad y,y'\in P.
\]
As $\{u_R>0\}\cap(B_2\setminus\ol{B_1})$ is covered by a finite number of pieces $P$, depending only on $n$ and $\d$, the standard covering argument yields \eqref{thm:eigdecay:eq:main}. Finally, the pointwise bound $|w(x)|\le C_\mu|x|^{-\mu}$ is the case $m=0$ of \eqref{thm:eigdecay:eq:ptwise}.
\end{proof}
\end{proof}

As a consequence, the lowest eigenvalue $\l<0$ is simple, the associated ground state eigenfunction $w$ is of constant sign, it inherits the symmetry invariance of $u$ and has rapid decay at infinity. 

\begin{coro}\label{coro:grounddecay}
Let $u:\R^n\to [0,\infty)$ be a classical solution to the one-phase FBP, satisfying the Hessian bound \eqref{eq:Hess}. Assume that $0<\ind(u)<\infty$. Then the lowest $L^2(\o)$ eigenvalue $\l<0$ of \eqref{eq:Robineig} is simple and any eigenfunction $w\in L^2(\o)$, associated to $\l$, is either strictly positive or strictly negative in $\o$. Thus, there exists a unique positive, $L^2$-normalized eigenfunction $w$, associated with $\l$, and if $G_0$ is a group of isometries of $\R^n$ which leaves $u$ invariant, then $w$ is $G_0$-invariant, as well. Furthermore, for any $\a\in(0,1)$, $\mu\geq 0$, we have
\begin{equation}\label{coro:grounddecay:eq}
R^\mu \|w_R\|_{C^{2,\a}(\{u_R>0\}\cap (B_2\setminus \ol{B_1}))} \leq C_\mu
\end{equation}
for some constants $C_\mu>0$ independent of $R\geq 1$.
\end{coro}
\begin{proof}
The decay in \eqref{coro:grounddecay:eq} follows from the previous theorem and we only need to argue why $\l$ is simple and the associated eigenfunction is of constant sign. This is a standard argument (for instance, see \cite[Proposition 2.6]{BK24}), once one knows that $w\in H^1(\o)$. The $G_0$-invariance of $w$ is then a consequence of the uniqueness of $w$. 
\end{proof}

\subsection{Morse index and Jacobi operator kernel nondegeneracy for the double hairpin solution}

In our previous work \cite[Theorem 1.1]{BK24}, we showed that the double hairpin solution $HP$ has Morse index $1$, that is, the eigenvalue problem \eqref{eq:Robineig} has a unique negative $L^2$ eigenvalue $\l<0$, which is simple. Since the Hessian bound \eqref{eq:Hess} holds on account of Lemma \ref{lem:Hessdecay}, we may apply Corollary \ref{coro:grounddecay} to deduce that there is a unique $L^2(\o)$-normalized positive ground state $w_\l$ and it decays faster than any negative power at infinity. As $HP$ remains invariant under the action of the group $G_0$ of reflections $x_i\to -x_i$, $i=1,2$, we get that $w_\l$ is $G_0$-invariant. 

We will now proceed to establish the nondegeneracy of the kernel of the second variation form $Q$, associated to $\text{HP}$. Observe that $w_1 = HP_{x_1}$ and $w_2=HP_{x_2}$ solve \eqref{eq:Robineig} with $\l =0$; they are precisely the Jacobi fields arising from translations along the coordinate axes. We will show that any kernel element $w$ that is bounded and has finite Dirichlet energy
\begin{equation}\label{eq:finiteDir}
\int_\o |\n w|^2 \,dx <\infty,
\end{equation}
must be a linear combination of $w_1$ and $w_2$. 

\begin{prop}[Kernel nondegeneracy of the second variation]\label{prop:kernondeg} Let $u=HP$, $\o:=\{HP>0\}$. Let $w\in C^\infty(\ol{\o})$ be a solution of \eqref{eq:Robineig} for $\l = 0$. If $w$ is bounded and satisfies \eqref{eq:finiteDir}, then $$w=c_1 u_{x_1} + c_2 u_{x_2} \quad \text{for some } c_1, c_2\in \R. $$ 
\end{prop}
\begin{proof}
Let $z=x_1+i x_2$ and use the conformal change of variable $$z=x_1 + i x_2 \to y=\Phi(z) := 2\de_z u=u_{x_1} - i u_{x_2}$$ 
to map $\ol{\o}$ to the punctured closed disk $\ol{\mathbb{D}}_p :=\ol{\mathbb{D}}\setminus \{y=\pm 1\}$ (see Remark \ref{rem:disk-confcoords}). Then $v(y):=w(\Phi^{-1}(y))$ is bounded on $\ol{\mathbb{D}}_p$ and by the conformality of $\Phi$, we have
\[
\int_{\mathbb{D}} |\n_y v|^2 \, dy = \int_{\o} |\n_x w|^2 \, dx < \infty.
\]
Thus, $v\in H^1(\mathbb{D})$ and has a trace on $\de \mathbb{D}$ (which we again denote by $v$) lying in $H^{1/2}(\de\mathbb{D})$. The weak formulation of \eqref{eq:Robineig} for $\l=0$ says $Q(w,\f) = 0$ for all $\f\in C^1_c(\ol{\o})$, and transforms in the $y$-variable to 
\begin{align}\label{prop:kernondeg:eq:transf}
0 =\int_\o \n w \cdot \n \f \, dx - \int_{\de\o} H w\f \, d\s = \int_{\mathbb{D}} \n v \cdot \n \psi\, dy - \int_{\de \mathbb{D}} v \psi \, d\s
\end{align}
for all $\psi = \phi\circ \Phi^{-1} \in C^1_c(\ol{\mathbb{D}}_p)$, on account of the transformation rule \eqref{eq:transform} for the boundary measure. Since the singular set $\{\pm 1\}\subset \de\mathbb{D}$ has zero $H^1$-capacity, \eqref{prop:kernondeg:eq:transf} holds by density of $C^1_c(\ol{\mathbb{D}}_p)\subset H^1(\mathbb{D})$ for any $\psi\in H^1(\mathbb{D})$. Taking a constant test function $\psi \equiv 1$, we see that $\int_{\de\mathbb{D}} v = 0$, so that $v\in H^1(\mathbb{D})$ is a harmonic function with mean $0$. 

Expand the mean-zero trace $v|_{\de \mathbb{D}} = \sum_{k=1}^\infty (c_k e_k(\theta) + d_k f_k(\theta))$ in Fourier series, where $e_k(\theta):=\cos(k\theta)/\sqrt{\pi}$ and $f_k(\theta):=\sin(k\theta)/\sqrt{\pi}$, $k\in\mathbb{N}$, form an orthonormal set in $L^2(\de\mathbb{D})$. Given that $v|_{\de \mathbb{D}}\in H^{1/2}(\de\mathbb{D})$, we have $\sum_{k=1}^\infty k (c_k^2+d_k^2)<\infty$, meaning that we can express the harmonic $v$ in polar coordinates as a series, converging in $H^1(\mathbb{D})$:
\[
v(r,\theta) = \sum_{k=1}^\infty v_k(r,\theta), \quad \text{with } v_k(r,\theta):= r^k (c_k e_k(\theta) + d_k f_k(\theta)).
\]
Plugging in $\psi = v$ in \eqref{prop:kernondeg:eq:transf} and using the harmonicity of each $v_k$, we  get
\[
0 = \int_{\mathbb{D}} |\n v|^2 \, dy - \int_{\de \mathbb{D}} v^2 \, ds = \sum_{k=1}^\infty \int_{\de \mathbb{D}} (\de_r v_k - v_k) v_k \, ds =  \sum_{k=1}^\infty (k-1)(c_k^2 + d_k^2). 
\]
The above implies that $c_k=d_k=0$ for $k\geq 2$, so that 
$$
w(x)= v(y) = c_1 r e_1(\theta) + d_1 r f_1(\theta) = \tilde{c}_1 y_1 + \tilde{d}_1 y_2. 
$$
We complete the proof of the proposition by observing that $y_1 = u_{x_1}$ and $y_2 = - u_{x_2}$. 
\end{proof}

\subsection{Distance nondegeneracy for stable and finite index solutions}

In this subsection we establish Theorem \ref{thm:nondeg}. The starting point is the nondegeneracy estimate for stable classical solutions, obtained by Kelei Wang and the second-named author:
\begin{theorem}[\cite{KW24}]\label{thm:KW} There is $c_0=c_0(n)>0$ such that if $u$ is a stable classical solution of \eqref{FBP} in $B_2$, then for every $x\in \text{FB}(u)\cap B_{1/2}$, $r\in (0,1/2)$, 
\[
\sup_{B_r(x)} u \geq \dashint_{\de B_r(x)} u \geq c_0(n) r. 
\]
\end{theorem}

The second ingredient is the following compactness result of Kriventsov and Weiss \cite{Kriventsov-Weiss25}.
\begin{theorem}[Theorem 9.3 in \cite{Kriventsov-Weiss25}]\label{thm:KrivW}
Let $u_k:B_1\to \R$ be a sequence of classical solutions of \eqref{FBP} in $B_1$, with $\sup_{k}\|u _k\|_{C^{0,1}(B_1)}<\infty$. 
Then, up to subsequence, $u_k\to u$ locally uniformly on $B_1$ and if $u\not\equiv 0$, then $\chrc{\{u_k>0\}}\to \chrc{\{u>0\}}$ in $L^1_\text{loc}(B_1)$.
\end{theorem}

Armed with these, we can provide a quick proof of Theorem \ref{thm:nondeg}(a).

\begin{proof}[Proof of Theorem \ref{thm:nondeg}(a)]
Assume by contradiction that the statement is false. Then there exist classical solutions $u_k$ in domains $D_k\subset \R^n$ and points $x_k\in \{u_k>0\}$, $y_k\in \text{FB}(u_k)$, such that for $d_k:=d(x_k, \text{FB}(u_k))=|x_k-y_k|$, one has  $B_{4d_k}(y_k)\subset D_k$, but $u_k(x_k)\leq d_k/k$. Let us rescale and recenter, setting $w_k(z):=u_k(y_k+d_k z)/d_k$ for $z\in B_{4}$. Then $0\in \text{FB}(w_k)$ and after a possible rotation, we may assume that $w_k(e_n)=u_k(x_k)/d_k\leq 1/k$, with $B_1(e_n)\subset \{w_k>0\}$.  The Harnack inequality implies that $w_k\leq c_1 /k$ in $B_{1/2}(e_n)\subset B_{3/2}$ for some $c_1=c_1(n)>0$. Since $w_k$ is stable in $B_4$ and $0\in \text{FB}(w_k)$, the nondegeneracy estimate in Theorem \ref{thm:KW} implies that $\sup_{B_{1/4}} w_k \geq c_0/4$. Furthermore,  as $w_k$ satisfies a uniform Lipschitz constant $L=L(n)$ in $B_2$ by Proposition \ref{prop:grad}(a), $(w_k)_k$ converges (up to taking a subsequence) uniformly to a Lipschitz-continuous function $w_\infty$ in $B_2$ with  $w_\infty = 0$ in $B_{1/2}(e_n)$, yet $\sup_{B_{1/4}} w_\infty \geq c_0/4$. Thus, $w_\infty\not\equiv 0$ in $B_2$, so by Theorem \ref{thm:KrivW}, $\chrc{\{w_k>0\}}\to \chrc{\{w_\infty>0\}}$ in $L^1(B_{1/2}(e_n))$. We have reached a contradiction, since $|\{w_k>0\}\cap B_{1/2}(e_n)| = |B_{1/2}|$, while $|\{w_\infty>0\}\cap B_{1/2}(e_n)| = 0$.
\end{proof}

We end this section with the proof of Theorem \ref{thm:nondeg}(b) regarding the nondegeneracy estimate for finite Morse index solutions of \eqref{FBP}, having noncompact free boundary.
\begin{proof}[Proof of Theorem \ref{thm:nondeg}(b)]
Since $u$ has finite Morse index, $u$ is stable outside a compact set \cite{BK24}. After recentering and rescaling, we may assume that $0\in \text{FB}(u)$ and $u$ is stable in $\R^n\setminus \ol{B_1}$. We observe that Theorem \ref{thm:nondeg}(a) applies for every $x\in \o:=\{u>0\}$ with $|x|\geq 1+ 5d(x)$, since then $|y|\geq |x|-d(x)\geq 1+ 4d(x)$ so that $B_{4d(x)}(y)\subset \R^n\setminus \ol{B_1}$ is inside the stability region of $u$. Thus,
\begin{equation}\label{thm:nondeg:b:eq:1}
u(x)\geq c d(x) \quad \text{whenever } x\in \o, \text{ with } |x|\geq 1 + 5 d(x). 
\end{equation}

Assume by contradiction that the theorem is false. Then, by \eqref{thm:nondeg:b:eq:1} there exist points $x_k\in E:=\{x\in \o: |x|<1+5d(x)\}$, for which $u(x_k)\leq d_k/k$, where  $d_k:=d(x_k)=|x_k-y_k|$ is the distance to the free boundary, realized at $y_k\in \de\o$.  Since $\de B_{d_k}(x_k)$ touches $\de\o$ at $y_k$, we have $\nu_k:= \frac{x_k-y_k}{d_k} = \n u(y_k)$. There are two possibilities: either $(d_k)_k$ (or a subsequence) is bounded, or $d_k\to \infty$.

In the case when the (sub)sequence $d_k\leq M<\infty$, we know that $(x_k)$ are contained in $\ol{B_{1+5M}}$. Passing, if necessary, to another subsequence, we have $x_k\to x_\infty\in\ol{\o}$. If $x_\infty\in \o$, then $u(x_k)/d_k\to u(x_\infty)/d(x_\infty)>0$. If $x_\infty\in\de\o$, then $d_k\to 0$, $y_k\to x_\infty$, and by uniform continuity of $\nabla u$ on $\ol{\o}\cap B_1(x_\infty)$,
\[
u(x_k)=\int_0^{d_k}\nabla u(y_k+t \nu_k)\cdot\nu_k\,dt=d_k\big(|\nabla u(y_k)|^2+o(1)\big)=d_k(1+o(1)).
\]
In both scenarios, we contradict $u(x_k)\leq d_k/k$. 

Therefore, it must be that $d_k\to \infty$. Since $0\in \text{FB}(u)$ and $x_k\in E$, we know that 
$$d_k\leq |x_k| < 1 + 5d_k\leq 6d_k,$$
so that $d_k$ and $|x_k|$ are comparably large. Consider the blowdowns $w_k(x):=u(d_k x)/d_k$ and set $z_k:=x_k/d_k \in B_6\setminus B_{1}$. Up to taking a subsequence, $z_k\to z_\infty \in \ol{B_6}\setminus B_{1}$. Then for all large $k$, $B_{2/3}(z_\infty)\subset B_{1}(z_k)\subset \{w_k>0\}$, so that by the Harnack inequality, $$\sup_{B_{1/2}(z_\infty)}w_k \leq c_1 w_k(z_k) \leq c_1 /k,$$ for a constant $c_1=c_1(n)>0$. Furthermore, as  $0\in \text{FB}(w_k)$ and $w_k$ are uniformly Lipschitz in $\R^n$, $(w_k)_k$ subconverges locally uniformly in $\R^n$ to a Lipschitz continuous $w_\infty:\R^n\to \R_{\geq 0}$. If $w_\infty\not\equiv 0$, then Theorem \ref{thm:KrivW} tells us that $\chrc{\{w_k>0\}} \to \chrc{\{w_\infty>0\}}$ in $L^1_\text{loc}$, which is impossible as $w_k>0$ in $B_{1/2}(z_\infty)$, while $w_\infty=0$ in $B_{1/2}(z_\infty)$. 

Hence, $w_\infty \equiv 0$. We will show that this contradicts the fact that $\text{FB}(u)$ is noncompact. First, claim that for any given $N>6$, 
\begin{equation*}
\text{FB}(w_k)\cap \{6 \leq |z| \leq N\} =\emptyset \quad \text{for all large enough } k\geq k_0(N).
\end{equation*}
Indeed, if it were the case that $p_k\in\text{FB}(w_k)\cap\{6\leq |z|\leq N\}$ for some subsequence $p_k\to p_\infty\in \{6\leq |z|\leq N\}$, then by the stability of $w_k$ in $\{4\leq |z|\leq N+2\}$, we would get from the nondegeneracy estimate in Theorem \ref{thm:KW} that $\sup_{B_{1}(p_\infty)} w_k\geq \sup_{B_{1/4}(p_k)} w_k \geq c_0/4$. This entails $\sup_{B_1(p_\infty)} w_\infty \geq c_0/4$ by uniform convergence, which contradicts the fact that $w_\infty \equiv 0$. We deduce that $\text{FB}(u)\cap \{6 d_k \leq  |x| \leq N d_k\} =\emptyset$ for large enough $k\geq k_0(N)$. 

Since $\text{FB}(u)$ is closed and unbounded, $$\rho_k:=\min\{|y|: y\in \text{FB}(u), |y|\geq 6 d_k\}$$ 
is attained at some $q_k\in \de B_{\rho_k}\cap \text{FB}(u)$, and by the above $\rho_k/d_k \to \infty$ (up to taking a subsequence). The annulus $A_k:=\{6 d_k<|x| < \rho_k\}$ is connected and is disjoint from $\text{FB}(u)$, so either $A_k\subset \o$ or $A_k\cap \o = \emptyset$. In the latter case, the component of $\o$ containing $x_k \in B_{6d_k}$ must lie inside $B_{6d_k}$, which is impossible, as every component of $\o$ is unbounded (a bounded domain cannot support a positive harmonic function that vanishes on its boundary). Therefore, $A_k\subset \o$.

Consider now the blowdowns $v_k(z):=u(\rho_k z)/\rho_k$, which are positive and harmonic on the annulus $\{6d_k/\rho_k< |z| < 1\}$, whose inner radius $6d_k/\rho_k \to 0$, and let $\bar{y}_k:=q_k/\rho_k\in \de B_1$. Along a subsequence, $v_k\to v_\infty$ locally uniformly, so that $v_\infty$ is a nonnegative bounded harmonic function in $B_1\setminus \{0\}$ with $v_\infty(0)=0$. The singularity is removable, thus $v_\infty\equiv 0$ in $B_1$ by the strong maximum principle. On the other hand, by Theorem \ref{thm:nondeg}(a), we know that at $\bar{x}_k:=\frac{7}{8}\bar{y}_k$, $$v_k(\bar{x}_k)\geq c d( \bar{x}_k, \text{FB}(v_k)) = c/8,$$ because $v_k$ is stable in $B_{1/2}(\bar{y}_k)$. Hence, $\sup_{B_1} v_k\geq c/8$, which contradicts the fact that $v_k\to 0$ in $B_1$, and completes the proof.
\end{proof}

\section{Setting up the FBP as a nonlinear operator equation\protect\footnote{The calculations in this section and the accompanying Appendix~\ref{app:technical} were first carried out by the authors and then streamlined with the help of Claude Opus 5.5.} }\label{sec:setup}

In this section we recast the problem of finding $T$-periodic solutions of \eqref{FBP}
near a given product solution as a family of nonlinear operator equations, acting between Banach spaces that do not depend on $T$. The period enters
through the anisotropic dilation $S_T(x,t)=(x,Tt)$. Pulling back by $S_T$ turns $T$-periodic
solutions into $1$-periodic solutions of \eqref{FBP} for the metric
$g_T:=dx^2+T^2dt^2$ (see Subsection~\ref{ss:rescaling}). We then deform the product domain in the normal direction by a function $\f$ on the free boundary, pull the problem back to the fixed product domain $\O=\o\times \R$ (Subsection~\ref{ss:operator}), and express the boundary value problem in $\O$ as an operator equation $$N(T,\f,v)=(N_1, N_2, N_3)(T,\f,v)=(0,0,0),$$ where $N_1$ stands for the transformed PDE in the interior of  $\O$, while $N_2$ and $N_3$ -- for the gradient and Dirichlet boundary conditions on $\G$. Eventually, we will be able to eliminate $\f$ via the Dirichlet condition (Subsection~\ref{ss:reduction}), reducing the problem to the operator $M(T,v)$ to which the bifurcation argument of Section \ref{sec:bif} is later applied. The resulting operators are \emph{affine} in $T^{-2}$ (Theorem~\ref{thm:N}), so their regularity in $T$ is immediate.

\subsection{Setting}\label{ss:setting}

Throughout this section the base solution $u$ falls under one of the following
two scenarios:
\begin{description}
\item[Case A] $u:\R^2\to\R$ is the double hairpin solution $u=HP$;
\item[Case B] $u:\R^n\to\R$, $n\geq 3$, satisfies the Hessian bound \eqref{eq:Hess} and (\ref{H2}).
\end{description}

Denote $\o:=\{u>0\}$, $\O:=\o\times\R$ and $\G:=\p\O=\p\o\times\R$.
The outer unit normal to $\G$ is horizontal, $\nu(p_x,t)=(\nu_\o(p_x),0)$. Let
\[
  \r(X):=\la x\ra,\qquad X=(x,t)\in\R^n\times\R.
\]
and note that $\r\geq 1$ is $G_0$-invariant because $G_0\subset O(n)$. Let $\bar u$ be the harmonic extension of $u|_\o$ to $\o\cup \N_0$, where $\N_0$ is the tubular neighbourhood of $\p\o$ from Remark \ref{rem:norm_coords}, and set
\[
  U(x,t):=\bar u(x)\qquad\text{for } (x,t)\in\mathcal D:=(\o\cup \N_0)\times\R.
\]
By \eqref{eq:ext-decay}
\begin{equation}\label{eq:U-bounds}
  |D^kU(X)|\le c_k\la x\ra^{1-k}\qquad\text{for } X=(x,t)\in\mathcal D,\ k\in\mathbb{N}_0.
\end{equation}
For a rescaling parameter $R>0$, a set $S\subset\R^{n+1}$ and a function
$f:S\to\R^d$ we write $$S^R:=\{X/R: X\in S\} \quad\text{and}\quad f_R(Y):=f(RY), ~ Y\in S^R.$$

\noindent\emph{The period as a parameter}. For $T>0$ let
\[
  S_T(x,t):=(x,Tt),\qquad g_T:=S_T^*g_0=dx^2+T^2dt^2,
\]
where $g_0=dx^2+dt^2$ is the Euclidean metric. We also write $S_T^*f:=f\circ S_T$ for functions. Thus $S_T:(\R^{n+1},g_T)\to(\R^{n+1},g_0)$ is an isometry, and a function $f$ is $T$-periodic in $t$ if and only if $S_T^*f$ is $1$-periodic. By this change of variables, \eqref{FBP} becomes the boundary value problem
\begin{equation}\label{eq:FBP-g}
  \D_{g_T} \tilde{U}=0\ \text{ and }\ \tilde U>0\ \text{ in }\tilde\O,\qquad
  \tilde U=0\ \text{ and }\ |d\tilde U|_{g_T}=1\ \text{ on }\p\tilde\O.
\end{equation}

The product solution is trivially periodic for any period: $S_T$ maps $\O$ and $\G$ onto themselves and $S_T^*U=U$. Moreover, since $U$ does not depend on $t$,
\begin{equation*}
  \D_{g_T}U=\D_x\bar u=0,\qquad |dU|_{g_T}=|\n_x\bar u|\qquad\text{in }\mathcal D,
\end{equation*}
so $(\O,U)$ solves \eqref{eq:FBP-g} with $g=g_T$ for every $T>0$. We will look for
$1$-periodic solutions of \eqref{eq:FBP-g}, $g=g_T$, near $(\O,U)$, treating $T$ as a parameter.

\emph{Symmetry groups.} For $\t>0$ let $G_1(\t)$ be the group of isometries of $\R$ generated by
$t\mapsto t+\t$ and $t\mapsto -t$, and set
\[
  G(\t):=G_0\times G_1(\t)
\]
Every element of $G(\t)$ is an isometry of $g_T$ for every $T>0$, and preserves $\O$, $\G$, $\nu$,
$\r$ and $U$. Moreover, $S_T^*$ maps $G(T)$-invariant functions to $G(1)$-invariant functions. The function spaces introduced below will consist of $G(1)$-invariant functions; in particular, they do not depend on $T$.

\subsection{Fermi coordinates and normal deformations}\label{ss:Fermi}

We measure the Fermi coordinate $z$ \emph{into} $\O$: every point $X$ in
\begin{equation}\label{eq:N1vicinity}
\N_1:=\{X=p-z\nu(p):\ p\in\G,\ -\d_0<z<4\th \rho(p)\}
\end{equation}
is uniquely represented as $X=p-z\nu(p)$ with $(p,z)$ as above, so that
$z>0$ in $\O$ and $z<0$ in the zero phase. That this is a diffeomorphism
follows from Remark \ref{rem:norm_coords}, since $\G$ has a product structure. We record that
\begin{equation*}
  B_{\d_0}(X)\subset\mathcal D\qquad\text{for every } X\in\ol{\O} .
\end{equation*}
Reducing $\d_0$ and $\theta$, if necessary, we may assume that $\d_0\leq \theta$ and by \eqref{eq:U-bounds} that
\begin{equation}\label{eq:theta-small}
  \tfrac12 \rho(p)\le  \r(X) \le 2 \rho(p) \quad\text{and}\quad
  |D^2U(X)|\le \frac{1}{16\,\th\, \rho(p)}\qquad\text{for } X=p-z\nu(p)\in \N_1.
\end{equation}

Fix a cutoff function $\chi\in C^\infty([0,\infty))$ with $0\le\chi\le1$, $\chi=1$ on $[0,2]$,
$\chi=0$ on $[3,\infty)$, and use it to extend any given $\f:\G\to\R$ and the outer unit normal $\nu$ to $\ol{\O}$ by
\begin{equation}\label{eq:extensions}
  (\bar\nu,\bar\f)(X):=
  \begin{cases}
    \chi\big(\tfrac{z}{\th\r(p)}\big)\,(\nu(p),\f(p)) & X=p-z\nu(p)\in\ol{\O}\cap \N_1,\\
    0 & X\in\ol{\O}\setminus \N_1.
  \end{cases}
\end{equation}
Now define the normal deformation field and the deformation map
\begin{equation*}
  h_\f:=\bar\f\,\bar\nu,\qquad \Phi_\f(X):=X+h_\f(X),\qquad X\in\ol{\O} .
\end{equation*}
Note that $h_\f$ is \emph{horizontal}, $h_\f=(h'_\f,0)$, with $h_\f|_\G=\f\nu$, and that $h_\f=0$ outside $\{0\le z\le 3\th\r(p)\}$.

\begin{lemma}\label{lem:deformation}
Let $\varrho_0:=\min\{\d_0/2,\ \th/(4\|\chi'\|_\infty)\}$ and let
$\f\in C^{2,\a}_{\rm loc}(\G)$ with $\|\f\|_{L^\infty(\G)}<\varrho_0$.
\begin{enumerate}
\item[(a)] $\Phi_\f$ preserves each normal line: in Fermi coordinates
  \begin{equation*}
    \Phi_\f\big(p-z\nu(p)\big)=p-\z_\f(p,z)\,\nu(p),\qquad
    \z_\f(p,z):=z-\chi\big(\tfrac{z}{\th\r(p)}\big)^2\f(p),
  \end{equation*}
  for $0\le z<4\th\r(p)$, $z\mapsto \z_\f(p,z)$ is increasing for each $p\in\G$, and $\Phi_\f=\id$ on $\ol{\O}\setminus \N_1$.
  
\item[(b)] $\Phi_\f$ is a $C^{2,\a}$-diffeomorphism of $\ol{\O}$ onto its image
  $\ol{\O_\f}$, where $\O_\f:=\Phi_\f(\O)\subset\mathcal D$. Moreover $\det D\Phi_\f>0$ and
  $\Phi_\f(p)=p+\f(p)\nu(p)$ for $p\in\G$.

\item[(c)] If $\g\in G(\t)$ for some $\t>0$, then $\Phi_{\f\circ\g}=\g^{-1}\circ\Phi_\f\circ\g$. In particular,
  if $\f$ is $G(\t)$-invariant, then $\Phi_\f$ commutes with every element of $G(\t)$.
\end{enumerate}
\end{lemma}

\begin{proof}
Part (a) is simply a restatement of the definitions. For (b), we have
$$\p_z\z_\f=1-2\chi\chi'(z/\th\r)\,\f(p)/(\th\r(p))\ge 1/2,$$
while $\z_\f(p,0)=-\f(p)>-\d_0$ and $\z_\f(p,z)=z$ for $z\ge3\th\r(p)$.
Hence, for each $p$, $z\mapsto\z_\f(p,z)$ is an increasing bijection of
$[0,4\th\r(p))$ onto $[-\f(p),4\th\r(p))$. So $\Phi_\f$ maps
$\ol{\O}\cap N_1$ bijectively onto $\{p-\z\nu(p):-\f(p)\le\z<4\th\r(p)\}\subset \N_1$
and is the identity elsewhere; it is therefore injective on $\ol{\O}$.
In the Fermi chart it reads $(p,z)\mapsto(p,\z_\f(p,z))$ with $\p_z\z_\f>0$,
so $\det D\Phi_\f>0$. 

For (c), $\g$ is a Euclidean isometry that preserves $\O$, $\G$, the outer
normal ($\nu(\g p)=D\g\,\nu(p)$) and $\r$. Hence $\g$ maps the point with
Fermi coordinates $(p,z)$ to the one with Fermi coordinates $(\g p,z)$, and
$h_{\f\circ\g}=D\g^{-1}\,(h_\f\circ\g)$, which is the claim.
\end{proof}

We will now present the weighted H\"older spaces of functions, in which we will be working. 

\subsection{Weighted H\"older spaces}
We first define the relevant spaces over subsets of the base $\ol{\o}$.  
\begin{defn}
For a point $x\in S$ in a subset $S\subset \R^n$ and $r>0$, recall that $S_r(x)$ denotes the connected component of $S\cap B_r(x)$ (in the relative topology) containing $x$. 

Let $k\in \mathbb{N}_0$, $\a\in (0,1]$ be a fixed H\"older exponent and $\b\geq 0$ be a fixed decay rate. Let $S$ be an open set  of $\ol{\o}$ (respectively of $\de\o$) in the relative topology.  For $g\in C(S)$ and $f\in C^{k,\a}_\text{loc}(S)$, define the $\bar{C}_{\b}(S)$-norm and the $\bar{C}^{k,\a}_{\b}(S)$-norm respectively by 
\begin{align*}
\|g\|_{\bar{C}_{\b}(S)} &:= \sup_{x\in S} \{ R^\b |g(x)|: R=\la x \ra\} \quad \text{and} \\ 
\|f\|_{\bar{C}^{k,\a}_{\b}(S)} &:= \sup_{x\in S} \{R^\b \|f_R\|_{C^{k,\a}([S^R]_{1/4}(x^R))}: R=\la x \ra\}. 
\end{align*}
Define $\bar{C}_{\b}(S)$ to be space of functions $g\in C(S)$ of finite $\bar{C}_{\b}(S)$-norm, while $\bar{C}^{k,\a}_{\b}(S)$ -- the space of functions $f\in C^{k,\a}_\text{loc}(S)$ that have finite $\bar{C}^{k,\a}_\b(S)$-norm.  These are all Banach spaces  and
$$|D^lf(x)|\le\|f\|_{\bar C^{k,\a}_\b(S)}\la x\ra^{-(\b+l)} \quad \text{for } 0\le l\le k.$$
A subscript $G_0$ will denote the corresponding closed subspace of $G_0$-invariant functions.
\end{defn} 

In a similar fashion, we next define the weighted H\"older spaces on subdomains of the product domain $\O$ and its boundary $\G$. 

\begin{defn}
Let $k, \a,\b$ be as above. Let $S$ be an open subset of $\ol{\O}$ (respectively of $\G$) in the relative topology.  For $g\in C(S)$ and $f\in C^{k,\a}_\text{loc}(S)$ define
\begin{align*}
\|g\|_{C_\b(S)} &:=\sup_{X=(x,t)\in S} \{ R^\b |g(X)|: R=\la x \ra\} \quad \text{and} \\
\|f\|_{C^{k,\a}_{\b}(S)} &:=\sup_{X=(x,t)\in S}  \{R^\b \|f_R\|_{C^{k,\a}([S^R]_{1/4}(X^R))}: R=\la x \ra\},
\end{align*}
and let $C_\b(S)$, $C^{k,\a}_{\b}(S)$ be the corresponding Banach spaces. Vector- and matrix-valued functions are treated componentwise. Note that if $f\in C^{k,\a}_\b(S)$, then derivatives of $f$ of order $l=a+b\leq k$ belong to $C^{k-l, \a}_{\b+l}(S)$; in particular, they decay like:
\[
|\de_x^a \de_t^b f(x,t)| \leq \|f\|_{C^{k,\a}_\b(S)} \la x\ra^{-(\b+a+b)}, \quad (x,t)\in S, \quad  0\leq a+b\leq k.
\]
Moreover, $C^{k,\a}_{\b}(S)\emb C^{l,\a}_{\b'}(S)$ for $k\geq l$ and $\b\geq \b'$, and there is $C_*=C_*(k,\a)$ with
\begin{equation}\label{eq:product}
\|f_1 f_2\|_{C^{k,\a}_{\b_1+\b_2}(S)}\leq C_* \|f_1\|_{C^{k,\a}_{\b_1}(S)}\| f_2\|_{C^{k,\a}_{\b_2}(S)}.
\end{equation}
For $\b>0$ and $S=\ol{\O}$ or $\G$, we will also work with the Banach spaces $\X^{k,\a}_\b(S):= C^{k,\a}_\b(S) \oplus \R$, normed by 
\[
\|f+c\|_{\X^{k,\a}_\b(S)}:= \|f\|_{C^{k,\a}_\b(S)} + |c|,
\]
(the decomposition is unique as $\b>0$). A subscript $G(\t)$ denotes the closed subspace of $G(\t)$-invariant functions.
\end{defn}

\begin{rem}\label{rem:conncomp} In the definition of the weighted H\"older spaces, after rescaling, we take the connected component of the set $\o^R\cap B_{1/4}(x^R)$ in order to keep the norm local when the zero phase of the solution is thin, as for the double hairpin.
\end{rem}

We set for each of the two cases at the start of the section:
\begin{align}
  &(\textbf{Case A})\quad \X:=\X^{2,\a}_\b(\ol{\O}),\quad \Y:=\X^{2,\a}_\b(\G),\quad \b\in(0,1];\label{eq:spacesA}\\
  &(\textbf{Case B})\quad \X:=C^{2,\a}_\b(\ol{\O}),\quad \Y:=C^{2,\a}_\b(\G),\quad \b\in(0,\infty);\label{eq:spacesB}
\end{align}
and in both cases $\Z:=C^{0,\a}_{\b+2}(\ol{\O})\times C^{1,\a}_{\b+1}(\G)$.
Note that $\Y\emb C^{2,\a}_0(\G)$ with norm at most $1$, that
$w\mapsto w|_\G$ is bounded $\X\to\Y$, and that the trace operator $w\mapsto\n w|_\G$ is bounded
$\X\to C^{1,\a}_{\b+1}(\G)$.

We record the asymptotic properties of the geometric data. By Proposition \ref{prop:asymp} and Corollary \ref{coro:curv}, we have in Case B
\begin{equation}\label{eq:H-B}
  H\in C^{1,\a}_1(\G),\qquad \nu\in C^{2,\a}_0(\G),
\end{equation}
while in Case A the improved bounds of Corollary \ref{coro:doublehpbounds} give
\begin{equation}\label{eq:H-A}
  H\in C^{1,\a}_2(\G)\emb C^{1,\a}_{\b+1}(\G)\quad\text{for }\b\in(0,1].
\end{equation}
By Lemma~\ref{lem:hphi}, the map $\f\mapsto h_\f$ is bounded and linear from $\Y$ into the horizontal vector
fields in $C^{2,\a}_0(\O;\R^{n+1})$, with norm at most $C_1$.

\subsection{The pulled-back problem and the operator}\label{ss:operator}

Treating $\tilde{\O}$ in \eqref{eq:FBP-g} as a normal deformation $\tilde{\O}=\O_\f:= \Phi_\f(\O)$, it is convenient to pull the problem back to $\O$: 
\begin{equation}\label{eq:FBP-pback}
  \D_{g_{T,\f}} \tilde{U}=0\ \text{ and }\ \tilde U>0\ \text{ in }\O,\qquad
  \tilde U=0\ \text{ and }\ |d\tilde U|_{g_{T,\f}}=1\ \text{ on }\p\O,
\end{equation}
where $g_{T,\f}:=\Phi_\f^* g_T$ denotes the pullback metric on $\ol{\O}$. We will look closer at the dependence of $g_{T,\f}$ on both $T$ and $\f$. 

For the purpose, we work in Euclidean coordinates $X=(X_1,\dots,X_{n+1})$, $X_{n+1}=t$, write $\p_i=\p_{X_i}$, and sum over repeated indices in $\{1,\dots,n+1\}$; sums over horizontal indices only are written explicitly. For a horizontal
vector field $h=h^k\p_k$ ($h^{n+1}\equiv0$) let $\Phi_h:=\id+h$ and $M_h:=D\Phi_h$, $(M_h)^k_i=\d^k_i+\p_ih^k$.
By Lemma~\ref{lem:metric}, if $\|h\|_{C^{2,\a}_0(\ol{\O})}<\e_0$, then for every $T>0$ the metric $g_{T,h}$ has components
\begin{equation}\label{eq:metric-formulas-main}
  (g_{T,h})_{ij}=(g_T)_{ij}+(a_h)_{ij},\qquad g_{T,h}^{ij}=p_h^{ij}+T^{-2}q_h^{ij},\qquad \sqrt{|g_{T,h}|}=T\,|M_h|,
\end{equation}
where $(a_h)_{ij}=\p_ih^j+\p_jh^i+\sum_{k=1}^n\p_ih^k\p_jh^k$,
$p_h^{ij}=\sum_{k=1}^n(M_h^{-1})^i_k(M_h^{-1})^j_k$ and $q_h^{ij}=\xi_h^i\xi_h^j$ with $\xi_h^i=(M_h^{-1})^i_{n+1}$. None of these depend on $T$.
For ease of notation, we will write $g_{T,\f}:=g_{T,h_\f}$, and similarly $M_\f$, $p_\f$, $q_\f$, $\xi_\f$. 

Fix
$\varrho\in(0,\varrho_0]$ so small that $C_1\varrho<\min(\e_0,\d_0/2)$, and set
\begin{equation}\label{eq:V}
  V:=\{\f\in\Y:\ \|\f\|_{\Y}<\varrho\}.
\end{equation}
For $\f\in V$, $\Phi_\f:(\ol{\O},g_{T,\f})\to(\ol{\O_\f},g_T)$ is an isometry by
Lemma~\ref{lem:deformation}. We look for a solution $\tilde U$ of \eqref{eq:FBP-pback} in $\O$ of the form
\[
  \tilde U=U_\f+v,\qquad U_\f:=\Phi_\f^*U=U\circ\Phi_\f .
\]
As $\D_{g_{T,\f}}U_\f=\Phi_\f^*(\D_{g_T}U)=0$ in $\O$, we are led to define, for $T>0$, $\f\in V$ and $v\in\X$,
\begin{equation}\label{eq:N-def}
  N(T,\f,v):=(N_1,N_2,N_3)=\Big(-\D_{g_{T,\f}}v,\ \ \tfrac12\big(1-|d(U_\f+v)|^2_{g_{T,\f}}\big)\Big|_\G,\ \ (U_\f+v)|_\G\Big).
\end{equation}
We see that $N(T,\f, v)=0$ when $\tilde{U}=U_\f+v$ solves \eqref{eq:FBP-pback}. In particular, $N(T,0,0)=0$. Conversely, if $N(T,\f, v)=0$ and $\tilde{U}> 0$ in $\O$, then $\tilde{U}$ solves \eqref{eq:FBP-pback}. The next lemma says that if $\f$ and $v$ are small in their respective norms, then $\tilde{U}$ is positive up to a unit distance from $\G$ inside $\O$.

\begin{lemma}\label{lem:zeros-to-solutions}
Let $T>0$, $\f\in V$ and $v\in\X$ with $N(T,\f,v)=0$, and let
$\tilde U:=U_\f+v$.
\begin{enumerate}
\item[(a)] If $\|\n v\|_{L^\infty(\O)}\leq 1/4$ and $\|\f\|_{L^\infty(\G)} \leq \theta/2$, then $\tilde{U}=U_\f + v >0$ in $\N_1\cap \O$. 

\item[(b)] Suppose that $\tilde U>0$ in $\O$. Set $\O^\sharp:=S_T(\O_\f)$ and
  \[
  \hat U:=(S_T^{-1})^*(\Phi_\f^{-1})^*\tilde U\ \text{ in }\O^\sharp,\qquad \hat U:=0\ \text{ in }\R^{n+1}\setminus\O^\sharp .
  \]
  Then $\hat U$ is a classical solution of \eqref{FBP} in $\R^{n+1}$ with positive phase $\O^\sharp$, and its  free boundary $S_T(\Phi_\f(\G))$ is of class $C^\infty$. If $\f$ and $v$ are $G(1)$-invariant, then $\hat U$   is $G(T)$-invariant; in particular, it is $T$-periodic and even in $t$.
\end{enumerate}
\end{lemma}

\begin{proof}
(a) We will abuse notation by writing $f(p,z)=f(X)$ for functions $f$ on $\N_1$, expressed through the Fermi coordinates $(p,z)\mapsto X=p-z\nu(p)\in \N_1$. 
Set $f_p(s):=U(p,s)$ for $s\in I = (-\d_0,4\th\r(p))$.
In this way, by Lemma \ref{lem:deformation}, 
$$U_\f(p,z) = U(p, \zeta_\f(p,z)) = f_p(\zeta_\f(p,z)), \quad \text{where }\z_\f(p,z):=z-\chi^2\big(\tfrac{z}{\th\r(p)}\big)\f(p).$$ 
Since $\tilde{U}=0$ on $\G$, we can write
\begin{align}
\tilde{U}(p, z) &= \tilde{U}(p,z) - \tilde{U}(p,0) = U_\f(p,z) + v(p,z)) - (U_\f(p,0) + v(p,0)) \notag \\
& = f_p(\zeta_\f(p,z))-f_p(\zeta_\f(p,0))+v(p,z)-v(p,0).\label{lem:zeros-to-solutions:eq1}
\end{align}
Noting that $f_p(0)=0$ and $f_p'(0)=-\p_\nu U(p)=1$, we get from the Hessian bound \eqref{eq:theta-small}, $\d_0\le\th$ and $\rho(p)\geq 1$ that
$|f_p'(s)-1|\le \frac{1}{16\theta \rho(p)}(\d_0 + 4\theta \rho(p)) \leq \frac12 \quad \text{for }s\in I,$ 
so that 
\[
f'_p(s)\geq \frac{1}{2} \quad \text{for all } s\in I. 
\]
Combining this estimate with $\|\n v\|_{L^\infty(\O)}\leq 1/4$, $\|\f\|_{L^\infty(\G)}\leq \theta/2$ and recalling that $z\to \zeta_\f(p,z)$ is increasing in $z$, \eqref{lem:zeros-to-solutions:eq1} readily yields
\begin{equation}\label{lem:zeros-to-solutions:eq2}
\tilde{U}(X) \geq \left(z - (1-\chi^2)|\f(p)| \right)/2 - z/4 \geq  z/4 - (1-\chi^2)\theta/4 \quad \text{for } X\in \N_1.
\end{equation}
For $z\in [0, 2\theta \rho(p)]$, $\chi=1$, so the right-hand side of \eqref{lem:zeros-to-solutions:eq2} is $\geq z/4>0$. For $z\in (2\theta \rho(p), 4\theta \rho(p)]$, the right-hand side can be bounded from below by $2\theta \rho(p)/4 - \theta/4 \geq \theta/4>0$  as $\rho(p)\geq 1$. We deduce that $\tilde{U}>0$ in $\N_1\cap \O$.

(b) Changing variables back by applying $\Phi_\f^{-1}$ and then $S_T^{-1}$, it is easy to see that $\hat{U}$ solves \eqref{FBP}. Its free boundary is $C^{2,\a}$ because
$\f\in C^{2,\a}_{\rm loc}(\G)$, and therefore $C^\infty$ (in fact analytic) by the higher regularity theory of the one-phase FBP (\cite{KindNiren}). If $\f$ and $v$ are
$G$-invariant, then $\tilde U$ and $\O_\f$ are $G$-invariant by Lemma~\ref{lem:deformation}(c). Finally, $S_T$ conjugates $t\mapsto t+1$ and $t\mapsto -t$ to $t\mapsto t+T$ and $t\mapsto-t$, so $\hat{U}(x,t)$ is even and $T$-periodic in $t$.
\end{proof}

\begin{rem}\label{rem:positivity}
Part (a) shows that the constructed solutions are automatically positive in the band $\N_1\cap \O$. Positivity in $\O\setminus \N_1$, which is needed in part (b), is verified at the end of the bifurcation argument in Section \ref{sec:bif}, where $v$ and $\f$ are small.
\end{rem}

The main result of this section is the following.

\begin{theorem}\label{thm:N}
Assume that we are in Case A or Case B, with $\X$, $\Y$, $\Z$ as in
\eqref{eq:spacesA}--\eqref{eq:spacesB} according to the case, and $V\subset\Y$ as in \eqref{eq:V}. Then
\begin{equation}\label{eq:N-affine}
  N(T,\f,v)=N^0(\f,v)+T^{-2}\, N^1(\f,v),\qquad (T,\f,v)\in(0,\infty)\times V\times\X,
\end{equation}
where $N^0:V\times\X\to\Z\times\Y$ and
$N^1:V\times\X\to\Z\times\{0\}$ are smooth maps, independent of $T$. In particular,
$N\in C^\infty((0,\infty)\times V\times\X;\,\Z\times\Y)$ and
$N(T,0,0)=(0,0,0)$ for all $T>0$. The partial derivatives at $(\f,v)=(0,0)$ are
\begin{align}
  D_v N(T,0,0)[w]&=\big(-(\D_x+T^{-2}\p_t^2)w, \p_\nu w,  w|_\G\big),\label{eq:DvN}\\
  D_\f N(T,0,0)[\phi]&=\big(0, -H\phi, -\phi\big).\label{eq:DphiN}
\end{align}
Finally, $N(T,\cdot,\cdot)$ is equivariant: for $\g\in G(\t)$, $\t>0$ arbitrary,
$N(T,\f\circ\g,v\circ\g)=N(T,\f,v)\circ\g$. In particular,
$N(T,\cdot,\cdot)$ maps $(V\cap\Y_{G(1)})\times\X_{G(1)}$ into $\Z_{G(1)}\times\Y_{G(1)}$.
\end{theorem}

\begin{rem}\label{rem:mu}
The neighbourhood $V$ does not depend on $T$, and $T$ enters \eqref{eq:N-affine} only through the scalar factor $T^{-2}$. In particular, the joint
regularity in $(T,v)$ required by the Crandall--Rabinowitz theorem, including the continuity of $D_TD_v$, is obvious.
\end{rem}

\subsection*{Proof of Theorem~\ref{thm:N}}\label{ss:proof-N}

The boundary components of $N$ involve $U$ and its derivatives at displaced boundary points; these are handled by Lemma~\ref{lem:composition} in Appendix~\ref{app:technical}, whose notation $\mathcal C^s_\Psi$ and classes $\mathcal{S}_\g$ we use below. For ease of reference, $\mathcal{C}^s_\Psi$ is the map on vector fields $h$ on $\G$:
\[
 \mathcal{C}^s_\Psi:\ h\mapsto\Psi\circ(\id+s h)=\Big(p\mapsto\Psi\big(p+s h(p)\big)\Big),\qquad p\in\G,
\]
where $\Psi$ will stand for smooth derivative quantities, such as the Hessian components $\Psi=\de^2_{kl} U$ or $\Psi=|\n U|^2-1$, which belong for some $\g\in \R$ to the class $\mathcal{S}_\g$, decaying as $\la x\ra^{-\g}$ in the cross-sectional variables and whose $k$-th order derivatives decay as $\la x\ra^{-\g-k}$.

We will apply Lemma~\ref{lem:composition} with $h=\f\nu$, $\f\in V$.
Note that $\Phi_\f|_\G=\id+\f\nu$, so that $\Psi\circ\Phi_\f|_\G=\mathcal C^1_\Psi(\f\nu)$.

\begin{lemma}[The interior equation]\label{lem:N1}
For $\f\in V$ and $v\in\X$,
\[
  N_1(T,\f,v)=-\mathcal L^0_\f v-T^{-2}\mathcal L^1_\f v,
\]
where, with $M=M_{h_\f}$, $p=p_{h_\f}$ and $q=q_{h_\f}$ as in Lemma~\ref{lem:metric},
\begin{align*}
  \mathcal L^0_\f v&:=\frac1{|M|}\,\p_i\big(|M|\,p^{ij}\,\p_jv\big)
      =p^{ij}\p^2_{ij}v+\big(\p_ip^{ij}+p^{ij}\,\p_i\log|M|\big)\p_jv,\\
  \mathcal L^1_\f v&:=\frac1{|M|}\,\p_i\big(|M|\,q^{ij}\,\p_jv\big)
      =q^{ij}\p^2_{ij}v+\big(\p_iq^{ij}+q^{ij}\,\p_i\log|M|\big)\p_jv .
\end{align*}
The maps $(\f,v)\mapsto\mathcal L^i_\f v$, $i=0,1$, are smooth from $V\times\X$ into
$C^{0,\a}_{\b+2}(\ol{\O})$ and linear in $v$, with $\mathcal L^0_0=\D_x$ and $\mathcal L^1_0=\p_t^2$. In particular,
\[
  D_\f N_1(T,0,0)=0,\qquad D_v N_1(T,0,0)[w]=-(\D_x+T^{-2}\p_t^2)w .
\]
\end{lemma}

\begin{proof}
By \eqref{eq:metric-formulas-main},
\[
  \D_{g_{T,\f}}v=\frac{1}{\sqrt{|g_{T,\f}|}}\,\p_i\Big(\sqrt{|g_{T,\f}|}\,g_{T,\f}^{ij}\,\p_jv\Big)
  =\frac1{|M|}\,\p_i\Big(|M|\,\big(p^{ij}+T^{-2}q^{ij}\big)\,\p_jv\Big),
\]
since the factor $T$ in $\sqrt{|g_{T,\f}|}=T|M|$ cancels. By Lemmas~\ref{lem:hphi} and \ref{lem:metric}, the maps
$\f\mapsto p-p_0$, $\f\mapsto q-q_0$ and $\f\mapsto \log|M|$ are smooth $V\to C^{1,\a}_1(\ol{\O})$. Hence the coefficients $p^{ij}$, $q^{ij}$
lie in $C^{0,\a}_0(\O)$, and $(\p_ip^{ij}+p^{ij}\p_i\log|M|)$ and $(\p_iq^{ij}+q^{ij}\p_i\log|M|)$ lie in
$C^{0,\a}_2(\O)$, smoothly in $\f$. For $v\in\X$ (in Case A, $v=c+v_0$ and only $v_0\in C^{2,\a}_\b(\O)$ is
differentiated), $\p^2_{ij}v\in C^{0,\a}_{\b+2}(\O)$ and $\p_jv\in C^{0,\a}_{\b+1}(\O)$. Therefore, by \eqref{eq:product}, $\mathcal L^i_\f v\in C^{0,\a}_{\b+2}(\O)$ depend bilinearly on $v$ and the coefficients of $\mathcal L^i_\f$. Since $|M_0|=1$,
$p_0^{ij}=\sum_{k\le n}\d^i_k\d^j_k$ and $q_0^{ij}=\d^i_{n+1}\d^j_{n+1}$, we get $\mathcal L^0_0=\D_x$ and
$\mathcal L^1_0=\p_t^2$. Finally, $N_1(T,\f,0)=0$ for all $\f$, so $D_\f N_1(T,0,0)=0$.
\end{proof}

\begin{lemma}[The Dirichlet condition]\label{lem:N3}
For $\f\in V$ and $v\in\X$,
\begin{equation}\label{eq:N3}
  N_3(\f,v)=v|_\G-\f+\f^2\,\mathcal V_\f,\qquad
  \mathcal V_\f:=\int_0^1(1-s)\,\mathcal C^s_{\p^2_{kl}U}(\f\nu)\,ds\;\nu^k\nu^l,
\end{equation}
and $N_3$ does not depend on $T$. The map $\f\mapsto\mathcal V_\f$ is smooth from $V$ into
$C^{2,\a}_1(\G)$, and $N_3$ is smooth from $V\times\X$ into $\Y$ with
\begin{equation}\label{eq:N3:ders}
  D_\f N_3(0,0)[\phi]=-\phi,\qquad D_v N_3(0,0)[w]=w|_\G .
\end{equation}
\end{lemma}

\begin{proof}
For $p\in\G$ we have $U_\f(p)=U(p+\f(p)\nu(p))$. Since $U(p)=0$ and $\p_\nu U(p)=-1$, Taylor's formula
yields \eqref{eq:N3}. No metric enters $N_3$, so there is no dependence on $T$. By \eqref{eq:classes} and Lemma~\ref{lem:composition}
(with $\g=1$, $h=\f\nu$ and $m(s)=1-s$), together with $\nu\in C^{2,\a}_0(\G)$, we obtain that the map
$\f\mapsto\mathcal V_\f\in C^{2,\a}_1(\G)$ is smooth. In Case A, $\f^2\in \X^{2,\a}_\b(\G)\emb C^{2,\a}_0(\G)$, so
$\f^2\mathcal V_\f\in C^{2,\a}_1(\G)\emb C^{2,\a}_\b(\G)\subset\Y$ because $\b\le1$.
In Case B, $\f^2\in C^{2,\a}_{2\b}(\G)$, so $\f^2\mathcal V_\f\in C^{2,\a}_{2\b+1}(\G)\emb\Y$.
Since $v\mapsto v|_\G$ is bounded $\X\to\Y$, $N_3$ is smooth. The derivatives follow
because the last term in \eqref{eq:N3} is quadratic in $\f$.
\end{proof}

\begin{lemma}[The Neumann condition]\label{lem:N2}
Let $E:=|\n U|^2-1$, and for $\f\in V$ write $M=M_{h_\f}$, $p=p_{h_\f}$, $\xi=\xi_{h_\f}$. Then, on $\G$,
\begin{equation}\label{eq:gradient-identity}
  |d(U_\f+v)|^2_{g_{T,\f}}=1+E\circ\Phi_\f+2\,\mathcal B_\f^j\,\p_jv+p^{ij}\,\p_iv\,\p_jv
      +T^{-2}\big(\xi^i\p_iv\big)^2 ,
\end{equation}
where $\mathcal B_\f^j:=\sum_{k=1}^n(M^{-1})^j_k\,(\p_kU\circ\Phi_\f).$
Consequently,
\begin{align}
N_2(T,\f,v)&= \frac{1}{2}\left(1 - |d (U_\f + v)|_{g_{T,\f}}^2\right) = N_2^0(\f,v)+T^{-2}N_2^1(\f,v), \quad \text{with} \notag \\
  N_2^0(\f,v)&=-H\f-\tfrac12\f^2\,\mathcal W_\f-\mathcal B_\f^j\,\p_jv-\tfrac12\,p^{ij}\,\p_iv\,\p_jv,\label{eq:N20}\\
  N_2^1(\f,v)&=-\tfrac12\big(\xi^i\p_iv\big)^2,\label{eq:N21}
\end{align}
all restricted to $\G$, where
\[
  \mathcal W_\f:=\int_0^1(1-s)\,\mathcal C^s_{\p^2_{kl}E}(\f\nu)\,ds\;\nu^k\nu^l .
\]
The maps $N_2^0,N_2^1$ are smooth from $V\times\X$ into $C^{1,\a}_{\b+1}(\G)$, and
\[
  D_\f N_2(T,0,0)[\phi]=-H\phi,\qquad D_v N_2(T,0,0)[w]=\p_\nu w .
\]
\end{lemma}

\begin{proof}
By the chain rule, $\p_iU_\f=(\p_kU\circ\Phi_\f)\,M^k_i$. Moreover, $g_{T,\f}=M^Tg_TM$ gives
$M^k_i\,g_{T,\f}^{ij}\,M^l_j=g_T^{kl}$ and $M^k_i\,g_{T,\f}^{ij}=g_T^{kl}(M^{-1})^j_l$. Since $\p_{n+1}U=0$ and
$g_T^{kl}=\d^{kl}$ for $k,l\le n$, we obtain
\begin{align*}
  g_{T,\f}^{ij}\,\p_iU_\f\,\p_jU_\f&=\sum_{k=1}^n(\p_kU\,\p_kU)\circ\Phi_\f=1+E\circ\Phi_\f,\\
  g_{T,\f}^{ij}\,\p_iU_\f\,\p_jv&=\sum_{k=1}^n(\p_kU\circ\Phi_\f)(M^{-1})^j_k\,\p_jv=\mathcal B_\f^j\,\p_jv,
\end{align*}
while $g_{T,\f}^{ij}\p_iv\,\p_jv=p^{ij}\p_iv\,\p_jv+T^{-2}(\xi^i\p_iv)^2$ by \eqref{eq:metric-formulas-main}. This proves
\eqref{eq:gradient-identity}. On $\G$ we have $\Phi_\f(p)=p+\f(p)\nu(p)$ and $E=0$, and
\[
  \p_\nu E=2\,\p^2_{kl}U\,\p_kU\,\nu^l=-2U_{\nu\nu}=2H .
\]
So Taylor's formula gives $E\circ\Phi_\f|_\G=2H\f+\f^2\mathcal W_\f$, and \eqref{eq:N20}--\eqref{eq:N21} follow.

\emph{Regularity.} By Lemma~\ref{lem:composition} (with $h=\f\nu$) and \eqref{eq:classes}, $\f\mapsto\mathcal W_\f$ is smooth
$V\to C^{2,\a}_2(\G)$, and $\f\mapsto\p_kU\circ\Phi_\f|_\G=\mathcal C^1_{\p_kU}(\f\nu)$ is smooth $V\to C^{2,\a}_0(\G)$.
By Lemmas~\ref{lem:hphi} and \ref{lem:metric}, $\f\mapsto M^{-1}|_\G$, $p|_\G$ and $\xi|_\G$ are smooth $V\to C^{1,\a}_0(\G)$.
We now check the weights term by term.
\begin{itemize}
\item $H\f$: in Case A, $H\in C^{1,\a}_2(\G)$ by \eqref{eq:H-A} and $\f\in C^{2,\a}_0(\G)$, so
  $H\f\in C^{1,\a}_2(\G)\emb C^{1,\a}_{\b+1}(\G)$ as $\b\le1$. In Case B,
  $H\in C^{1,\a}_1(\G)$ by \eqref{eq:H-B} and $\f\in C^{2,\a}_\b(\G)$, so $H\f\in C^{1,\a}_{\b+1}(\G)$.
\item $\f^2\mathcal W_\f$: it lies in $C^{1,\a}_2(\G)$ in Case A and in $C^{1,\a}_{2\b+2}(\G)$ in Case B,
  both contained in $C^{1,\a}_{\b+1}(\G)$.
\item $\mathcal B_\f^j\p_jv$, $p^{ij}\p_iv\,\p_jv$ and $N_2^1$: in both cases $\p_jv|_\G\in C^{1,\a}_{\b+1}(\G)$, and the coefficients lie in $C^{1,\a}_0(\G)$. So these terms lie in
  $C^{1,\a}_{\b+1}(\G)$ and in $C^{1,\a}_{2\b+2}(\G)\emb C^{1,\a}_{\b+1}(\G)$, respectively.
\end{itemize}
Each term is a continuous multilinear expression in $v$ with coefficients depending smoothly on $\f$, so
$N_2^0$ and $N_2^1$ are smooth. 

For the partial derivatives, note first that $N_2^1$ and the last term of
\eqref{eq:N20} are quadratic in $v$. The term $\mathcal B_\f^j\p_jv$ vanishes for $v=0$, and
$\mathcal B_0^j=\p_jU=-\nu^j$ on $\G$, whence $D_v(-\mathcal B_\f^j\p_jv)|_{(0,0)}[w]=\nu^j\p_jw=\p_\nu w$.
Finally, $-\tfrac12\f^2\mathcal W_\f$ is quadratic in $\f$, so $D_\f N_2(T,0,0)[\phi]=-H\phi$.
\end{proof}

\begin{proof}[Proof of Theorem~\ref{thm:N}]
The decomposition \eqref{eq:N-affine}, the smoothness of its coefficients and the formulas
\eqref{eq:DvN}--\eqref{eq:DphiN} are collected from Lemmas~\ref{lem:N1}, \ref{lem:N3} and \ref{lem:N2}, with $N^0:=(-\mathcal L^0_\f v,N_2^0,N_3)$ and $N^1:=(-\mathcal L^1_\f v,N_2^1,0)$. Joint smoothness in $(T,\f,v)$ follows immediately.

For the equivariance, let $\g\in G(\t)$. It is an isometry of $g_T$ that preserves $\O$, $\G$ and $U$, and
$\Phi_{\f\circ\g}=\g^{-1}\circ\Phi_\f\circ\g$ by Lemma~\ref{lem:deformation}(c). Hence
$g_{T,\f\circ\g}=\g^*g_{T,\f}$ and $U_{\f\circ\g}=U_\f\circ\g$, and each component of
$N$ is natural under isometries.
\end{proof}

\subsection{The reduced operator \texorpdfstring{$M(T,v)$}{M(T,v)}}\label{ss:reduction}

Since $D_\f N_3(0,0)=-\mathrm{id}_{\Y}$ is an isomorphism of $\Y$ and $N_3$ does not
depend on $T$, the Dirichlet condition determines the normal perturbation in terms of $v$,
simultaneously for all $T$.

\begin{prop}\label{prop:reduction}
Assume Case A or Case B, with $\X$, $\Y$, $\Z$ as in \eqref{eq:spacesA}--\eqref{eq:spacesB}.
There exist balls $\mathcal U\subset\X$ and $V'\subset V$ centred at $0$ and a smooth map
$F:\mathcal U\to V'$ with $F(0)=0$ and $DF(0)[w]=w|_\G$, such that for $(\f,v)\in V'\times\mathcal U$
\[
  N_3(\f,v)=0\iff\f=F(v).
\]
$F$ does not depend on $T$ and is equivariant: $F(v\circ\g)=F(v)\circ\g$ for $\g\in G(\t)$, $\t>0$.
Define, for $T>0$ and $v\in\mathcal U$,
\begin{equation}\label{eq:Mbar}
  M(T,v):=\Big(N_1\big(T,F(v),v\big), N_2\big(T,F(v),v\big)\Big).
\end{equation}
Then
\begin{enumerate}
\item[(a)] $M(T,v)=M^0(v)+T^{-2}M^1(v)$ with $M^0,M^1\in C^\infty(\mathcal U;\Z)$ independent of $T$.
  In particular $M\in C^\infty((0,\infty)\times\mathcal U;\Z)$, and $M(T,0)=0$ for all $T>0$.
\item[(b)] For every $T>0$,
  \begin{equation}\label{eq:DvMbar}
    D_vM(T,0)[w]=\big(-(\D_x+T^{-2}\p_t^2)w, \p_\nu w-Hw|_\G\big),\qquad w\in\X .
  \end{equation}
\item[(c)] The mixed second partial derivative
  \begin{equation}\label{eq:DTDvMbar}
    D_TD_vM(T,0)[w]=\big(2T^{-3}\p_t^2w,\ 0\big),\qquad w\in\X .
  \end{equation}
\item[(d)] For every $T>0$, $M(T,\cdot)$ maps $\mathcal U\cap\X_{G(1)}$ into $\Z_{G(1)}$.
\item[(e)] If $v\in\mathcal U$ satisfies $M(T,v)=0$, then $N(T,F(v),v)=0$, and
  Lemma~\ref{lem:zeros-to-solutions} applies to $(\f,v)=(F(v),v)$.
\end{enumerate}
\end{prop}

\begin{proof}
The existence, uniqueness and smoothness of $F$ follow from the Implicit Function Theorem applied to
$N_3:V\times\X\to\Y$ at $(0,0)$ (Lemma~\ref{lem:N3}). Differentiating
$N_3(F(v),v)=0$ at $v=0$ and using \eqref{eq:N3:ders} yields $$-DF(0)[w]+w|_\G=0.$$ Since the norms are $G(\t)$-invariant, the balls
$\mathcal U$ and $V'$ are $G(\t)$-invariant. The equivariance of $F$ then follows from that of $N_3$ and the uniqueness in the Implicit Function Theorem.

Statement (a) follows from Theorem~\ref{thm:N} and the chain rule, with
$$M^i(v):=(N^i_1(F(v),v),N^i_2(F(v),v)), \quad i=0,1.$$ 
For (b), we obtain by the chain rule, \eqref{eq:DvN}, \eqref{eq:DphiN}
and $DF(0)[w]=w|_\G$ that
\begin{align*}
  D_vM(T,0)[w] & =D_\f(N_1,N_2)(T,0,0)[w|_\G]+D_v(N_1,N_2)(T,0,0)[w] \\
   & =\big(-(\D_x+T^{-2}\p_t^2)w, -Hw+\p_\nu w\big). 
\end{align*}
Differentiating then in $T$, we obtain (c). Part (d) is the equivariance of $N$ and $F$, and (e) is immediate.
\end{proof}

\subsection{The linearized operator and the $T$-periodic problem}\label{ss:AT}

The Fredholm theory of Sections \ref{sec:linHP} and \ref{sec:lingen} will be formulated for the Euclidean operator $\mathcal{L}:=(-\D,\ \p_\nu-H)$ acting on $T$-periodic ($G(T)$-invariant) functions. It transfers to the linearized operator $D_vM(T,0)$ through conjugations by the pull-backs $S_T^*$ which, for fixed $T>0$, are bounded automorphisms on $\X$ and $\Z$, mapping $\X_{G(T)}$ onto $\X_{G(1)}$, and $\Z_{G(T)}$ onto $\Z_{G(1)}$.

\begin{coro}\label{cor:conjugation}
Let $\mathcal{L}:=(-\D,\ \p_\nu-H)$. For every $T>0$,
\[
  D_vM(T,0)=S_T^*\circ\mathcal L\circ(S_T^*)^{-1}\quad\text{as operators } \X_{G(1)}\to\Z_{G(1)},
\]
where $\mathcal L:\X_{G(T)}\to\Z_{G(T)}$. In particular, $D_vM(T,0):\X_{G(1)}\to\Z_{G(1)}$ is Fredholm of index $0$ whenever $\mathcal L:\X_{G(T)}\to\Z_{G(T)}$ is. 
\end{coro}

\begin{proof}
For $w\in\X_{G(1)}$, let $W:=(S_T^*)^{-1}w\in\X_{G(T)}$. Then $S_T^*(\p_t^2W)=T^{-2}\p_t^2w$, while $\D_x$, $\p_\nu$ and
multiplication by $H$ commute with $S_T^*$ because $\nu$ and $H$ do not depend on $t$. Hence
$S_T^*\mathcal LW=D_vM(T,0)w$ by \eqref{eq:DvMbar}.
\end{proof}

\section{Fredholm properties of the linearized operator in the case of the double hairpin solution}\label{sec:linHP}

In this section we will focus solely on the \textbf{Case A} of the double hairpin solution $u=HP$, $\o=\{HP>0\}$, in $n=2$ dimensions. Fix $T>0$ and denote  $G:=G_0\times G_1(T)$.  The goal is to show that the operator 
\begin{equation}\label{eq:Mlin}
\begin{aligned}
\mathcal{L} :\X_{G}=\mathcal{X}^{2,\a}_{\b,G}(\overline{\O}) &\to \Z_G=C^{0,\a}_{\b+2,G}(\overline{\O})\times C^{1,\a}_{\b+1,G}(\G),\\
\mathcal{L}w &:= (-\D w, \de_\nu w - H w|_\G),
\end{aligned}
\end{equation}
is a Fredholm operator of index $0$, for an appropriate choice of the symmetry group $G_0$ and the decay rate $\b\in (0,1]$. For the remainder of this section we will assume that $G_0\subset O(2)$ leaves $HP$ invariant and \emph{contains the reflection} $x_1\to - x_1$.

\begin{theorem}\label{thm:Fredhp} Let $T>0$, $\b\in (0,1)$ and $G$ be as above. Then the linearized operator $\mathcal{L}$ from \eqref{eq:Mlin} is a Fredholm operator of index $0$, i.e.\ its kernel and cokernel have the same finite dimension.
\end{theorem}
\begin{proof}
Let us decompose the formula for $\mathcal{L}$ into 
\begin{equation}\label{eq:decomp}
\mathcal{L} = L - K, \quad \text{where} \quad L w:= (-\D w, \de_\nu w)  \quad\text{and}\quad  Kw = (0,Hw), 
\end{equation}
the Neumann problem operator part $L$ and the perturbation $K$. To establish the theorem, it suffices to prove that, when the decay rate $\b\in(0,1)$, $L$ is a Fredholm operator of index $0$, while $K$ is a compact operator. This will be achieved in Theorem \ref{thm:Fred} and Lemma \ref{lem:compact} below, respectively.
\end{proof}

\begin{lemma}\label{lem:compact} Let $\b\in (0,1]$. 
The linear operator $K:\mathcal{X}^{2,\a}_{\b,G}(\overline{\O}) \to C^{0,\a}_{\b+2,G}(\overline{\O})\times C^{1,\a}_{\b+1,G}(\G)$ is compact.
\end{lemma}
\begin{proof}
First, we note that the restriction operator $v\to v|_\G$ is compact from $C^{2,\a}_{\b,G}(\overline{\O})$ into $C^{1,\a}_{0,G}(\G)$, as $\b>0$. Hence, taking a bounded sequence $\{w_k = c_k + v_k\}$ in $\mathcal{X}^{2,\a}_{\b,G}(\ol{\O})$, we can extract a subsequence (relabelling indices) such that $c_k\to c$ in $\R$ and  $v_k|_\G \to \f$ in $C^{1,\a}_{0, G}(\G)$, for some $c \in \R$ and $\f\in C^{1,\a}_{\b, G}(\G)$. Since $H\in C^{1,\a}_{2,G}(\G)$ because of \eqref{eq:H-A}, we see that $H c_k \to Hc$ in $C^{1,\a}_{2, G}(\G)$ and $H v_k|_\G \to H\f$ in $C^{1,\a}_{2, G}(\G)$. As $\b+1\leq 2$ for $\b\in (0,1]$, we deduce that
\begin{equation*}
H w_k = H c_k + H v_k|_\G \to H(c+\f) \quad \text{in } C^{1,\a}_{\b+1, G}(\G),
\end{equation*}
concluding the proof that $K$ is compact. 
\end{proof}

We will now address the Fredholm property of $L$. For the purpose, denote by 
$$
\O_\text{per} := \o \times (-T/2, T/2), \qquad \G_\text{per}:= \G\times (-T/2, T/2),
$$
the periodicity cells for $\O$ and $\G$, associated to the $T$-translation in the axial variable. Denote further the truncations $$\O_{\text{per},R}:=\{(x,t)\in \O_\text{per}: |x|<R\} \quad \text{and} \quad \G_{\text{per},R}=\{(x,t)\in \G_\text{per}: |x|<R\}.$$ We remark that a function $f\in C^{0,\a}_{\b+2}(\overline{\O})$ belongs to $L^1(\O_\text{per})$, since $|f(x,t)|\lesssim|x|^{-2-\b}$ while the Lebesgue measure $|\O_{\text{per}, R}| \lesssim T R^2$, grows (at most) quadratically in $R$. Similarly, $g\in C^{1,\a}_{\b+1}(\G)$ belongs to $L^1(\G_\text{per}, d\s)$, since $|g(x,t)|\lesssim |x|^{-1-\b}$, while the Euclidean area $\sigma(\G_{\text{per},R})\lesssim T R$ on account of Proposition \ref{prop:areagrowth}. Hence, 
\begin{equation}\label{eq:constraint}
l(f,g):=\int_{\O_\text{per}} f\,dX + \int_{\G_\text{per}} g\, dS
\end{equation}
defines a bounded linear functional on $C^{0,\a}_{\b+2,G}(\overline{\O})\times C^{1,\a}_{\b+1,G}(\G)$. This section's principal result states:

\begin{theorem}\label{thm:Fred} Let $\b\in (0,1)$ and $G$ be as in Theorem \ref{thm:Fredhp}. The operator $$L = (-\D, \de_\nu):  C^{2,\a}_{\b, G}(\overline{\O}) \oplus \R \to C^{0,\a}_{\b+2,G}(\overline{\O})\times C^{1,\a}_{\b+1,G}(\G)$$
has a one-dimensional kernel spanned by the constants $\R$ and its image 
$$
\text{im}\, L = \text{ker }l:= \{(f,g)\in C^{0,\a}_{\b+2,G}(\overline{\O})\times C^{1,\a}_{\b+1,G}(\G): l(f,g)=0\}
$$
is closed of codimension $1$. As a consequence, $L$ is a Fredholm operator of index $0$. 
\end{theorem} 
\begin{proof}
Clearly, constant functions belong to the kernel of $L$. That  $\text{im}\, L\subset \text{ker }l$ is a consequence of the Divergence theorem. Indeed, applying it to $-\D w$  in the expanding domains $\O_{\text{per},R}$, where $w=c+v\in \mathcal{X}^{2,\a}_{\b,G}(\overline{\O})$, we get due to periodicity:
\begin{equation}\label{thm:Fred:eq:div}
\int_{\O_{\text{per},R}} -\D (c+v) \,dX = -\int_{\de\O_{\text{per},R}} \de_{\nu} v \, dS  = - \int_{\G_{\text{per},R}} \de_{\nu} v \, dS - \int_{S_R}  \de_{\nu} v \, dS,
\end{equation}
where $S_R:=\{(x,t)\in \O_{\text{per}}: |x|=R\}$. Since $\int_{S_R}  \de_{\nu}  v\, d\s \lesssim R\times R^{-1-\b} = R^{-\b}\to 0$ as $R\to \infty$, and $(-\D v, \de_\nu v)\in C^{0,\a}_{\b+2}(\ol{\O})\times C^{1,\a}_{\b+1}(\G) \subset L^1(\O_\text{per})\times L^1(\G_\text{per}),$ we obtain that 
\[
\int_{\O_{\text{per}}} -\D (c+v) \,dX + \int_{\G_{\text{per}}} \de_{\nu} (c + v) \, dS = 0,
\]
by taking the limit $R\to \infty$ in \eqref{thm:Fred:eq:div} above. Since $l$ is a bounded linear functional, $\text{ker } l$ is a closed subspace of $C^{0,\a}_{\b+2,G}(\overline{\O})\times C^{1,\a}_{\b+1,G}(\G)$ of codimension $1$. 

Therefore, we will establish the statement of the theorem once we prove that the restricted operator $\bar{L}:=L|_{C^{2,\a}_{\b, G}(\overline{\O})}$ is a topological isomorphism onto $\text{ker }l$. Since $\bar{L}$ is a bounded linear operator, the Open Mapping Theorem tells us that we only need to show that $\bar{L}$ is both injective and surjective. This will be accomplished in Propositions \ref{prop:inj} (injectivity) and \ref{prop:surj} (surjectivity) below. 
\end{proof}

\begin{prop}\label{prop:inj} Let $\b>0$. If $w(x,t)\in C^{2,\a}_{\b}(\ol{\O})$ is $T$-periodic in the axial variable $t$ and solves
\begin{equation}\label{prop:inj:eq}
\left\{
\begin{aligned}
-\Delta w &= 0 && \text{in } \Omega,\\
\de_\nu w &= 0 && \text{on } \G,
\end{aligned}
\right.
\end{equation}
then  $w\equiv 0$. In particular, the linear operator $\bar{L}:=(-\D, \de_\nu): C^{2,\a}_{\b, G}(\overline{\O})\to \text{ker }l$ is injective.
\end{prop}
\begin{proof}
Assume that $w(x,t)\in C^{2,\a}_{\b}(\overline{\O})$ is a nontrivial solution to \eqref{prop:inj:eq} that is $T$-periodic in $t$. Without loss of generality, we may assume that it attains a positive value somewhere in $\ol{\O}$. Since $w(x,t)$ is continuous in $\overline{\O}$, $T$-periodic in $t$ and $w(x,t)\lesssim |x|^{-\b} \to 0$ as $x\to \infty$, then $w$ attains a positive maximum over $\overline{\O}$, say at a point $X_0\in \overline{\O}$. If $X_0\in \O$, then by the strict maximum principle, $w(X) = w(X_0)>0$ for all $X\in \O$, which contradicts the decay of $w(x,t)$ as $x\to \infty$. If $X_0\in \G$, then the Hopf Lemma implies that $\de_\nu w(X_0)>0$, which violates the zero Neumann condition for $w$. Hence, it must be that $w\equiv 0$. 
\end{proof}

The proof of the surjectivity of the operator $\bar{L}$, i.e.\ that the inhomogeneous problem
\begin{equation}\label{eq:inhom}
\left\{
\begin{aligned}
-\D w &= f_1 && \text{in } \Omega,\\
\de_\nu w &= g_1 && \text{on } \G,
\end{aligned}
\right.
\end{equation}
is solvable within $C^{2,\a}_\b(\ol{\O})$ for every choice of data $(f_1,g_1)\in \text{ker }l$ and $\b\in(0,1)$, is more intricate. First, we will decompose each $(f_1,g_1)\in \text{ker }l$ into a part $(\bar{f}, \bar{g})$, that is constant in $t$, and a remainder that is mean zero in $t$:
\begin{equation*}
(f_1,g_1) = (\bar{f}, \bar{g}) + (f_0, g_0),
\end{equation*}
where $\bar{f}(x,t)=f(x)$ and $\bar{g}(p,t)=g(p)$ are effectively functions, defined on $\o$ and $\de\o$, respectively: 
\begin{equation*}
f(x):=\frac{1}{T}\int_{-T/2}^{T/2} f_1(x,t)\, dt  \quad\text{for } x\in \o \quad \text{and}\quad g(p) := \frac{1}{T}\int_{-T/2}^{T/2} g_1(p,t)\,dt \quad\text{for } p\in \de\o.
\end{equation*}
We will then separately explore the solvability of \eqref{eq:inhom} with a right hand side $(\bar{f},\bar{g})$ and a right-hand side $(f_0, g_0)$. Note that since $\bar{f}(x,t), \bar{g}(x,t)$ are constant in $t$, the Neumann problem in the former case reduces to one on the double hairpin domain $\o$.   

\subsection{The solvability of the Neumann problem on the double hairpin domain.}

Recall that the symmetry group is $G=G_0\times G_1(T)$. Taking averages of $(f_1,g_1)\in \text{ker }l$, we clearly get $(f,g)\in \bar{C}^{0,\a}_{\b+2, G_0}(\overline{\o})\times \bar{C}^{1,\a}_{\b+1, G_0}(\de\o)$. Moreover, by Fubini's theorem the pair $(f,g)$ satisfies the natural compatibility condition
\begin{equation}\label{eq:compatible}
\int_\o f \, dx + \int_{\de\o} g\, d\s = \frac{1}{T}\left( \int_{\O_\text{per}} f_1 \,dX + \int_{\G_\text{per}} g_1\,dS \right) = 0.
\end{equation}

We would like to show that the Neumann boundary value problem on $\o \subset \R^2$,
\begin{equation}\label{eq:inhomhp}
\left\{
\begin{aligned}
-\Delta w &= f && \text{in } \o,\\
\de_\nu w &= g && \text{on } \de\o,
\end{aligned}
\right.
\end{equation}
is solvable in $\bar{C}^{2,\a}_{\b,G_0}(\o)$ for any data $(f,g)\in \bar{C}^{0,\a}_{\b+2, G_0}(\overline{\o})\times \bar{C}^{1,\a}_{\b+1, G_0}(\de\o)$, satisfying \eqref{eq:compatible}, for decay rates $\b \in (0,1)$. Namely, we will establish the following theorem.

\begin{theorem}\label{thm:solv} Let $\b\in(0,1)$. Assume that $(f,g)\in \bar{C}^{0, \a}_{\b+2, G_0}(\overline{\o})\times \bar{C}^{1,\a}_{\b+1, G_0}(\de\o)$ satisfies the compatibility condition \eqref{eq:compatible}. Then the problem \eqref{eq:inhomhp} has a unique solution $w\in \bar{C}^{2,\a}_{\b, G_0}(\overline{\o})$.
\end{theorem}


We will find a solution to \eqref{eq:inhomhp} of the form $w=w_1+w_2$, where $w_1$ solves the Poisson equation with homogeneous Dirichlet conditions
\begin{equation}\label{eq:homDir}
\left\{
\begin{aligned}
-\Delta w_1 &= f && \text{in } \o,\\
w_1 &= 0 && \text{on } \de\o,
\end{aligned}
\right.
\end{equation}
whereas $w_2$ solves the Neumann problem for the Laplace equation:
\begin{equation}\label{eq:harmNeu}
\left\{
\begin{aligned}
-\Delta w_2 &= 0 && \text{in } \o,\\
\de_\nu w_2 &= g_2:=g - \de_\nu w_1 && \text{on } \de\o.
\end{aligned}
\right.
\end{equation}

We first treat the solvability of the Dirichlet problem \eqref{eq:homDir}.

\begin{prop}\label{prop:Dir} Let  $\b\in (0,1)$ and $f\in \bar{C}^{0,\a}_{\b+2, G_0}(\overline{\o})$. Then there exists a unique solution $w_1\in \bar{C}^{2,\a}_{\b}(\overline{\o})$ of the  Dirichlet problem \eqref{eq:homDir}. Moreover, $w_1\in \bar{C}^{2,\a}_{\b, G_0}(\overline{\o})$. 
\end{prop}
\begin{proof}
Uniqueness follows from the maximum principle and the fact that $w_1(x)\to 0$ when $|x|\to \infty$, since the decay rate $\b>0$. As a result, the $G_0$-invariance of $f$ implies the $G_0$-invariance of $w_1$. 

To show the existence of a solution, consider the solutions $w_{1,n}$, $n\in \mathbb{N}$, of the Dirichlet problem
\begin{equation*}
\left\{
\begin{aligned}
-\Delta w_{1,n} &= f && \text{in } \o_{R_n},\\
w_{1,n} &= 0 && \text{on } \de\o_{R_n},
\end{aligned}
\right.
\end{equation*}
in the truncated domains $\o_{R}:=\o\cap \{|x|<R\}$, as they expand $R_n\nearrow \infty$ to all of $\o$. We claim that 
\begin{equation}\label{prop:Dir:eq:claim}
\|w_{1,n}\|_{\bar{C}_\b(\o_{R_n})}\leq C \|f\|_{\bar{C}^{0,\a}_{\b+2}(\o_{R_n})} 
\end{equation}
for a constant $C>0$ independent of $n$. Assume not; then there is a subsequence (call it $R_n$ again) such that 
\[
\|w_{1,n}\|_{\bar{C}_\b(\o_{R_n})}\geq n \|f\|_{\bar{C}^{0,\a}_{\b+2}(\o_{R_n})}.
\]
Normalizing $\bar{f}_n:= f/\|w_{1,n}\|_{\bar{C}_\b(\o_{R_n})}$, $\bar{v}_{n}:=w_{1,n}/\|w_{1,n}\|_{\bar{C}_\b(\o_{R_n})}$, we get a solution $\bar{v}_n$ of
\begin{equation*}
\left\{
\begin{aligned}
-\Delta \bar{v}_{n} &= \bar{f}_n && \text{in } \o_{R_n},\\
\bar{v}_{n} &= 0 && \text{on } \de\o_{R_n},
\end{aligned}
\right.
\end{equation*}
with $\| \bar{f}_n\|_{\bar{C}^{0,\a}_{\b+2}(\o)}\leq 1/n$, so that by the a priori local boundary Schauder estimate of Lemma \ref{lem:Dir-apriori}, we have 
\[
\|\bar{v}_{n}\|_{\bar{C}^{2,\a}_\b(\o_{R_n/2})}  \leq C \left(  \|\bar{f}_n\|_{\bar{C}^{0,\a}_{\b+2}(\o_{R_n})}  +  \|\bar{v}_n\|_{\bar{C}_\b(\o_{R_n})} \right) \leq c(1/n+1)\leq 2C.
\]
Up to taking another subsequence (and relabeling indices), we see that $(\bar{v}_n)_n$ converges in $C^2_\text{loc}$ to a $\bar{C}^{2,\a}_\b(\overline{\o})$-function $\bar{v}$, which solves 
\begin{equation*}
\left\{
\begin{aligned}
-\Delta \bar{v} &= 0  && \text{in } \o,\\
\bar{v} &= 0 && \text{on } \de\o.
\end{aligned}
\right.
\end{equation*}
By the uniqueness of solutions, it must be that $\bar{v}\equiv 0$. 

Now let $h$ be the positive supersolution for the Dirichlet problem, provided by Lemma \ref{lem:supersol} (to be proved below), for which
\[
 0<h(x)\leq \frac{c}{|x|^{\b}}, \quad -\D h(x) \geq \frac{1}{|x|^{\b+2}}, \quad \text{for } |x|>r_0.
\]
Given that $\bar{v}_n\to \bar{v}\equiv 0$ uniformly on compacts, $s_n:=\sup_{\omega_{r_0}} |\bar{v}_n|\to 0$. Therefore, for $\mu_n:=\max(s_n/\min_{\ol{\o}\cap \de B_{r_0}} h, 1/n) \to 0$, we have
\begin{align*}
\pm \bar{v}_n + \mu_n h & \geq 0 \quad \text{on}  \quad \de (\o_{R_n}\setminus \overline{\o_{r_0}}), \\
-\D (\pm \bar{v}_n + \mu_n h) & \geq 0 \quad \text{in} \quad \o_{R_n}\setminus \overline{\o_{r_0}},
\end{align*}
so that the maximum principle yields
\[
|\bar{v}_n|\leq \mu_n h \quad \text{in } \o_{R_n}\setminus \overline{\o_{r_0}}. 
\]
Combining this with the fact that $\|\bar{v}_n\|_{\bar{C}_\b(\o_{r_0})} \leq s_n r_0^\b$, we conclude that 
\[
\|\bar{v}_n\|_{\bar{C}_\b(\o_{R_n})}\leq \max\left(s_n r_0^\b, \mu_n \|h\|_{\bar{C}_\b(\o\setminus \overline{\o_{r_0}})} \right) \to 0 \quad \text{as } n\to \infty,
\]
which contradicts the fact that $\|\bar{v}_n\|_{\bar{C}_\b(\o_{R_n})}=1$ and completes the proof of \eqref{prop:Dir:eq:claim}. As a result, the a priori estimate \eqref{lem:Dir-apriori:eq:est} gives us 
\[
\|w_{1,n}\|_{\bar{C}^{2,\a}_\b(\o_{R_n/2})}  \leq C (\|f\|_{\bar{C}^{0,\a}_{\b+2}(\o_{R_n})} + \|w_{1,n}\|_{\bar{C}_\b(\o_{R_n})})  \leq  \bar{C} \|f\|_{\bar{C}^{0,\a}_{\b+2}(\o)},
\]
for some constant $\bar{C}>0$, independent of of $n$. Now a standard compactness argument produces a global solution $w_1\in \bar{C}^{2,\a}_\b(\overline{\o})$ of \eqref{eq:homDir}. 
\end{proof}

In the next lemma, we construct the necessary barrier for the proof of Proposition \ref{prop:Dir} above.
\begin{lemma}\label{lem:supersol} Let $\b\in (0,1)$. There exist $r_0>0$, $c>0$ and a function $h\in C^2(\overline{\o}\cap \{|x|\geq r_0\})$ such that
\[
 0<h(x)\leq \frac{c}{|x|^{\b}}, \quad -\D h(x) \geq \frac{1}{|x|^{\b+2}}, \quad  \text{for all } x\in \overline{\o}\cap \{|x|\geq r_0\}.
\] 
\end{lemma}

\begin{proof}
Identifying $\R^2\simeq \C$, let us first change variables $x=x_1 + i x_2 = F(y_1+i y_2)$ via the conformal diffeomorphism \eqref{HHPconfdifeo}
\[
F(y)=y+\sinh(y), \quad F: \mathcal{I}:=\{y \in \C: |y_2|<\pi/2\} \to \o.
\] 
We note that $|x| = |F(y)|\sim e^{|y_1|}$ for large $x\in \o\subset \C$. Denote by $h:\omega \to \R$ the function that we are looking for and let $\bar{h}(y):=h(F(y))$ for $y\in \mathcal{I}$. Since $F$ is conformal, we have
$$-\D_x h  = -\D_y \bar{h} |F'(y)|^{-2} = -\D_y  \bar{h} |1+ \cosh(y)|^{-2} \sim  (-\D_y \bar{h}) e^{-2|y_1|} \sim  (-\D_y \bar{h})|x|^{-2}, $$
for all large $|x|$, so it suffices to find a function $\bar{h}:\overline{\mathcal{I}}\to \R$ such that 
\begin{equation}\label{lem:supersol:eq:main}
0<\bar{h}\leq e^{-\b|y_1|} \quad \text{and} \quad -\D_y\bar{h}\geq \bar{c} e^{-\b |y_1|} \quad \text{for all } y\in \ol{\mathcal{I}} \text{ with } |y_1|\geq 1,
\end{equation}
for some constant $\bar{c}>0$. Take $\b_0\in (\b,1)$ and let us verify that $$\bar{h}(y_1,y_2)=e^{-\b|y_1|}\cos(\b_0 y_2)$$ is precisely such a function. Since $\b_0\in (0,1)$,  $c_1:=\min_{|y_2|\leq \pi/2} \cos(\b_0 y_2) >0$, so that $$c_1 e^{-\b |y_1|}\leq\bar{h}\leq e^{-\b|y_1|} \quad \text{on } \overline{\mathcal{I}}.$$ Now, for $\bar{c} := c_1  (\b_0^2 - \b^2)>0$,
\[
-\D_y \bar{h} = (\b_0^2 - \b^2) \bar{h} \geq c_1  (\b_0^2 - \b^2) e^{-\b|y_1|} = \bar{c} e^{-\b |y_1|} \quad \text{for all } y \in \overline{\mathcal{I}} \text{ with } y_1\neq 0. 
\]
Multiplying $h=\bar{h}\circ F^{-1}$ by an appropriate positive constant, if necessary, we can ensure that both inequalities in \eqref{lem:supersol:eq:main} are met. 
\end{proof}

We now proceed to address the Neumann problem \eqref{eq:harmNeu}: 
\begin{equation*}
\left\{
\begin{aligned}
-\Delta w_2 &= 0 && \text{in } \o,\\
\de_\nu w_2 &= g_2:=g - \de_\nu w_1 && \text{on } \de\o,
\end{aligned}
\right.
\end{equation*}
with $g_2=g - \de_\nu w_1\in \bar{C}^{1,\a}_{\b+1, G_0}(\de\o)$, where $w_1\in \bar{C}^{2,\a}_{\b, G_0}(\overline{\o})$ solves the Dirichlet problem \eqref{eq:homDir}. Given the decay of $|g|+ |\de_\nu w_1|\lesssim |x|^{-\b-1}$ and $|f|=|\D w_1|\lesssim |x|^{-\b-2}$ at infinity, we know $g, \de_\nu w_1 \in L^1(\de\o, d\s)$ and $f\in L^1(\o)$, so an application of the Divergence theorem yields
\begin{equation}\label{eq:compat1}
\int_{\de\o} g_2 \, d\s =  \int_{\de\o} g \, d\s - \int_{\de\o} \de_\nu w_1 \, d\s  =  \int_{\de\o} g \, d\s +  \int_{\o} (-\D w_1)  \, dx = \int_{\de\o} g \, d\s +  \int_{\o} f  \, dx  = 0,
\end{equation}
by the compatibility condition \eqref{eq:compatible}. We will now prove the following key proposition.

\begin{prop}\label{prop:keyN} Let $\b \in (0,1)$ and suppose that $g_2\in \bar{C}^{1,\a}_{\b+1, G_0}(\de\o)$ satisfies \eqref{eq:compat1}. Then the Neumann problem \eqref{eq:harmNeu} has a solution $w_2\in \bar{C}^{2,\a}_{\b}(\ol{\o})$.
\end{prop}
\begin{proof}
As  $g_2\in \bar{C}^{1,\a}_{\b+1}(\de\o)$, Lemma \ref{lem:Neum-apriori} suggests that it is enough to prove the existence of a classical solution $w_2\in C^2(\o)\cap C^1_\text{loc}(\ol{\o})$ that has the correct decay rate at infinity:
\[
|w_2(x)| \lesssim |x|^{-\b} \quad \text{as } |x|\to \infty \quad \text{in }\o.
\]
For the purpose, we will once again change variables $x=x_1 + i x_2 = F(y_1+i y_2)$ via the conformal diffeomorphism $F$ from \eqref{HHPconfdifeo}, mapping the problem back to the strip $\mathcal{I}$. Setting $$v(y):=w_2(F(y)) \quad \text{for  } y\in \ol{\mathcal{I}} \quad\text{and} \quad g_*(p):=g_2(F(p))|F'(p)| \quad \text{for } p\in \de\mathcal{I},$$ we are looking for a solution $v\in  C^2(\I)\cap C^1_\text{loc}(\ol{\I})$ of 
\begin{equation}\label{prop:keyN:eq:pullback}
\left\{
\begin{aligned}
-\Delta_y v &= 0 && \text{in } \I,\\
\de_\nu v &= g_* && \text{on } \de\I,
\end{aligned}
\right.
\end{equation}
such that 
\begin{equation*}
|v(y)|\lesssim |x|^{-\b} \sim e^{-\b|y_1|} \quad \text{as } y=(y_1, y_2)\to \infty \quad \text{in } \mathcal{I}.
\end{equation*}
 Note that since $|F'(y)|\sim e^{|y_1|}$ and $|g_2(F(y))|\lesssim |F(y)|^{-(\b + 1)} \lesssim e^{-(\b + 1)|y_1|}$, as $y\to \infty$ in $\de\I,$ we have
\begin{equation}\label{prop:keyN:eq:decay}
|g_*(y)|  = |g_2(F(y))||F'(y)| \lesssim e^{-\b|y_1|}. 
\end{equation}
Moreover, since  $g_2$ is invariant under $x_1\to -x_1$ by the assumed $G_0$-invariance, we know that $g_*(y_1,y_2)$ is even in $y_1$. Changing variable in the boundary integral, $ds(x) = |F'(y)|dy_1$,  the compatibility condition \eqref{eq:compat1} takes the form
\begin{equation}\label{prop:keyN:eq:mean}
\int_{\de\I} g_* \, dy_1 = \int_{\de\o} g_2 \, ds = 0.
\end{equation}

To construct a solution of \eqref{prop:keyN:eq:pullback}, we will employ the Green's function for the Neumann Laplacian in $\I$ (\cite{NeumannStrip}):
\begin{equation}\label{prop:keyN:eq:Green}
G_N(y,y')=-\frac{1}{4\pi} \left( \ln(\cosh(y_1-y_1')-\cos(y_2-y_2')) + \ln(\cosh(y_1-y_1')+\cos(y_2+y_2'))\right).
\end{equation}
Note that for $y'=(y_1', \pm \pi/2) \in \de\I$ and $\xi:=y_1-y_1'$, the formula \eqref{prop:keyN:eq:Green} simplifies to 
\[
G_N(y,(y_1',\pm\pi/2))= - \frac{1}{2\pi} \ln\left(\cosh \xi \mp \sin(y_2)\right),
\]
and we can get the following linear growth estimate
\begin{align}\label{prop:keyN:eq:Greendefect}
\left|G_N(y,(y_1',\pm\pi/2)) - \frac{\ln 2 - |\xi|}{2\pi}\right| &= \frac{1}{2\pi} \left|\ln \left(1+e^{-2|\xi|} \mp 2 \sin (y_2) e^{-|\xi|}\right) \right| \notag \\
& \leq C e^{-|\xi|} \quad \text{whenever } |\xi|\geq 1,
\end{align}
for some absolute constant $C>0$. Given that $g_*(y')$ has exponential decay \eqref{prop:keyN:eq:decay}, while the Green's function on the boundary $G_N(y,(\xi,\pm \pi/2)$ grows linearly as $|y_1'|\to\infty$, the function 
\begin{equation}\label{prop:keyN:eq:sol}
v(y_1,y_2):=\int_{\de\I} G_N(y,y') g_*(y') \, dy_1' \quad \text{for } y=(y_1,y_2)\in \I,
\end{equation}
is well defined. It is a standard computation to check that $v\in C^2(\I)\cap C^1_\text{loc}(\ol{\I})$ and solves \eqref{prop:keyN:eq:pullback}. 

We claim that 
\begin{equation}\label{prop:keyN:eq:claim}
|v(y_1,y_2)|\leq C(g_*, \b) e^{-\b |y_1|} \quad \text{for all } y=(y_1,y_2)\in \I \quad \text{with } |y_1|\geq 1,
\end{equation}
which will imply the desired decay estimate $|w_2(x)|\lesssim |x|^{-\b}$, given that $|x|=|F(y)|\sim e^{|y_1|}$. The evenness of $g_*(y_1, \pm\pi/2)$ in $y_1$, implies that the function $v$, defined by \eqref{prop:keyN:eq:sol}, will also be even in $y_1$. Hence, it suffices to check that $|v(y_1,y_2)|\leq C_2 e^{-\b y_1}$ for all $y_1\geq 1$. 

Set 
\[
R_N(y,y'):= G_N(y,y')- \frac{\ln 2 - (y_1-y_1')}{2\pi} \quad \text{for } (y,y')\in \I\times \de\I.
\]
and note that because of \eqref{prop:keyN:eq:Greendefect}, we have 
\begin{equation}\label{prop:keyN:eq:Rest}
|R_N(y,y')|\leq \begin{cases}
C e^{-\xi} & \text{when } \xi=y_1-y_1' > 1, \\
C (1+ |\xi| + |\log |\xi||) & \text{when } \xi = y_1-y_1' \leq 1,\\
\end{cases}
\quad \text{for } (y,y')\in \I\times \de\I.
\end{equation}
Since $g_*(y')$ is even with respect to $y_1'$ and has integral $0$ over $\de\I$ by \eqref{prop:keyN:eq:mean}, we can write
\begin{align}\label{prop:keyN:eq:vexp}
v(y) &= \int_{\de\I} G_N(y,y') g_*(y') \,dy_1' - \int_{\de\I} \frac{\ln 2 - (y_1-y_1')}{2\pi} g_*(y') \, dy_1' \\
& = \int_{\de\I} R_N(y,y') g_*(y')\,dy_1' = A + B, \quad \text{where} \notag \\
A & := \int_{\de\I\cap \{y_1'<y_1-1\}} R_N(y,y') g_*(y')\,dy_1'  \quad \text{and}\quad  B:=\int_{\de\I \cap \{y_1'\geq y_1-1\}} R_N(y,y') g_*(y')\,dy_1'. \notag
\end{align}
Now, let us bound  $A$ and $B$ separately. Because of \eqref{prop:keyN:eq:Rest} and the exponential decay \eqref{prop:keyN:eq:decay} of $g_*$, we obtain
\begin{equation}\label{prop:keyN:eq:A}
\begin{aligned}
|A| & \leq 2 \int_{-\infty}^{y_1-1} e^{-(y_1-y_1')} C e^{-\b |y_1'|}  \, dy_1' \leq  2C e^{-y_1} \int_{-\infty}^{y_1} e^{y_1'-\b|y_1'|}  \, dy_1'  \\
& = 2C e^{-y_1}\left(\int_{-\infty}^0 e^{(1+\b)y_1'} \, dy_1' + \int_0^{y_1} e^{(1-\b)y_1'} \,dy_1'\right) \leq 2C e^{-y_1} (c_1 + c_2 e^{(1-\b) y_1}) \\
& \leq C_2 e^{-\b y_1}. 
\end{aligned}
\end{equation}
To bound $B$, we use \eqref{prop:keyN:eq:Rest} again and change variables $y_1'\to \xi = y_1-y_1'$ in the integral 
\begin{align}
|B| & \leq 2 \int_{y_1-1}^\infty C(1+|y_1-y_1'| + |\log |y_1-y_1'||) e^{-\b y_1'}\, dy_1' = \int_{-\infty}^1 2C(1+ |\xi| + |\log|\xi||) e^{-\b(y_1-\xi)}\, d\xi \notag  \\
& \leq 2Ce^{-\b y_1} \int_{-\infty}^1 (1+|\xi| + |\log|\xi||) e^{\b \xi} \,d\xi = C_3 e^{-\b y_1}, \label{prop:keyN:eq:B}
\end{align} 
where we used in the last line that $\log|\xi|$ is integrable near the $\xi=0$ and the exponential decay $e^{\b\xi}$ ensures integrability as $\xi\to -\infty$. Combining \eqref{prop:keyN:eq:A} and \eqref{prop:keyN:eq:B}, we get the desired bound \eqref{prop:keyN:eq:claim} for $v$.
\end{proof}

\begin{rem} We crucially used the symmetry $x_1\to -x_1$ of $G_0$ to cancel the term $$\mu = \frac{1}{2\pi}\int_{\de \I} y_1' g_*(y') \, dy_1'$$ in \eqref{prop:keyN:eq:vexp}. Without it, $v$ tends to different constants $\pm \mu$ at the two ends of the strip, so that the solvability theory of \eqref{eq:harmNeu} in $C^{2,\a}_\b(\ol{\o})$ would fail. 
\end{rem}

We finally possess all the ingredients to give the proof of Theorem \ref{thm:solv}.
\begin{proof}[Proof of Theorem \ref{thm:solv}]
If $w, w'\in \bar{C}^{2,\a}_\b(\ol{\o})$ solve \eqref{eq:inhomhp}, then $v=w-w'$ is harmonic with zero Neumann data on $\de\o$ and a positive decay rate at infinity. Arguing as in Proposition \ref{prop:inj}, we establish that $w=w'$. As a result, if $w\in \bar{C}^{2,\a}_\b(\ol{\o})$ solves \eqref{eq:inhomhp} with $G_0$-invariant data $(f,g)$, then $w$ is automatically $G_0$-invariant.

To prove the existence of a solution, first find the unique solution $w_1\in \bar{C}^{2,\a}_{\b, G_0}(\overline{\o})$ for the Dirichlet problem \eqref{eq:homDir} and then consider the Neumann problem \eqref{eq:harmNeu} with data $g_2:= g- \de_\nu w_1 \in \bar{C}^{1,\a}_{\b+1, G_0}(\ol{\o})$. By Proposition \ref{prop:keyN}, \eqref{eq:harmNeu} possesses a solution $w_2\in \bar{C}^{2,\a}_\b(\ol{\o})$. Hence, $w=w_1+w_2\in \bar{C}^{2,\a}_\b(\ol{\o})$ solves \eqref{eq:inhomhp} and by the argument in the first paragraph, $w\in \bar{C}^{2,\a}_{\b,G_0}(\ol{\o})$.  

\end{proof}

\subsection{The solvability of the Neumann problem on the product domain.}

Let us now consider the solvability of the Neumann problem on $\O = \o\times \R$ with data $(f_0, g_0)$ that have zero mean in the axial $t$-direction: 
\begin{equation}\label{eq:inhomzeromean}
\left\{
\begin{aligned}
-\D w &= f_0 && \text{in } \Omega,\\
\de_\nu w &= g_0 && \text{on } \G,
\end{aligned}
\right.
\end{equation}
where 
\begin{align*}
f_0(x,t) &= f_1(x,t) - \frac{1}{T}\int_{-T/2}^{T/2} f_1(x,t)\,dt \quad \text{for }(x,t)\in \O, \\
g_0(p,t) &=g_1(p,t) -\frac{1}{T}\int_{-T/2}^{T/2} g_1(p,t)\,dt \quad \text{for }(p,t)\in \G=\de\O.
\end{align*}
Obviously $(f_0, g_0)\in C^{0,\a}_{\b+2, G}(\ol{\O})\times C^{1,\a}_{\b+1, G}(\G)$ once $f_1\in C^{0,\a}_{\b+2, G}(\ol{\O})$ and $g_1\in C^{1,\a}_{\b+1, G}(\G)$.

We will first formulate the problem \eqref{eq:inhomzeromean} weakly in a suitable Hilbert space and obtain the existence of a weak solution via the Riesz Representation Theorem. For measurable functions $\f:\O\to \R$ and $\phi:\G \to \R$ that are $T$-periodic in the $t$-variable, set 
\begin{align*}
\|\f\|^2_{L^2_\text{per}(\O)} &:= \int_{\O_\text{per}} |\f|^2 \, dX, \qquad \|\phi\|^2_{L^2_\text{per}(\G)} := \int_{\G_\text{per}} |\phi|^2 \, dS, 
\end{align*}
where we recall the notation $\O_\text{per} = \o \times (-T/2,T/2)$ and $\G_\text{per} = \de\o \times  (-T/2,T/2)$ for the periodicity cells. Define the Sobolev space of functions
\begin{equation}\label{eq:H1per}
\begin{aligned}
H^1_\text{per}(\O) &:=\{\f \in H^1_\text{loc}(\O): \f(x,t+T)=\f(x,t) \quad \forall (x,t)\in \O, ~\text{with } \|\f\|_{H^1_\text{per}(\O)} <\infty \},   \\
 \|\f\|_{H^1_{\text{per}}(\O)}^2 &:= \|D \f\|_{L^2_\text{per}(\O)}^2 +  \|\f\|_{L^2_\text{per}(\O)}^2 
\end{aligned}
\end{equation}
and let 
\begin{equation}\label{eq:Hilbert}
\mathcal{H}:= \{\f \in H^1_\text{per}(\O): \int_{-T/2}^{T/2} \f(x,t) \, dt  = 0 \quad \text{for a.e. } x\in \o\}
\end{equation}
be the subspace of zero-mean functions along the axial direction. It is a closed subspace of $H^1_\text{per}(\O)$, since if we take a sequence $\f_k\in \mathcal{H}$ converging to $\f$ in $H^1_{\text{per}}(\O)$, then by the Cauchy--Schwarz inequality and the Fubini--Tonelli theorem, we will obtain that
\begin{align*}
\int_\o \left|\int_{-T/2}^{T/2} \f(x,t)\, dt\right|^2  dx &= \int_\o \left|\int_{-T/2}^{T/2} (\f(x,t)-\f_k(x,t))\, dt\right|^2  dx  \\
& \leq T \int_\o \int_{-T/2}^{T/2}  |\f(x,t)-\f_k(x,t)|^2 \,dt \, dx = T\|\f-\f_k\|^2_{L^2_\text{per}(\O)} \to 0,
\end{align*}
as $k\to \infty$, implying $\int_{-T/2}^{T/2} \f(x,t)\, dt =0$ for a.e.\ $x\in \o$, so that $\f\in \mathcal{H}$. Hence, $\H$ is complete with respect to the $H^1_{\text{per}}(\O)$-norm. Furthermore, the Poincar\'e inequality implies that for every $\f \in \mathcal{H}$,
\begin{equation}\label{thm:solvezeromean:poincare}
\int_{\O_\text{per}} |D \f|^2 \, dX \geq  \int_{\o}\int_{-T/2}^{T/2} (\de_{t}\f)^2  \, dt  \, dx \geq \frac{c_0^2}{T^2} \int_{\o}\int_{-T/2}^{T/2} |\f|^2\, dt\, dx = c_1^2 \|\f\|^2_{L^2_\text{per}(\O)}.
\end{equation}
Therefore, 
\[
\langle \psi, \f\rangle_\H:= \int_{\O_\text{per}} D \psi \cdot D \f \, dX
\]
defines an inner product on $\H$, whose induced norm $\|\cdot\|_{\H}$ on $\H$ is equivalent to the $\|\cdot\|_{H^1_\text{per}(\O)}$-norm, thereby turning $\mathcal{H}$ into a Hilbert space. We are now ready to establish the necessary solvability result.

\begin{theorem}\label{thm:solvezeromean}
Let $\b\in (0,1]$. Suppose that $(f_0, g_0)\in C^{0,\a}_{\b+2, G}(\ol{\O})\times C^{1,\a}_{\b+1, G}(\G)$ and that 
$$
\int_{-T/2}^{T/2} f_0(x,t)\, dt = 0 \quad \text{for every } x\in \o \quad \text{and}\quad \int_{-T/2}^{T/2} g_0(p,t)\, dt  =0 \quad \text{for every } p\in\de\o.
$$
Then \eqref{eq:inhomzeromean} possesses a unique solution $w_0\in C^{2,\a}_{\b, G}(\ol{\O})$. 
\end{theorem}
\begin{proof}
Claim that 
\[
l(\f):=\int_{\O_\text{per}} f_0 \f \, dX + \int_{\G_\text{per}} g_0 \f \,  dS
\]
defines a bounded linear functional on $\mathcal{H}$. Indeed, as $|f_0(x,t)| \lesssim |x|^{-\b-2}$ and $|g_0|(x,t)\lesssim |x|^{-\b-1}$,  we have $f_0\in L^2_\text{per}(\O)$ and $g_0\in L^2_\text{per}(\G)$. Then an application of Cauchy--Schwarz  yields
\begin{equation}\label{thm:solvezeromean:func}
|l(\f)| \leq \|f_0\|_{L^2(\O_\text{per})}  \|\f \|_{L^2(\O_\text{per})} + \|g_0\|_{L^2(\G_\text{per})} \|\f\|_{L^2(\G_\text{per})}.
\end{equation}
At this stage we combine the trace inequality \eqref{prop:tracePer:ineq} and \eqref{thm:solvezeromean:poincare} to bound
\[
\|\f\|_{L^2(\G_\text{per})} \leq (2 \|\f\|_{L^2(\O_\text{per})}  \|D \f\|_{L^2(\O_\text{per})})^{1/2} \leq c_2 \|\f\|_{\H},
\]
so that \eqref{thm:solvezeromean:func} implies 
\[
|l(\f)| \leq C_0 (\|f_0\|_{L^2(\O_\text{per})} + \|g_0\|_{L^2(\G_\text{per})}) \|\f\|_\H.
\]
Hence, by the Riesz Representation Theorem there exists a unique $w_0\in \H$ such that $\langle w_0, \f \rangle_\H = l(\f)$ for all $\f \in \H$, i.e. that 
\begin{equation}\label{thm:solvezeromean:eq:dual}
\int_{\O_\text{per}} D w_0 \cdot D \f \, dX  = \int_{\O_\text{per
}} f_0 \f \, dX + \int_{\G_\text{per}} g_0 \f  \, dS
\end{equation}
for all $\f \in \H$. Moreover, 
\begin{equation}\label{thm:solvezeromean:LMbd}
\|w_0\|_{H^1_\text{per}(\O)} \leq c \|w_0\|_\H = c\|l\|_{\H^*} \leq C(\|f_0\|_{L^2(\O_\text{per})} + \|g_0\|_{L^2(\G_\text{per})}).
\end{equation}
Using the fact that $f_0$, $g_0$, $w_0$ and $D_x w_0$ have zero mean along $t$, it is not hard to check that \eqref{thm:solvezeromean:eq:dual} actually holds for every smooth test function $\psi(x,t)$, that is $T$-periodic in $t$ and compactly supported in $x$. 
Thus, $w_0\in \H \subset H^1_\text{loc}(\O)$ is a weak solution to \eqref{eq:inhomzeromean} that is $T$-periodic in $t$. Taking into account the regularity of $f_0$, $g_0$ and the smoothness of $\G$, it must be that $w_0\in C^{2,\a}_\text{loc}(\ol{\O})$ by the elliptic regularity theory for the Neumann problem. To establish $w_0\in C^{2,\a}_{\b}(\ol{\O})$, it will be enough to show $w_0\in C_\b(\ol{\O})$
and then invoke Lemma \ref{lem:Neum-apriori-prod}. 

Let $X=(x,t)\in \ol{\O}$, $R:=\la x \ra\geq 1$ and let $\d>0$ be the constant from Proposition \ref{prop:asymp}. By a covering argument, it suffices to prove that there is a constant $C$, independent of $R$, such that
\begin{equation}\label{thm:solvezeromean:eq:w0}
 \|w_0\|_{L^\infty(S)} \leq C R^{-\b}, 
\end{equation}
\begin{itemize}
\item when $S=\ol{\o}_{\d R}(y)\times (-T R/2, TR/2)$ for some free boundary point $y\in B_{R/4}(x)\cap \de\o$. Note that $S\subset S':=\ol{\o}_{2\d R}(y)\times (-T R, TR)$;
\item when $S = B_{\d R /4}(y)\times (-TR/2, TR/2)$  for some $y\in B_{R/4}(x)$, such that $S':=B_{\d R/2}(y)\subset \o$.
\end{itemize}
In the first case, since $S\subset S'$ are uniformly Lipschitz in $R$, we can apply the local boundedness principle for the oblique derivative problem (\cite[Theorem 5.36]{Liebermanbook}) to obtain
\begin{align}
\|w_0\|_{L^\infty(S)} &\leq C \left(R^2 \|f_0\|_{L^\infty(S')} + R \|g_0\|_{L^\infty(\de S'\cap \G)} + R^{-3/2}\|w_0\|_{L^2(S')} \right) \notag \\
& \leq C R^{-\b}\left(\|f_0\|_{C^{0,\a}_{\b+2}(\O)} + \|g_0\|_{C^{1,\a}_{\b+1}(\G)} \right) + C R^{-3/2}\|w_0\|_{L^2(S')} \label{thm:solvezeromean:eq:locmax}. 
\end{align}
On the other hand, we can bound $\|w\|_{L^2(S')}$ using the periodicity of $w_0$ in $t$:
\begin{align}
\|w_0\|^2_{L^2(S')} &\leq \int_{\o\times [-T\lceil R \rceil, T \lceil R\rceil ]} w_0^2 \, dx \, dt \leq 4 R \int_{\o\times [-T/2, T/2]} w_0^2 \, dx \, dt = 4 R \|w_0\|_{L^2_\text{per}(\O)}^2 \notag \\ 
 & \leq 4 C^2 (\|f_0\|_{L^2(\O_\text{per})} + \|g_0\|_{L^2(\G_\text{per})})^2 R = C_1(f_0,g_0, T) R, \label{thm:solvezeromean:eq:periodest} 
\end{align}
where in the penultimate inequality we used \eqref{thm:solvezeromean:LMbd}. Combining \eqref{thm:solvezeromean:eq:locmax} and \eqref{thm:solvezeromean:eq:periodest}, we get that 
\[
\|w_0\|_{L^\infty(S)} \leq C_2(f_0, g_0, T) (R^{-\b} + R^{-3/2+1/2}) \leq 2C_2(f_0,g_0, T) R^{-\b},
\]
as $\b \in (0,1]$ and confirm \eqref{thm:solvezeromean:eq:w0}. The second case when $S=B_{\d R /4}(y)\times (-TR/2, TR/2)$ follows in the same fashion, directly from the local boundedness principle for the Poisson equation. 

Finally, by Proposition \ref{prop:inj}, this solution is the unique one among functions $C^{2,\a}_\b(\ol{\O})$ that are $T$-periodic in the $t$-variable. Since $(f_0,g_0)$ is $G$-invariant, we have $w_0\in C^{2,\a}_{\b,G}(\ol{\O})$.
\end{proof}

Putting together Theorems \ref{thm:solv} and \ref{thm:solvezeromean}, we finally establish the surjectivity of the operator $\bar{L}$:
\begin{prop}\label{prop:surj}  Let $\b\in (0,1)$. Then the operator $$\bar{L} = (-\D, \de_\nu): C^{2,\a}_{\b, G}(\overline{\O}) \to \text{ker }l$$
is surjective.
\end{prop}
\begin{proof}
Decompose any $(f_1,g_1)\in \text{ker }l$ into 
\[
(f_1,g_1)=(\bar{f}, \bar{g}) + (f_0, g_0)
\]
where $\bar{f}, \bar{g}$ are  the averages of $f_1, g_1$ over a single period in the axial variable $t$, and are effectively $G_0$-invariant functions on $\o$ and $\de\o$, respectively, satisfying the compatibility condition \eqref{eq:compatible}. Using Theorem \ref{thm:solv}, we find a solution $\bar{w}\in C^{2,\a}_{\b, G}(\ol{\O})$ to $\bar{L}(\bar{w})=(\bar{f}, \bar{g})$, which is constant in the axial direction. Then, Theorem \ref{thm:solvezeromean} gives us a solution $w_0\in C^{2,\a}_{\b, G}(\ol{\O})$ to $\bar{L}(w_0)=(f_0,g_0)$. Their sum $w_1:=\bar{w}+w_0\in  C^{2,\a}_{\b, G}(\ol{\O})$ solves $\bar{L}(w_1) = (f_1, g_1).$ 
\end{proof}

\section{Fredholm properties of the linearized operator in the general case} \label{sec:lingen}

In this section, we will treat the operator $\mathcal{L}$ for \textbf{Case B}:
\begin{equation}\label{eq:MlinB}
\begin{aligned}
\mathcal{L} : \X_G=C^{2,\a}_{\b,G}(\overline{\O}) &\to \Z_G=C^{0,\a}_{\b+2,G}(\overline{\O})\times C^{1,\a}_{\b+1,G}(\G),\\
\mathcal{L}[w] &= (-\D w, \de_\nu w - H w|_\G).
\end{aligned}
\end{equation}
in $n\geq 3$ dimensions. Here $G=G_0\times G_1(T)$, $T>0$, where $G_0\subset O(n)$ leaves the base solution $u:\R^n \to \R_{\geq 0}$ invariant and $G_1(T)$ is generated by translations $t\to t+T$ and reflections $t\to -t$. In addition to hypotheses (\ref{H1}) and (\ref{H2}), we will assume 
\begin{itemize}
\item[(H3)]  The free boundary of $u$ satisfies
\[
\lim_{R\to \infty} \sup_{x\in \text{FB}(u), |x|\geq R} (x/|x|)\cdot \nu \leq 0.
\]
\end{itemize}
and prove the following analogue of Theorem \ref{thm:Fredhp}:
\begin{theorem}\label{thm:Fredgen}
Let $n \geq 3$ and let $u:\R^n \to [0,\infty)$ be a solution of \eqref{FBP}, satisfying (\ref{H1})--(\ref{H3}). Let $G$ be as above. Then the operator $\mathcal{L}$ from \eqref{eq:MlinB} is a Fredholm operator of index $0$, provided  $\b \in \left(\frac{n-3}{2}, \min(\frac{n}{2}, n-2)\right)$.
\end{theorem}

The proof of Theorem \ref{thm:Fredgen} will proceed along the same basic lines as for Theorem \ref{thm:Fredhp}. For this reason, we will omit certain repetitive details and focus on the aspects that are different.

We decompose again
\[
\mathcal{L} = L - K, \quad \text{where } L w := (-\D w, \de_\nu w) \quad \text{and} \quad K w := (0, H w), \quad \text{for } w\in \X_G.
\]
As before, $K$ is a compact operator due to the faster than natural decay of the Hessian (\ref{H1}). 
\begin{prop}\label{prop:compactop} Let $\b>0$. Then $K: \X_G \to \Z_G$ is a compact operator. 
\end{prop}
\begin{proof}
Denote the norm of the restriction operator $w\to w|_\G$ by $C$ and let $(w_k)_k$ be a sequence in $C^{2,\a}_{\b, G}(\ol{\O})$, bounded by $1$. Then $\|w_k|_\G\|_{C^{2,\a}_\b(\G)}\leq C$ and by Arzel\`a-Ascoli, we can find a subsequence (relabelling indices), such that $w_k|_\G \to \f$ in $C^{1,\a}_\text{loc}(\G)$, for some $\f \in C^{2,\a}_{\b, G}(\G)$ with $\|\f\|_{C^{2,\a}_{\b}(\G)}\leq C$. Because of the improved Hessian decay (\ref{H1}), Corollary \ref{coro:curv} tells us that
\begin{equation*}
\|H\|_{C^{1,\a}_1(\G\cap \{|x|>R\})} \to 0 \quad \text{as } R\to \infty.
\end{equation*}
Given $\e>0$, we can then find $R\geq 1$ large enough such that for all $k\in \mathbb{N}$,
\begin{align}
\|H \f - H w_k|_\G\|_{C^{1,\a}_{\b+1}(\G\cap \{|x|>R\})} & \leq C_*\|H\|_{C^{1,\a}_1(\G\cap \{|x|>R\})} (\|\f\|_{C^{1,\a}_{\b}(\G)}+ \|w_k|_\G\|_{C^{1,\a}_{\b}(\G)}) \notag \\ 
& \leq CC_* \|H\|_{C^{1,\a}_1(\G\cap \{|x|>R\})} \leq \e. \label{prop:compactop:eq1}
\end{align}
Now, by the local convergence $w_k|_\G \to \f$ in $C^{1,\a}_\text{loc}(\G)$, we have for all large enough $k\in \mathbb{N}$,  
\begin{equation}\label{prop:compactop:eq2}
\|H \f - H w_k|_\G\|_{C^{1,\a}_{\b+1}(\G\cap \{|x|<4R\})} \leq C_*\|H\|_{C^{1,\a}_1(\G)} \|\f - w_k|_\G\|_{C^{1,\a}_\b(\G\cap \{|x|< 4R\})} \leq \e.
\end{equation}
By \eqref{prop:compactop:eq1} and \eqref{prop:compactop:eq2}, we deduce that $\|H \f - H w_k|_\G\|_{C^{1,\a}_{\b+1}(\G)} \leq \e$ for all large enough $k$ and, thus, the compactness of the operator $K$.
\end{proof}

Hence, it suffices to prove that $L$ has Fredholm index $0$. In fact, we will show that $L$ is a topological isomorphism for appropriate decay rates $\b>0$.

\begin{theorem}\label{thm:iso-higher} 
Assume that we are in the setting of Theorem \ref{thm:Fredgen}. Then $$L = (-\D, \de_\nu): C^{2,\a}_{\b, G}(\ol{\O})\to C^{0,\a}_{\b+2, G}(\ol{\O})\times C^{1,\a}_{\b+1, G}(\G)$$ is a topological isomorphism for $\b \in \left(\frac{n-3}{2}, \min(\frac{n}{2}, n-2)\right)$, i.e.\ for given data $(f, g) \in C^{0,\a}_{\b+2, G}(\ol{\O})\times C^{1,\a}_{\b+1, G}(\G)$ the Neumann problem
 \begin{equation}\label{thm:iso-higher:eq:Neum}
\left\{
\begin{aligned}
-(\D_x + \de_t^2) w &= f && \text{in } \Omega,\\
\de_\nu w &= g && \text{on } \G.
\end{aligned}
\right.
\end{equation}
has a unique solution $w\in C^{2,\a}_{\b, G}(\ol{\O})$.
\end{theorem}
\begin{proof}
The argument for the injectivity of $L$ is verbatim the one in Proposition \ref{prop:inj}. To prove the surjectivity of $L$, we  break up the problem again into a mean-zero part in the axial variable $t$ and a part that lives on $\o\subset \R^n$: $(f,g) = (f_0, g_0) + (\bar{f}, \bar{g})$, where
\begin{align*}
\bar{f}(x) &= \frac{1}{T} \int_{-T/2}^{T/2} f(x,t)\, dt, \quad \text{for } x\in \o, \quad \text{and} \quad \bar{g}(p) = \frac{1}{T} \int_{-T/2}^{T/2} g(p,t)\, dt, \quad \text{for } p\in \de\o.
\end{align*}
It then suffices to find solutions $w_0\in C^{2,\a}_{\b, G}(\ol{\O})$ and $\bar{w} \in \bar{C}^{2,\a}_{\b, G_0}(\ol{\o})$ to the Neumann problems $L w_0 = (f_0, g_0)$ and $\bar{L} \bar{w} = (\bar{f}, \bar{g})$, respectively, where $\bar{L} = (-\D_x, \de_\nu)$ is the operator on the base $\o$. The solvability of $L w_0 = (f_0, g_0)$ will be established in Theorem \ref{thm:solvezm_higher} in the range $\b\in ((n-3)/2,n/2]$, while the one for  $\bar{L} \bar{w} = (\bar{f}, \bar{g})$ -- in Theorem \ref{thm:LM_higherD}, for the range $\b\in (0,n-2)$. As a result $w(x,t):=w_0(x,t) + \bar{w}(x)$ belongs to $C^{2,\a}_{\b,G}(\ol{\O})$ and solves \eqref{thm:iso-higher:eq:Neum}, provided $\b\in \left((n-3)/2, \min(n-2, n/2)\right)$.
\end{proof}

\begin{theorem}\label{thm:solvezm_higher} Let $\b \in \left(\frac{n-3}{2}, \frac{n}{2}\right]$, $n\geq 3$. Suppose that $(f_0, g_0)\in C^{0,\a}_{\b+2, G}(\ol{\O})\times C^{1,\a}_{\b+1, G}(\G)$ and that
$$
\int_{-T/2}^{T/2} f_0(x,t)\, dt = 0 \quad \text{for every } x\in \o \quad \text{and}\quad \int_{-T/2}^{T/2} g_0(q,t)\, dt  =0 \quad \text{for every } q\in\de\o.
$$
Then there exists a solution $w_0\in C^{2,\a}_{\b,G}(\ol{\O})$ to \eqref{thm:iso-higher:eq:Neum} with right-hand side $(f_0, g_0).$
\end{theorem}
\begin{proof}
We proceed just like in the proof of Theorem \ref{thm:solvezeromean}, so we will omit some of the details. Define the Hilbert space $\H$ as in \eqref{eq:Hilbert} and consider the weak formulation of $Lw = (f_0, g_0)$ in $\H$:
\begin{equation}\label{thm:solvezm_higher:LM}
\la w, \f\ra_\H = l(\f):= \int_{\O_\text{per}} f_0 \f \, dX + \int_{\G_\text{per}} g_0 \f\, dS \quad \text{for all } \f\in \H.
\end{equation}
We observe that $|f_0(x, t)|^2\lesssim |x|^{-2(\b+2)}$ for large $x\in \o$, and $|g_0(q,t)|^2\lesssim |q|^{-2(\b+1)}$ for large $q\in\de\o$. Given that $2(\b+2)>n$ and $2(\b+1)>n-1$ when $\b> \frac{n-3}{2}$, we have $f_0\in L^2(\O_\text{per})$ and $g_0\in L^2(\G_\text{per})$, taking into account the volume/area growth of $\O_\text{per}$ and $\G_\text{per}$, respectively. Applying the Cauchy--Schwarz, the Poincar\'e and the trace inequality \eqref{prop:tracePer:ineq}, we again deduce that $l$ defines a bounded linear functional on $\H$. The Riesz Representation Theorem then yields a unique  function $w=w_0\in \H\subset H^1_\text{per}(\O)$, solving \eqref{thm:solvezm_higher:LM}, which is ultimately a weak solution of $Lw = (f_0, g_0)$ in $\O$. By uniqueness, $w_0$ must be $G$-invariant, since $(f_0, g_0)$ are.

By elliptic regularity, we get that $w_0\in C^{2,\a}_\text{loc}(\ol{\O})$ is a classical solution of the problem. To prove that $w_0\in C^{2,\a}_\b(\ol{\O})$ it suffices to show $w_0\in C_\b(\ol{\O})$ and invoke the a priori estimates of Lemma \ref{lem:Neum-apriori-prod}. To accomplish the latter, we prove $\|w_0\|_{L^\infty(S)} \leq \bar{C} R^{-\b}$, for the uniformly Lipschitz domains $S\subset S'$ as in Theorem \ref{thm:solvezeromean}.

Applying the local boundedness principle \cite[Theorem 5.36]{Liebermanbook} gives 
\begin{align}
\|w_0\|_{L^\infty(S)} &\leq C \left(R^2 \|f_0\|_{L^\infty(S')} + R \|g_0\|_{L^\infty(\de S'\cap \G)} + R^{-(n+1)/2}\|w_0\|_{L^2(S')} \right) \notag \\
& \leq C R^{-\b}\left(\|f_0\|_{C^{0,\a}_{\b+2}(\O)} + \|g_0\|_{C^{1,\a}_{\b+1}(\G)} \right) + C R^{-(n+1)/2}\|w_0\|_{L^2(S')} \label{thm:solvezm_higher:locmax}. 
\end{align}
As before, we have that $\|w_0\|_{L^2(S')}\leq c   \|w_0\|_\H R^{1/2} \leq C_1(f_0, g_0, T) R^{1/2}$, so that  \eqref{thm:solvezm_higher:locmax} entails
\[
\|w_0\|_{L^\infty(S)} \leq C_2(f_0, g_0, T)(R^{-\b} + R^{-n/2}) \leq \bar{C} R^{-\b},
\]
since $\b\leq n/2$. 
\end{proof}

The next theorem addresses the solvability of the Neumann problem on the base $\o$.

\begin{theorem}\label{thm:LM_higherD}
Let $\b\in (0, n-2)$, $n\geq 3$. Suppose that $(\bar{f}, \bar{g})\in \bar{C}^{0,\a}_{\b+2, G_0}(\ol{\o})\times \bar{C}^{1,\a}_{\b+1, G_0}(\de\o)$. Then there exists a unique solution $\bar{w}\in \bar{C}^{2,\a}_{\b, G_0}(\ol{\o})$ of 
 \begin{equation*}
\left\{
\begin{aligned}
-\D_x \bar{w} &= \bar{f} && \text{in } \omega,\\
\de_\nu \bar{w} &= \bar{g} && \text{on } \de\o.
\end{aligned}
\right.
\end{equation*}
\end{theorem}
The solvability of the Neumann problem on the base $\o$ will follow from the existence of an appropriate barrier for the inhomogeneous Neumann problem on $\o$ in Lemma \ref{lem:Neumann_barrier}, which itself is due to hypothesis (\ref{H3}).  
\begin{lemma}\label{lem:Neumann_barrier}
For every $\b\in (0, n-2)$, there exist $C>0$, $r_0>0$ and a function $h\in C^2(A_{r_0})$, where $A_{r_0}:=\ol{\o}\cap \{|x|\geq r_0\}$, such that $0<h(x)\leq C/|x|^\b$ for all $x\in A_{r_0}$ and 
 \begin{equation*}
\left\{
\begin{aligned}
-\D h & \geq |x|^{-\b-2} && \text{in } A_{r_0},\\
\de_\nu h & \geq |x|^{-\b-1} && \text{on } \de\o \cap \{|x|\geq r_0\}.
\end{aligned}
\right.
\end{equation*}
\end{lemma}
\begin{proof}
Let us show that 
\[
h(x):= a |x|^{-\b} - 2 u(x) |x|^{-\b-1}.
\]
satisfies the properties of the lemma when $a$ and $r_0\geq u(0)$ are chosen large enough. Indeed, by the Lipschitz bound $|\n u|\leq 1$ of Proposition \ref{prop:grad}, we know that
\begin{equation}\label{lem:Neumann_barrier:eq:1}
u(x)\leq u(0)+ |x| \leq r_0 + |x| \leq 2|x| \quad \text{whenever } x\in A_{r_0}
\end{equation}
Thus, $(a-4)|x|^{-\b}\leq h(x)\leq a|x|^{-\b}$ in $A_{r_0}$, so $h>0$ in $A_{r_0}$ whenever $a> 4$. Now, using \eqref{lem:Neumann_barrier:eq:1} and $\|\n u \|_{L^\infty(\o)}=1$, we can bound:
\begin{align*}
-\D h &\geq a \b (n-2 -\b) |x|^{-\b-2} - 4 (\b+1) (\nabla u \cdot x) |x|^{-\b-3} - 2 u(x) (\b+1)(n-3-\b) |x|^{-\b-3} \\
& \geq a \b (n-2 -\b) |x|^{-\b-2} - 4(\b+1) \|\n u\|_{L^\infty(\o)} |x|^{-\b-2} - 4 (\b+1)|n-3-\b| |x|^{-\b-2} \\
& \geq \left(a \b (n-2 -\b)  - 4 (\b+1) (1 + |n-3-\b|)\right) |x|^{-\b-2} \quad \text{in } A_{r_0},
\end{align*}
so that $-\D h \geq |x|^{-\b-2}$ in $A_{r_0}$ for values of $a\geq c(\b,n):=\frac{1+ 4 (\b+1) (1 + |n-3-\b|)}{\b(n-2-\b)}>0$. Choosing $a=\max \left(5, c(\b,n)\right)$, we satisfy the required bounds for $h$ and $[-\D h]$ and it remains to estimate $\de_\nu h$ on $\de\o\cap \{|x|\geq r_0\}$.  Taking into account that $u=0$ and $u_\nu = -1$ on $\de \o$, we obtain:
\begin{align*}
\de_\nu h &= -\b a|x|^{-\b-1} (x/|x|)\cdot \nu - 2 u_\nu |x|^{-\b-1} \\
& = (2- \b a  (x/|x|)\cdot \nu )|x|^{-\b-1}  \geq |x|^{-\b-1} \quad \text{on } \de\o\cap \{|x|\geq r_0\},
\end{align*}
provided that $r_0$ is large enough on account of hypothesis (H3).
\end{proof}

\begin{proof}[Proof of Theorem \ref{thm:LM_higherD}]
That the equation $L \bar{w} = (0,0)$ has  $\bar{w}=0$ as the unique solution among $w\in \bar{C}^{2,\a}_\b(\ol{\o})$ follows from the maximum principle and the Hopf lemma, as $w(x)\to 0$ as $|x|\to\infty$.

The proof of the existence of a solution uses the same scheme as Proposition \ref{prop:Dir}. One considers the solutions $\bar{w}_n$ of the mixed Dirichlet--Neumann boundary value problem in the expanding truncated domain $\o_{R_n}=\o\cap B_{R_n}$, $R_n\nearrow \infty$,
 \begin{equation*}
\left\{
\begin{aligned}
-\D_x \bar{w}_n &= \bar{f} && \text{in } \omega_{R_n},\\
\de_\nu \bar{w}_n &= \bar{g} && \text{on } \de\o\cap B_{R_n}. \\
\bar{w}_n &= 0 && \text{on } \o\cap \de B_{R_n},
\end{aligned}
\right.
\end{equation*}
(which can be found by Lax--Milgram) and uses the a priori estimate \eqref{lem:Neum-apriori:eq:est} of Lemma \ref{lem:Neum-apriori} and, crucially, the barrier $h$ from Lemma \ref{lem:Neumann_barrier} with the weak maximum principle and Hopf Lemma, in order to arrive at the stronger estimate 
\[
\|\bar{w}_n\|_{\bar{C}^{2,\a}_\b(\o_{R_n/2})} \leq C \left(\|\bar{f}\|_{\bar{C}^{0,\a}_{\b+2}(\o)} + \|\bar{g}\|_{\bar{C}^{1,\a}_{\b+1}(\de\o)}\right),
\]
for a constant $C$ independent of $n$. Then a standard limiting argument produces the global solution $\bar{w}\in \bar{C}^{2,\a}_\b(\ol{\o})$. Since the data is $G_0$-invariant, so is $\bar{w}$ by uniqueness.
\end{proof}

\begin{proof}[Proof of Theorem \ref{thm:Fredgen}]
It is a direct consequence of Theorem \ref{thm:iso-higher} and Proposition \ref{prop:compactop}. 
\end{proof}

\section{The bifurcation argument}\label{sec:bif}

In this section we verify that the reduced operator $M(T,v)$ of Proposition~\ref{prop:reduction} meets the hypotheses of the Crandall--Rabinowitz theorem at an explicit value $T=T_0$ of the period, and we deduce Theorems \ref{thm:mainHP} and \ref{thm:maingen}. We use the Crandall--Rabinowitz theorem in the following form (see \cite[Theorem~I.5.1]{Kiel_bifbook}).

\begin{theorem}[Crandall--Rabinowitz]\label{thm:CR}
Let $E_1$, $E_2$ be real Banach spaces, $J\subset(0,\infty)$ an open interval, $\mathcal{U}_1\subset E_1$ an open neighbourhood of $0$, and let $\mathcal{M}\in C^\infty(J\times \mathcal{U}_1;E_2)$ satisfy $\mathcal{M}(T,0)=0$ for all $T\in J$. Assume that for some $T_0\in J$:
\begin{enumerate}[label=(CR\arabic*), ref=CR\arabic*]
\item\label{CR1} $D_v\mathcal{M}(T_0,0):E_1\to E_2$ is a Fredholm operator of index $0$ whose kernel is one-dimensional, spanned by some $W\in E_1$;
\item\label{CR2} $D_TD_v\mathcal{M}(T_0,0)[W]\notin \im D_v\mathcal{M}(T_0,0)$.
\end{enumerate}
Then there exist $\e>0$ and a smooth curve $(-\e,\e)\ni s\mapsto (T_s,v_s)\in J\times \mathcal{U}_1$, passing through $(T_0,0)$ at $s=0$, such that
\[
\mathcal{M}(T_s,v_s)=0\quad\text{and}\quad v_s=sW+\psi(s), \quad\text{with } \|\psi(s)\|_{E_1}=o(s) \text{ as } s\to 0.
\]
Moreover, there is a neighbourhood of $(T_0,0)$ in $J\times\mathcal{U}_1$ in which every zero of $\mathcal{M}$ is either of the form $(T,0)$ or lies on this curve.
\end{theorem}

Throughout this section we are in one of the following two cases (cf.\ Subsection~\ref{ss:setting}), in which we also fix the symmetry group $G_0$ and the decay rate $\b$:
\begin{description}
\item[Case A] $u=HP:\R^2\to[0,\infty)$ is the double hairpin solution, $G_0$ is the group generated by \emph{both} coordinate reflections $x_i\mapsto -x_i$, $i=1,2$, and $\b\in(0,1)$;
\item[Case B] $u:\R^n\to[0,\infty)$, $n\geq 3$, satisfies hypotheses (\ref{H1})--(\ref{H5}), $G_0$ is the group from (\ref{H5}), and $\b\in\left(\frac{n-3}{2},\min\left(\frac n2,n-2\right)\right)$.
\end{description}
In both cases we set $G:=G(1)=G_0\times G_1(1)$, where $G_1(1)$ is generated by $t\mapsto t+1$ and $t\mapsto -t$, and work with the invariant subspaces
\[
\X_1:=\X_{G},\qquad \Y_1:=\Y_{G},\qquad \Z_1:=\Z_{G},
\]
of the spaces $\X,\Y,\Z$ from \eqref{eq:spacesA}--\eqref{eq:spacesB}:
\begin{align*}
  &(\textbf{Case A})\quad \X:=\X^{2,\a}_\b(\ol{\O}),\quad \Y:=\X^{2,\a}_\b(\G),\quad \Z= C^{0,\a}_{\b+2}(\ol{\O})\times C^{1,\a}_{\b+1}(\G); \\
  &(\textbf{Case B})\quad \X:=C^{2,\a}_\b(\ol{\O}),\quad \Y:=C^{2,\a}_\b(\G),\quad \Z= C^{0,\a}_{\b+2}(\ol{\O})\times C^{1,\a}_{\b+1}(\G) .
\end{align*}
By Proposition \ref{prop:reduction}(d), $M(T,\cdot)$ maps $\mathcal{U}_1:=\mathcal{U}\cap\X_1$, an open neighbourhood of $0$ in $\X_1$, into $\Z_1$ for every $T>0$ and its linearized operator $D_v M(T,0): \X_1 \to \Z_1$ is given by
\begin{equation}\label{eq:linM}
D_v M(T,0)[w] =\left(-(\D_x+T^{-2}\p_t^2)w, \p_\nu w - Hw \right).
\end{equation}
From Corollary \ref{cor:conjugation}, Theorem \ref{thm:Fredhp} in Case A and Theorem \ref{thm:Fredgen} in Case B, we know that  $D_v M(T,0)$ is a Fredholm operator of index $0$ for every $T>0$.

In both cases $u$ satisfies (\ref{H1}) (in Case A by Lemma \ref{lem:Hessdecay}) and has finite nontrivial Morse index. Hence, by Corollary \ref{coro:grounddecay}, the lowest $L^2(\o)$-eigenvalue $\l<0$ of the Robin problem
\begin{equation}\label{eq:Jacobi}
\left\{
\begin{aligned}
- \Delta w &= \l w && \text{ in } \o, \\
\de_\nu w - Hw &= 0 && \text{ on } \de\o,
\end{aligned}
\right.
\end{equation}
associated to the second variation form $Q$, is simple, and it has a unique positive $L^2(\o)$-normalized eigenfunction $w_\l$, which is $G_0$-invariant. By Theorem~\ref{thm:eigdecay}, $w_\l\in \bar{C}^{2,\a}_{\mu}(\ol{\o})$ for every $\mu\geq 0$. We define
\begin{equation}\label{eq:T0W}
T_0:=\frac{2\pi}{\sqrt{|\l|}},\qquad W(x,t):=w_\l(x)\cos(2\pi t),\quad (x,t)\in\ol{\O},
\end{equation}
so that $(2\pi/T_0)^2=-\l$. Since $w_\l$ decays faster than any power and is $G_0$-invariant, while $\cos(2\pi t)$ is even and $1$-periodic, we have $W\in C^{2,\a}_{\b,G}(\ol{\O})\subset \X_1$ in both cases. We also note that $w_\l>0$ on $\ol{\o}$: if $w_\l(p)=0$ at some $p\in\de\o$, then the Hopf lemma, applied to $w_\l$, which satisfies $(-\D+|\l|) w_\l=0$ and attains its minimum $0$ at $p$, would give $\de_\nu w_\l(p)<0$, contradicting $\de_\nu w_\l(p)=H(p)w_\l(p)=0$.

We will now show that $M$ meets the Fredholm (\ref{CR1}) and transversality (\ref{CR2}) properties required by the Crandall--Rabinowitz Theorem.
 
\begin{prop}\label{prop:Jackernel}
Assume that we are in Case A or Case B above, and let $\l$, $w_\l$, $T_0$ and $W$ be as in \eqref{eq:T0W}. Then the linearized operator \eqref{eq:linM} has a nontrivial kernel  when $T =T_0$; the kernel is one-dimensional and is spanned by $W$. Moreover,
\begin{equation}\label{prop:Jackernel:transv}
D_T D_v M(T_{0}, 0) [W] \notin \text{Im}(D_v M(T_{0}, 0)).
\end{equation}
\end{prop}
\begin{proof}
It is easy to check that $D_v M(T_0,0)[W] = (0,0)$. Conversely, suppose $\bar{v}\in \X_1$ belongs to the kernel of $D_v M(T_0,0)$. In Case A we write $\bar{v} = c + v$, where $c \in \R$ and $v\in C^{2,\a}_{\b,G}(\ol{\O})$; in Case B we simply set $c=0$. Since $v$ is even and $1$-periodic in $t$, we can expand it in Fourier cosine series in $t$:
\begin{equation*}
    v(x,t)= \sum_{k=0}^{\infty}v_k(x) \cos(2\pi k t),
\end{equation*}
where every Fourier coefficient $v_k\in \bar{C}^{2,\a}_{\b, G_0}(\ol{\o})$ must solve
\begin{equation*}
        \left\{
\begin{aligned}
-\Delta v_k &=  -\left(2\pi k/T_0\right)^2 v_k= (k^2 \l) v_k && \text{ in } \o,\\
\partial_\nu v_k &= Hv_k && \text{ on } \partial \o,
\end{aligned}
\right.
\end{equation*}
for $k\in \mathbb{N}$, while $\bar{v}_0(x) := c + v_0(x)$ solves:
\begin{equation*}
\left\{
\begin{aligned}
-\Delta \bar{v}_0 &= 0 && \text{ in } \o,\\
\partial_\nu \bar{v}_0 &= H\bar{v}_0  && \text{ on } \partial \o.
\end{aligned}
\right.
\end{equation*}

For $k\in \mathbb{N}$, we easily check that $v_k = (T_0/(2\pi k))^2 \D v_k \in \bar{C}^{0,\a}_{\b+2}(\ol{\o})\subset L^2(\o)$ since $2(\b+2)>n$ in both cases (in Case A, $n=2$ and $\b>0$; in Case B, $\b>\frac{n-3}{2}$). If $v_k\not \equiv 0$, then $v_k$ is an $L^2(\o)$-eigenfunction of \eqref{eq:Jacobi} with eigenvalue $k^2\l$. For $k\geq 2$ this is impossible, since $k^2 \l <\l$ and $\l$ is the lowest $L^2(\o)$-eigenvalue; thus, $v_k\equiv 0$ for $k\geq 2$. For $k=1$, the simplicity of $\l$ tells us that $v_1=c_0 w_\l$. 

When $k=0$, we see that
\begin{enumerate}[label=(\alph*)]
\item In Case A,  $\bar{v}_0$ is bounded and satisfies
\[
\int_\o |D \bar{v}_0|^2 \, dx = \int_\o |D v_0|^2 \, dx <\infty,
\] 
where the $L^2(\o)$ integrability of $D v_0$ is a consequence of $|D v_0(x)|\leq \|v_0\|_{C^{2,\a}_\b(\o)}\la x\ra^{-\b-1}$. We can thus invoke the kernel nondegeneracy Proposition \ref{prop:kernondeg} to deduce that
\[
\bar{v}_0 = c_1 u_{x_1} + c_2 u_{x_2}, \quad \text{for some } c_1, c_2 \in \R.
\]
However, by the $G_0$-invariance, $\bar{v}_0$ is even in both $x_1$ and $x_2$, whence both $c_1=c_2 = 0$ and $\bar{v}_0\equiv 0$. 

\item In Case B, $c=0$ and $\bar{v}_0(x)=v_0(x) \in C^{2,\a}_{\b, G_0}(\ol{\o})$ decays to $0$ as $|x|\to \infty$, so the nondegeneracy hypothesis (\ref{H5}) gives $v_0\equiv 0$. 
\end{enumerate}
In both cases we get that $\bar{v}(x,t) = v_1(x) \cos(2\pi t) = c_0 W(x,t)$. 

To recover the transversality condition \eqref{prop:Jackernel:transv}, we use \eqref{eq:DTDvMbar} and $\de_t^2W=-4\pi^2W$ to compute:
\[
D_TD_vM(T_0,0)[W]=\big(2T_0^{-3}\de_t^2W,\,0\big)=\big(-8\pi^2T_0^{-3}W,\,0\big).
\]
Therefore, satisfying \eqref{prop:Jackernel:transv} is equivalent to the problem
\begin{equation}\label{prop:transv:eq:prob}
    \begin{aligned}
        \left\{
\begin{aligned}
-  \D_{T_0} v &= W && \text{ in } \O, \quad \text{where } \D_{T_0}:= \D_x + T_0^{-2}\de_t^2, \\
\de_\nu v  - H v &= 0 && \text{ on } \partial \O,
\end{aligned}
\right.
    \end{aligned}
\end{equation}
not having a solution in $\X_1$. Assume that it does. After multiplying \eqref{prop:transv:eq:prob} by $W$, the rapid decay of $W$ at infinity, the boundedness of $v$ and its derivatives, and periodicity allows us to apply Green's identity over $\O_\text{per} = \o \times (-1/2,1/2)$, $\G_\text{per}=\de\o \times (-1/2,1/2)$, to deduce
\begin{align*}
\int_{\O_\text{per}} W^2 \, dX &= \int_{\O_\text{per}} (-\D_{T_0} v)  W \, dX \\
& = \int_{\O_\text{per}} v(-\D_{T_0} W) \, dX - \int_{\G_\text{per}} (\de_\nu vW  - v \de_\nu W) \, dS =0,
\end{align*}
where the last equality is due to $D_v M(T_0, 0) [W] = (- \D_{T_0} W, \de_\nu W - H W)=(0, 0)$. Since $W\neq 0$, we reach a contradiction. 
\end{proof}

Proposition \ref{prop:Jackernel} puts us exactly in the framework of the Crandall--Rabinowitz Theorem \ref{thm:CR}.

\begin{theorem}\label{thm:bif} Assume that we are in the setting of Proposition \ref{prop:Jackernel}. Then there exists $\e>0$, for which the following holds.
\begin{enumerate}[label = (\alph*)]
\item There is a smooth curve $(-\e,\e)\ni s\mapsto (T_s,v_s)\in(0,\infty)\times\mathcal{U}_1$, passing through $(T_0,0)$ at $s=0$, such that $M(T_s,v_s)=0$. Denoting $\f_s:=F(v_s)\in\Y_1$, where $F$ is the map of Proposition \ref{prop:reduction}, we have
\begin{equation}\label{thm:bifur:eq:exp}
v_s=sW+o(s)\ \text{ in }\X\qquad\text{and}\qquad \f_s=sW|_\G+o(s)\ \text{ in }\Y,\qquad\text{as } s\to 0.
\end{equation}
Moreover, in a neighbourhood of $(T_0,0)$ in $(0,\infty)\times\X_1$, the zeros of $M$ are either of the form $(T,0)$ or lie on this curve. \\

\item For every $s\in(-\e ,\e)$, the function $U^s:=\hat U$ defined in Lemma~\ref{lem:zeros-to-solutions}(b) with data $(T,\f,v)=(T_s,\f_s,v_s)$ is a classical global solution of \eqref{FBP} in $\R^{n+1}$, which is $G_0$-invariant in $x$, even and $T_s$-periodic in $t$, and $U^0=U$. Its positive phase is $S_{T_s}(\Phi_{\f_s}(\O))$ and its free boundary is the normal graph over $\G=\text{FB}(U)$
\begin{equation}\label{thm:bifur:eq:graph}
\text{FB}(U^s)=\big\{\big(p+\check\f_s(p,t)\,\nu_\o(p),\,t\big):\ (p,t)\in\de\o\times\R\big\},\qquad \check\f_s(p,t):=\f_s(p,t/T_s).
\end{equation}
For $s\neq 0$, $\check\f_s$ is nonconstant in $t$; in particular, $U^s$ is not a product solution.
\end{enumerate}
\end{theorem}
\begin{proof}
By Proposition \ref{prop:Jackernel}, Theorem \ref{thm:CR} applies to $\mathcal{M}(T,v):=M(T,v)$ at $T=T_0$, with $E_1:=\X_1$ and $E_2:=\Z_1$. It yields $\e>0$ and the smooth curve $s\mapsto(T_s,v_s)$ in $(0,\infty)\times\mathcal U_1$ with $M(T_s,v_s)=0$, the expansion
\[
v_s  = s W + o(s) \text{ in }\X,\quad\text{as } s\to 0, 
\]
as well as the local uniqueness statement. By Proposition~\ref{prop:reduction}, $F$ is smooth and equivariant with $F(0)=0$ and $DF(0)[w]=w|_\G$, so $\f_s=F(v_s)\in V'\cap\Y_1$ and $\|F(v)-v|_\G\|_{\Y}=O(\|v\|_{\X}^2)$ as $v\to 0$. As $w\mapsto w|_\G$ is bounded $\X\to\Y$, the second expansion in \eqref{thm:bifur:eq:exp} follows, and we complete the proof of Part (a).

By Proposition~\ref{prop:reduction}(e), $N(T_s,\f_s,v_s)=0$, so that $\tilde{U}_s := U\circ \Phi_{\f_s} + v_s$  solves the boundary value problem
\[
- \D_{g_{T_s, \f_s}} \tilde{U}_s = 0 \quad \text{in } \O, \qquad \tilde{U}_s = 1- |d \tilde{U}_s|^2_{g_{T_s, \f_s}} = 0 \quad \text{on } \G.
\]
To show that it corresponds to a genuine solution of the one-phase FBP, we have to additionally confirm that $\tilde{U}_s>0$ in $\O$. By reducing $\e$, we can ensure that for all $s\in (-\e,\e)$,
\[
\|v_s\|_{L^\infty(\O)}+\|\n v_s\|_{L^\infty(\O)}\leq 2\|v_s\|_{\X} \leq C |s| \leq 1/4 \quad \text{and} \quad \|\f_s\|_\Y \leq C|s| \leq \theta/2,
\]
so that Lemma \ref{lem:zeros-to-solutions}(a) implies that $\tilde{U}_s>0$ in $\N_1\cap \O=\{X = p-z\nu(p): 0< z <4\theta \rho(p)\}$. 
Beyond this band, for $X\in \O\setminus \N_1$, we have $d(X,\G)\geq 4\theta$ and $\Phi_{\f_s}(X)=X$. If $\text{FB}(u)$ is noncompact, the distance nondegeneracy bound for such finite Morse index solutions in Theorem \ref{thm:nondeg}(b) implies that $U(X)\geq c_u d(X,\G) \geq 4\theta c_u$ for all $X\in \O\setminus \N_1$. Therefore, 
\[
\tilde{U}(X) = U\circ \Phi_{\f_s}(X) + v_s(X) = U(X) + v_s(X) \geq 4\theta c_u  - \|v_s\|_{L^\infty} \geq 4\theta c_u - C |s|>0,
\]
for all $X\in \O\setminus \N_1$ if $\e$ is taken small enough. If $\text{FB}(u)$ is compact\footnote{In this case, $u(x)=L_n(x) = (1-|x|^{2-n})^+/(n-2)$ is the ball complement solution by a result of \cite{Reichel}, but we do not need the full classification for the argument at hand.} in Case B, then $u$ is positive harmonic with bounded gradient $\|\n u\|_{L^
\infty}\leq 1$ on a exterior domain in $\R^n$, $n\geq 3$, so $u=c + O(|x|^{2-n})$ with $c>0$ by the maximum principle. By compactness and another application of the maximum principle, $U(x,t) = u(x)\geq c_u$ for all $(x,t)\in \{X: d(X, \G)\geq 4\theta\}$, so that the same argument as above gives $\tilde{U}>0$ in all of $\O$ again.

Now Lemma~\ref{lem:zeros-to-solutions}(b) shows that $U^s$ is a classical solution of \eqref{FBP} in $\R^{n+1}$ with positive phase $S_{T_s}(\Phi_{\f_s}(\O))$. Since $\f_s$ and $v_s$ are $G(1)$-invariant, $U^s$ is $G(T_s)$-invariant, i.e.\ $G_0$-invariant in $x$, even and $T_s$-periodic in $t$. For $s=0$ we have $\f_0=0$, $v_0=0$, so $U^0=(S_{T_0}^{-1})^*U=U$, as $U$ does not depend on $t$. By Lemma~\ref{lem:deformation}(b), $\Phi_{\f_s}(p,\t)=(p+\f_s(p,\t)\nu_\o(p),\t)$ for $(p,\t)\in\G$, and applying $S_{T_s}$ gives \eqref{thm:bifur:eq:graph}.

Finally, since $\f_s(p,t) = sW(p,t) + o(s) = s w_\l(p)\cos(2\pi t) + o(s)$ and $w_\l(p)>0$ for all $p$, we see that $\check{\f}_s(p,t)=\f_s(p,t/T_s)$ is non-constant in $t$ for $s\neq 0$, so that $U^s$ is not a product solution for $s\neq 0$. This  completes the proof of the Theorem. 
\end{proof}

\begin{proof}[The proofs of Theorems \ref{thm:mainHP} and \ref{thm:maingen}]
Reading the conclusions of Theorem \ref{thm:bif} for Case A yields Theorem \ref{thm:mainHP}, while for Case B -- Theorem \ref{thm:maingen}.
\end{proof}

\appendix

\section{A priori estimates}\label{app:apriori}

In this appendix we give the a priori estimates underlying the solvability theory of the Dirichlet and Neumann problems for the Laplacian on $\o=\{u>0\}$ and $\O=\o\times \R$. We assume that $u:\R^n\to \R_{\geq 0}$ is a global solution of \eqref{FBP} that satisfies the natural Hessian bound:
\[
|D^2 u(x)| \leq \bar{K}/\la x\ra, \quad \text{for all } x\in \o,
\]
for some $\bar{K}>0$. Recall that for a set $S\subset \R^n$, and $y\in S$, we denote by $S_r(y)$ the connected component of $S\cap B_r(y)$ containing $y$. Denote $\o_r:=\o\cap B_r$.

\begin{lemma}\label{lem:Dir-apriori} Let $\a\in (0,1)$, $\b\geq 0$, $M\geq 1$, and $f\in \bar{C}^{0,\a}_{\b+2}(\ol{\o})$. If $w\in C^2(\o_{2M})\cap C_\text{loc}(\ol{\o_{2M}}\cap B_{2M})$ solves
\[
\left\{
\begin{aligned}
  -\D w = f & \quad \mbox{in } \o_{2M}, \\
    w = 0   & \quad \mbox{on }  \de\omega_{2M}\cap B_{2M},
  \end{aligned}\right.
\]
then
\begin{equation}\label{lem:Dir-apriori:eq:est}
\|w\|_{\bar{C}^{2,\a}_\b(\o_M)} \leq C(\|f\|_{\bar{C}^{0,\a}_{\b+2}(\o_{2M})} + \|w\|_{\bar{C}_\b(\o_{2M})}).
\end{equation}
In particular, if $w\in \bar{C}_{\b}(\ol{\o})\cap C^2(\o)$ solves the global problem $(M=\infty)$, then $w\in \bar{C}^{2,\a}_\b(\ol{\o})$.
\end{lemma}
\begin{proof}
Take $x\in \ol{\o_M}$, set $R:= \la x\ra$ and rescale the problem by $R$. Setting
$x^R:=x/R$, $w_R(y):=w(Ry)$  and $f_R(y):=f(Ry)$, the rescaled problem reads
\[
\left\{
\begin{aligned}
  -\D w_R = R^2 f_R & \quad \mbox{in } \o_{2M}^R:=\o_{2M}/R, \\
    w_R = 0   & \quad \mbox{on }  \de\o_{2M}^R\cap B_{2M/R},
  \end{aligned}\right.
\]
Note that as $|x|\leq M$ and $R=\la x\ra \leq M+1\leq 2M$, we have the inclusions
\[
\ol{\o}_{R/4}(x)\subset \ol{\o}_{R/2}(x) = [\ol{\o}_{2M}]_{R/2}(x) \quad \Longrightarrow \quad [\ol{\o}^R]_{1/4}(x^R)\subset [\ol{\o}^R]_{1/2}(x^R) = [(\ol{\o}_{2M})^R]_{1/2}(x^R)
\]
In order to prove \eqref{lem:Dir-apriori:eq:est}, we have to show that 
$$R^\b \|w_R\|_{C^{2,\a}([\ol{\o}^R]_{1/4}(x^R))} \leq c(\|f\|_{\bar{C}^{0,\a}_{\b+2}(\o_{2M})} + \|w\|_{\bar{C}_\b(\o_{2M})}),$$
for a constant $c>0$ independent of $x$ and $R$. Let $\d>0$ be the constant from Proposition \ref{prop:asymp}. Then by a standard covering argument it is enough to prove that 
\begin{equation}\label{lem:Dir-apriori:eq1}
R^\b \|w_R\|_{C^{2,\a}(S)} \leq C(\|w\|_{\bar{C}_\b(\o_{2M})} +\|f\|_{\bar{C}^{0,\a}_{\b+2}(\o_{2M})})
\end{equation}
in the following two cases:
\begin{itemize}
\item when $S = B_{\d/4}(y)$  for some $y\in [\o^R]_{1/4}(x^R)$, such that $B_{\d/2}(y)\subset \o_{2M}^R$;
\item when $S=[\ol{\o}^R]_{\d}(y)$  for some free boundary point $y\in [\ol{\o}^R]_{1/4}(x^R) \cap \de\o^R$. Note that $[\ol{\o}^R]_{2\d}(y)\subset [\ol{\o}^R]_{1/2}(x^R) = [(\ol{\o}_{2M})^R]_{1/2}(x^R)$, as $\d\in (0,1/8)$ and $|y - x^R|<1/4$.
\end{itemize}
In the first case, interior Schauder estimates (\cite[Theorem 6.2]{GT}) give the desired bound \eqref{lem:Dir-apriori:eq1}:
\begin{align*}
R^\b \|w_R\|_{C^{2,\a}(B_{\d/4}(y))} &\leq c R^\b ( \|w_R\|_{L^\infty(B_{\d/2}(y))} + \|R^2 f_R\|_{C^{0,\a}(B_{\d/2}(y))}) \\
& =c  ( R^\b \|w_R\|_{L^\infty(B_{\d/2}(y))} + R^{2+\b}\| f_R\|_{C^{0,\a}(B_{\d/2}(y))})  \\
& \leq C(\|w\|_{\bar{C}_\b(\o_{2M})}+ \|f\|_{\bar{C}^{0,\a}_{\b+2}(\o_{2M})}).
\end{align*}
In the second case, \eqref{lem:Dir-apriori:eq1} follows in the same fashion from the local boundary Schauder estimates (\cite[Lemma 6.5]{GT}) for the homogeneous Dirichlet problem in $[\o^R]_{2\d}(y)$, taking into account that the $C^{2,\a}$-norm of the free boundary portion of the domain $[\o^R]_{2\d}(y)$, where $w_R=0$, is bounded uniformly in $R$, according to Proposition \ref{prop:asymp}.
\end{proof}

\begin{lemma}\label{lem:Neum-apriori} Let $\a\in (0,1)$, $\b\geq 0$ and let $f\in \bar{C}^{0,\a}_{\b+2}(\ol{\o})$, $g\in \bar{C}^{1,\a}_{\b+1}(\de\o)$. If $w\in C^2(\o_{2M})\cap C^1_\text{loc}(\ol{\o_{2M}}\cap B_{2M})$ solves
\[
\left\{
\begin{aligned}
  -\D w = f & \quad \mbox{in } \o_{2M}, \\
	w_\nu = g   & \quad \mbox{on }  \de\omega_{2M}\cap B_{2M}, \\
  \end{aligned}\right.
\]
then 
\begin{equation}\label{lem:Neum-apriori:eq:est}
\|w\|_{\bar{C}^{2,\a}_\b(\o_M)} \leq C(\|f\|_{\bar{C}^{0,\a}_{\b+2}(\o_{2M})} + \|g\|_{\bar{C}^{1,\a}_{\b+1}(\de\o_{2M}\cap B_{2M})}+\|w\|_{\bar{C}_\b(\o_{2M})}).
\end{equation}
In particular, if $w\in  C^2(\o)\cap C^1_{\text{loc}}(\ol{\o})$ solves the global problem $(M=\infty)$ and $w\in \bar{C}_\b(\ol{\o})$, then $w\in \bar{C}^{2,\a}_\b(\ol{\o})$. 
\end{lemma}
\begin{proof}
Picking any $x\in \o$, setting $R:=\la x\ra$ and rescaling the problem by $R$, we obtain
\[
\left\{
\begin{aligned}
  -\D w_R = R^2 f_R & \quad \mbox{in } \o^R_{2M}, \\
    \de_\nu w_R = R g_R   & \quad \mbox{on }  \de\omega_{2M}^R\cap B_{2M/R}.
  \end{aligned}\right.
\]
By the  covering argument from the proof of Lemma \ref{lem:Dir-apriori}, it suffices to consider the same interior and the boundary scenarios for $S$. The interior one is treated exactly as in Lemma \ref{lem:Dir-apriori}. We sketch the second case when  $S=[\ol{\o}^R]_{\d}(y)\subset [\ol{\o}^R]_{2\d}(y)$, for some free boundary point $y\in [\ol{\o}^R]_{1/4}(x_R) \cap \de\o^R$. Since the $C^{2,\a}$-norm of the free boundary portion $[\de \o^R]_{2\d}(y)$ is bounded uniformly in $R$ by Proposition \ref{prop:asymp}, the local boundary Schauder estimate (\cite[Section 6.7]{GT}) for the Neumann problem yields the desired estimate
\begin{align*}
R^\b \|w_R\|_{C^{2,\a}([\ol{\o}^R]_{\d}(y))} &\leq c R^\b ( \|w_R\|_{L^\infty([\ol{\o}^R]_{2\d}(y))} + \|R g_R\|_{C^{1,\a}([\de \o^R]_{2\d}(y))} + \|R^2 f_R\|_{C^{0,\a}([\ol{\o}^R]_{2\d}(y))}) \\
& =c  ( R^\b \|w_R\|_{L^\infty([\ol{\o}^R]_{2\d}(y))} + R^{1+\b} \| g_R\|_{C^{1,\a}([\de \o^R]_{2\d}(y))} + R^{2+\b}  \| f_R\|_{C^{0,\a}([\ol{\o}^R]_{2\d}(y))}) \\
& \leq C(\|w\|_{\bar{C}_\b(\o_{2M})}+ \|g\|_{\bar{C}^{1,\a}_{\b+1}(\de\o\cap B_{2M})} + \|f\|_{\bar{C}^{0,\a}_{\b+2}(\o_{2M})}),
\end{align*}
with a constant $C>0$ independent of $x$ and $R$. 
\end{proof}

An analogous result holds for the Neumann problem on the product domain $\O:=\o\times \R$ with data in the corresponding function spaces over $\O$. 

\begin{lemma}\label{lem:Neum-apriori-prod} Let $\a\in (0,1)$, $\b\geq 0$ and let $f\in C^{0,\a}_{\b+2}(\ol{\O})$, $g\in C^{1,\a}_{\b+1}(\G)$. If $w\in C^2(\O)\cap C^1_\text{loc}(\ol{\O})$ solves
\[
\left\{
\begin{aligned}
  -\D w = f & \quad \mbox{in } \O, \\
	w_\nu = g   & \quad \mbox{on }  \G, \\
  \end{aligned}\right.
\]
and $w\in C_\b(\ol{\O})$, then $w\in C^{2,\a}_\b(\ol{\O})$. 
\end{lemma}
\begin{proof}
The argument follows verbatim the proof above, given that for any $X= (x,t)\in \O$, $R:=\la x \ra$, the rescaled domain $\O^R =\O/R$ satisfies locally uniform $C^{2,\a}$ bounds.  
\end{proof}

\section{Sobolev inequalities}\label{app:Sobolev}
For the reader's convenience, we record the following Sobolev trace inequality from \cite{BK24}.

\begin{prop}[Proposition A.1 from \cite{BK24}]\label{prop:trace} Let $u$ be a classical solution of the one-phase FBP in a domain $D\subset \R^n$ and let $\o$ denote the positive phase $\{u>0\}$ or a connected component thereof. Assume that 
\begin{equation*}
|\n u| \leq L  \quad \text{in } \o.
\end{equation*}
Then 
\begin{equation}\label{prop:trace:L1}
\int_{\de \o\cap D} \psi \, d\s \leq L \|\n \psi \|_{L^1(\o)} \quad \text{for all } \psi \in C^1_c(\ol{\o}\cap D).
\end{equation}
\end{prop}

As a consequence, we get the following trace inequality for functions in $H^1_{\text{per}}(\O)$  (see \eqref{eq:H1per} for the definition) on the product domain $\O=\o\times \R$.
\begin{prop}\label{prop:tracePer}
For every function $\f\in H^1_\text{per}(\O)$, it holds
\begin{equation}\label{prop:tracePer:ineq}
\|\f\|^2_{L^2_\text{per}(\G)} \leq 2 \|\f\|_{L^2_\text{per}(\O)} \|D\f\|_{L^2_\text{per}(\O)},
\end{equation}
where the values of $\f$ on $\G$ are understood in the trace sense. 
\end{prop}
\begin{proof}
By approximation, it suffices to show \eqref{prop:tracePer:ineq} for functions $\f=\f(x,t)\in C^1(\ol{\o}\times \R)$ that are $T$-periodic in $t$ and compactly supported in $x\in\ol{\o}$. Since for each fixed $t$, the slice $\f_t(x):=\f(x,t)\in C^1_c(\ol{\o})$ we can apply the trace inequality \eqref{prop:trace:L1} to $\psi=\f_t^2$ in each cross-section (noting that $L=1$ on account of Proposition \ref{prop:grad}) to obtain
\begin{align*}
\|\f\|^2_{L^2_\text{per}(\G)} &= \int_{-T/2}^{T/2} \int_{\de\o} \f_t^2(s)\,dS \, dt \leq  \int_{-T/2}^{T/2} 2 \int_\o |\f(x,t)| |D_x \f(x,t)|\, dx\, dt \\
& \leq 2 \int_{\O_{\text{per}}} |\f| |D\f| \, dX \leq 2 \|\f\|_{L^2_\text{per}(\O)} \|D\f\|_{L^2_\text{per}(\O)},
\end{align*}
where the last inequality follows from Cauchy--Schwarz.
\end{proof}

\section{Technical lemmas for Section~\ref{sec:setup}}\label{app:technical}

Throughout this appendix section we are in Case A or Case B of Section~\ref{sec:setup}, with the notation
introduced there. We use Euclidean coordinates $X=(X_1,\dots,X_{n+1})$ on $\R^{n+1}$, with
$X_{n+1}=t$, write $\p_i=\p_{X_i}$, and use the Einstein summation convention over repeated
indices ranging over $\{1,\dots,n+1\}$. Sums over the horizontal indices $\{1,\dots,n\}$ only are written
explicitly. For a square matrix $M$ we write $|M|:=\det M$.

\subsection{The normal deformation field}

\begin{lemma}\label{lem:hphi}
The map $\f\mapsto h_\f$ is bounded and linear from $C^{2,\a}_0(\G)$ into
$C^{2,\a}_0(\O;\R^n\times\{0\})$. In particular,
$\|h_\f\|_{C^{2,\a}_0(\ol{\O})}\le C_1\|\f\|_{\Y}$ for $\f\in\Y$, in both cases.
\end{lemma}

\begin{proof}
Fix $X\in\O\cap \N_1$, write $X=p-z\nu(p)$, set $R:=\la x\ra$ and rescale by $R$.
By Proposition \ref{prop:asymp} the rescaled free boundary $\G^R$ is, near $X^R$, a union of
graphs with $C^{3,\a}$-norms bounded independently of $R$. So the rescaled Fermi chart
$(p,z)\mapsto p-z\nu(p)$ and the rescaled normal $\nu_R$ have $C^{2,\a}$-norms bounded
independently of $R$ on $[(\O\cap \N_1)^R]_{1/4}(X^R)$. By \eqref{eq:theta-small}, $\r(p)/R\in[1/2,2]$. So the rescaled cut-off
$\chi(Rz'/(\th\r(Rp'))),$ where $(p',z')$ are the rescaled Fermi coordinates, is bounded in $C^{2,\a}$ independently of $R$. Finally, $\|\f_R\|_{C^{2,\a}}\le\|\f\|_{C^{2,\a}_0(\G)}$ on the relevant components of $\G^R$. Since $(h_\f)_R$ is the product of these three quantities composed with the rescaled Fermi chart, we obtain
$$\|(h_\f)_R\|_{C^{2,\a}([\O^R]_{1/4}(X^R))}\le C_1\|\f\|_{C^{2,\a}_0(\G)}.$$
A finite covering takes care of the part of $[\O^R]_{1/4}(X^R)$ outside $\N_1^R$, where
$(h_\f)_R=0$. Taking the supremum over $X$ proves the lemma.
\end{proof}

\subsection{The pulled-back metric}

Let $h=h^k\p_k$ be a horizontal vector field on $\ol{\O}$, i.e.\ $h^{n+1}\equiv0$, and let
$\Phi_h:=\id+h$. Its Jacobian matrix $M_h:=D\Phi_h$ has entries
\[
  (M_h)^k_i=\d^k_i+\p_ih^k,\qquad\text{so that}\qquad (M_h)^{n+1}_i=\d^{n+1}_i .
\]
The components of the metric $g_T=dx^2+T^2dt^2$ and of its inverse are
\[
  (g_T)_{kl}=\d_{kl}+(T^2-1)\,\d_{k,n+1}\d_{l,n+1},\qquad
  g_T^{kl}=\d^{kl}+(T^{-2}-1)\,\d^{k,n+1}\d^{l,n+1}.
\]

\begin{lemma}\label{lem:metric}
There is $\e_0>0$ with the following property. Let $h$ be a horizontal vector field with
$\|h\|_{C^{2,\a}_0(\ol{\O})}<\e_0$. Then $|M_h|\ge1/2$ on $\ol{\O}$, and for every $T>0$ the components of
$g_{T,h}:=\Phi_h^*g_T$ are
\begin{equation}\label{eq:metric-components}
  (g_{T,h})_{ij}=(g_T)_{ij}+(a_h)_{ij},\qquad
  (a_h)_{ij}:=\p_ih^j+\p_jh^i+\sum_{k=1}^n\p_ih^k\,\p_jh^k ,
\end{equation}
where $a_h$ does not depend on $T$. The components of the inverse metric and the volume density are
\begin{equation}\label{eq:metric-formulas}
  g_{T,h}^{ij}=p_h^{ij}+T^{-2}q_h^{ij},\qquad \sqrt{|g_{T,h}|}=T\,|M_h|,
\end{equation}
where
\begin{equation}\label{eq:p-q}
  p_h^{ij}:=\sum_{k=1}^n(M_h^{-1})^i_k\,(M_h^{-1})^j_k,\qquad
  q_h^{ij}:=\xi_h^i\,\xi_h^j,\qquad \xi_h^i:=(M_h^{-1})^i_{n+1}.
\end{equation}
In particular $p_h^{i,n+1}=0$ and $\xi_h^{n+1}=1$. At $h=0$ we have $p_0^{ij}=\sum_{k=1}^n\d^i_k\d^j_k$
and $q_0^{ij}=\d^i_{n+1}\d^j_{n+1}$, so that $g_T^{ij}=p_0^{ij}+T^{-2}q_0^{ij}$. Moreover, the maps
\[
  h\mapsto a_h,\quad h\mapsto M_h^{-1}-I,\quad h\mapsto p_h-p_0,\quad h\mapsto q_h-q_0,\quad h\mapsto \log|M_h|
\]
are smooth (indeed real analytic) from $\{h:\|h\|_{C^{2,\a}_0(\ol{\O})}<\e_0\}$ into $C^{1,\a}_1(\ol{\O})$
and vanish at $h=0$.
\end{lemma}

\begin{proof}
The components of the pull-back metric are
\[
  (g_{T,h})_{ij}=g_T\big(D\Phi_h(\p_i),D\Phi_h(\p_j)\big)=(g_T)_{kl}\,(M_h)^k_i\,(M_h)^l_j .
\]
The term with $k=l=n+1$ equals $T^2\d^{n+1}_i\d^{n+1}_j$, because $h^{n+1}=0$. The terms with $k,l\le n$ give
\[
  \sum_{k=1}^n(\d^k_i+\p_ih^k)(\d^k_j+\p_jh^k)
  =\sum_{k=1}^n\d^k_i\d^k_j+\p_ih^j+\p_jh^i+\sum_{k=1}^n\p_ih^k\,\p_jh^k .
\]
Mixed terms vanish, since $(g_T)_{kl}=0$ for $k\ne l$. This proves \eqref{eq:metric-components}. Inverting
$g_{T,h}=M_h^T\,g_T\,M_h$, we get
\[
  g_{T,h}^{ij}=(M_h^{-1})^i_k\,(M_h^{-1})^j_l\,g_T^{kl}
  =\sum_{k=1}^n(M_h^{-1})^i_k(M_h^{-1})^j_k+T^{-2}(M_h^{-1})^i_{n+1}(M_h^{-1})^j_{n+1},
\]
which is \eqref{eq:metric-formulas}; also $|g_{T,h}|=|M_h|^2|g_T|=T^2|M_h|^2$. Since the last row of $M_h$ is
$e_{n+1}$, the same holds for $M_h^{-1}$, i.e.\ $(M_h^{-1})^{n+1}_k=\d^{n+1}_k$. This gives $p_h^{i,n+1}=0$ and
$\xi_h^{n+1}=1$.

For the regularity, the entries $\p_ih^k$ of $Dh=M_h-I$ belong to $C^{1,\a}_1(\ol{\O})$ with norms at most
$\|h\|_{C^{2,\a}_0(\ol{\O})}$. By \eqref{eq:product}, $C^{1,\a}_0(\O)$ is a Banach algebra in which
$C^{1,\a}_1(\ol{\O})$ is an ideal. So $a_h\in C^{1,\a}_1(\ol{\O})$ depends polynomially on $h$. For
$\|h\|_{C^{2,\a}_0}<1/2$ the Neumann series
\[
  M_h^{-1}=\sum_{m\ge0}(-Dh)^m
\]
converges, and $M_h^{-1}-I\in C^{1,\a}_1(\ol{\O})$ depends analytically on $h$. The entries of $p_h-p_0$ and
$q_h-q_0$ are polynomials without constant term in the entries of $M_h^{-1}-I$. Also, $|M_h|-1$ is a polynomial
without constant term in the entries $\p_ih^k$. Choosing $\e_0$ so small that $||M_h|-1|\le1/2$, the function
$\log|M_h|=\log(1+(|M_h|-1))$ is given by a convergent power series.
\end{proof}

\subsection{Composition with displaced boundary points}

Let $\d\in(0,1/8)$ be the constant of Proposition~\ref{prop:asymp} and $\N_1$ the tubular neighbourhood \eqref{eq:N1vicinity} of $\G$. For $\g\in\R$ let $\mathcal S_\g$ denote the class of functions
$\Psi\in C^\infty(\N_1)$ (possibly vector- or matrix-valued) that do not depend on $t$ and have the following
\emph{local extension property}. For every $p\in\G$, with $R:=\la p_x\ra$, the rescaled function $\Psi^{(R)}(y):=R^\g\,\Psi(Ry)$
coincides, on the connected component of $\N_1^R\cap B_{2\d}(p^R)$ containing $p^R=p/R$, with a function
$\Psi_p\in C^\infty(B_{2\d}(p^R))$ satisfying
\[
  \|D^j\Psi_p\|_{L^\infty(B_{2\d}(p^R))}\le C_j\qquad\text{for all } j\in\mathbb{N}_0,
\]
with constants $C_j$ independent of $p$. Thus, near each $p\in\G$, $\Psi$ extends across the free boundary to a band of
width comparable to $\la p_x\ra$, with scale-invariant bounds.

By Proposition~\ref{prop:asymp}, we have that
\begin{equation}\label{eq:classes}
  \p_kU\in\mathcal S_0,\qquad \p^2_{kl}U\in\mathcal S_1,\qquad \p^2_{kl}E\in\mathcal S_2,\qquad\text{where } E:=|\n U|^2-1 .
\end{equation}
Moreover, $\Psi\in\mathcal S_\g$ implies $D^j\Psi\in\mathcal S_{\g+j}$.

\begin{lemma}\label{lem:composition}
Let $\g\ge0$, $\Psi\in\mathcal S_\g$. For $s\in[0,1]$ the map
\[
  \mathcal C^s_\Psi:\ h\mapsto\Psi\circ(\id+s h)=\Big(p\mapsto\Psi\big(p+s h(p)\big)\Big),\qquad p\in\G,
\]
is smooth from $\tilde{V}:=\{h\in C^{2,\a}_0(\G;\R^{n+1}):\|h\|_{C^{2,\a}_0(\G)}<\d_0/2\}$ into $C^{2,\a}_\g(\G)$. 
Consequently, for bounded measurable $m:[0,1]\to\R$, the map $h\mapsto\int_0^1m(s)\,\mathcal C^s_\Psi(h)\,ds$ is smooth from
$\tilde{V}$ into $C^{2,\a}_\g(\G)$.
\end{lemma}

\begin{proof}
Fix $p_0\in\G$, let $R:=\la(p_0)_x\ra$, and work on $D:=[\G^R]_{\d}(p_0^R)$; by a covering argument it suffices to bound
the $C^{2,\a}(D)$-norm. For $h\in \tilde{V}$ we have
\[
  R^\g\big(\mathcal C^s_\Psi(h)\big)_R=\Psi^{(R)}\circ G,\qquad G(y):=y+R^{-1}s\,h(Ry),\quad y\in D,
\]
and $\|G-\id\|_{C^{2,\a}(D)}\le R^{-1}\|h\|_{C^{2,\a}_0(\G)}<\d_0/2$. In particular, the segment from $y\in D$ to $G(y)$ has length
less than $\d_0/2R\leq \d/2$. So it lies in $\N_1^R$ and in $B_{3\d/2}(p_0^R)$. Hence $G(D)$ is
contained in the connected component of $\N_1^R\cap B_{2\d}(p_0^R)$ containing $p_0^R$, where $\Psi^{(R)}=\Psi_{p_0}$. Thus
\[
  R^\g\big(\mathcal C^s_\Psi(h)\big)_R=\Psi_{p_0}\circ G\qquad\text{on }D,
\]
with $\Psi_{p_0}$ smooth on the convex set $B_{2\d}(p_0^R)\supset G(D)$, with $C^j$-bounds independent of $p_0$.

We now use the classical fact that, for $\psi\in C^{k+3}(B)$ with $B$ a ball, the map $f\mapsto\psi\circ f$ is of class $C^k$ from
$\{f\in C^{2,\a}(D;\R^{n+1}):\ f(D)\subset B\}$ into $C^{2,\a}(D)$. The bounds on its derivatives depend only on $\|\psi\|_{C^{k+3}(B)}$,
$\|f\|_{C^{2,\a}(D)}$ and the $C^{2,\a}$-geometry of $D$, which is controlled by Proposition~\ref{prop:asymp}. Applied with $\psi=\Psi_{p_0}$ and $f=G$, this gives uniform bounds for all derivatives of $h\mapsto R^\g(\mathcal C^s_\Psi(h))_R$ from $\tilde{V}$ into $C^{2,\a}(D)$. Taking the supremum over $p_0$ proves the claims. The statement about the integral follows by differentiating under the integral sign.
\end{proof}

\section*{Acknowledgements and declarations}

\subsection*{Acknowledgements}
The authors were partially supported by Proyecto ANID Fondecyt Regular No.\ 1241588.

\subsection*{Declaration of competing interest}
The authors report that there are no competing interests to declare.

\subsection*{Declaration of generative AI and AI-assisted technologies in the manuscript preparation process} All mathematical content, results, proofs, and exposition in this paper are the work of the authors, unless otherwise stated. 

During the preparation of this manuscript, the authors used Claude Opus 5.5 to streamline the calculations in Section \ref{sec:setup} (plus the accompanying Appendix \ref{app:technical}) and review the final draft for errors and inconsistencies. The authors subsequently revised the manuscript and take full responsibility for the content of the published article.

\bibliography{fb_bif_bib}

\end{document}